\documentclass[11pt]{article}
\usepackage{amsmath, amsfonts, amssymb, amsthm,  graphicx, mathtools, enumerate, dsfont}

\allowdisplaybreaks

\usepackage[blocks, affil-it]{authblk}

\usepackage[square]{natbib}
\usepackage[CJKbookmarks=true,
            bookmarksnumbered=true,
			bookmarksopen=true,
			colorlinks=true,
			citecolor=red,
			linkcolor=blue,
			anchorcolor=red,
			urlcolor=blue]{hyperref}
\usepackage[usenames]{color}

\usepackage[letterpaper, left=1.2truein, right=1.2truein, top = 1.2truein, bottom = 1.2truein]{geometry}
\usepackage[ruled, vlined, lined, commentsnumbered]{algorithm2e}

\usepackage{prettyref,soul}

\usepackage{hyperref}
    \hypersetup{ colorlinks = true, linkcolor = blue }

\newtheorem{lemma}{Lemma}[section]
\newtheorem{proposition}{Proposition}[section]
\newtheorem{thm}{Theorem}[section]

\newtheorem{remark}{Remark}[section]

\newrefformat{eq}{(\ref{#1})}
\newrefformat{chap}{Chapter~\ref{#1}}
\newrefformat{sec}{Section~\ref{#1}}
\newrefformat{algo}{Algorithm~\ref{#1}}
\newrefformat{fig}{Fig.~\ref{#1}}
\newrefformat{tab}{Table~\ref{#1}}
\newrefformat{rmk}{Remark~\ref{#1}}
\newrefformat{clm}{Claim~\ref{#1}}
\newrefformat{def}{Definition~\ref{#1}}
\newrefformat{cor}{Corollary~\ref{#1}}
\newrefformat{lmm}{Lemma~\ref{#1}}
\newrefformat{lemma}{Lemma~\ref{#1}}
\newrefformat{prop}{Proposition~\ref{#1}}
\newrefformat{app}{Appendix~\ref{#1}}
\newrefformat{ex}{Example~\ref{#1}}
\newrefformat{exer}{Exercise~\ref{#1}}
\newrefformat{soln}{Solution~\ref{#1}}
\newrefformat{cond}{Condition~\ref{#1}}

\def\P{{\mathbb P}}
\def\ind{\mathds{1}}

\def\text#1{\mbox{\rm #1}}

\DeclarePairedDelimiter{\ceil}{\lceil}{\rceil}
\newcommand{\floor}[1]{{\left\lfloor {#1}\right\rfloor}}

\newcommand{\argmin}{\mathop{\rm argmin}}

\newcommand{\norm}[1]{\|{#1} \|}

\newcommand{\wh}{\widehat}
\newcommand{\wt}{\widetilde}
\newcommand{\wb}{\overline}

\newcommand{\iprod}[2]{\Bigl \langle #1, #2 \Bigr\rangle}

\newcommand{\T}{{\sf T}}
\newcommand{\cF}{\mathcal{F}}
\newcommand{\reals}{\mathbb{R}}
\newcommand{\E}{\mathbb{E}}
\newcommand{\Z}{\mathbb{Z}}

\let\hat\widehat
\let\bar\overline

\newcommand\chzM[1]{{#1}}

\usepackage{xr}

\begin{document}

\setlength{\abovedisplayskip}{5pt}
\setlength{\belowdisplayskip}{5pt}
\setlength{\abovedisplayshortskip}{5pt}
\setlength{\belowdisplayshortskip}{5pt}

\title{\LARGE On Estimation of Isotonic Piecewise Constant Signals}

\author{Chao Gao\thanks{Department of Statistics, University of Chicago, Chicago, IL 60637, USA; email: {\tt chaogao@galton.uchicago.edu}},~~Fang Han\thanks{Department of Statistics, University of Washington, Seattle, WA 98195, USA; e-mail: {\tt fanghan@uw.edu}}, ~and~Cun-Hui Zhang\thanks{Department of Statistics and Biostatistics, Rutgers University, Piscataway, NJ 08854, USA; e-mail: {\tt cunhui@stat.rutgers.edu}} }

\date{}

\maketitle

\begin{abstract}
Consider a sequence of real data points $X_1,\ldots, X_n$ with underlying means $\theta^*_1,\dots,\theta^*_n$.
This paper starts from studying the setting that $\theta^*_i$ is both piecewise constant and monotone 
as a function of the index $i$. For this, we establish the exact minimax rate of estimating such monotone functions, and thus give a non-trivial answer to an open problem in the shape-constrained analysis literature. The minimax rate under the {loss of the sum of squared errors} involves an interesting iterated logarithmic dependence on the dimension, a phenomenon that is revealed through characterizing the interplay between the isotonic shape constraint and model selection complexity. We then develop a penalized least-squares procedure for estimating the vector $\theta^*=(\theta^*_1,\dots,\theta^*_n)^{\T}$. 
This estimator is shown to achieve the derived minimax rate adaptively. 
For the proposed estimator, we further allow the model to be misspecified and derive oracle inequalities with the optimal rates, and show there exists a computationally efficient algorithm to compute the exact solution. 
\end{abstract}


\section{Introduction}\label{sec:intro}

Consider an observed vector $X=(X_1,...,X_n)^{\T}$ of independent entries and an unknown underlying mean $\theta^*=(\theta_1^*,...,\theta_n^*)^{\T}$. This paper starts from the problem of estimating such $\theta^*$ that is isotonic piecewise constant.
Specifically, for any $ k\in (0:n]$, we define the parameter space of interest 
\chzM{as the set of all nondecreasing vectors with at most $k$ pieces,}  
\begin{align*}
\Theta_k^{\uparrow} =& \Big\{\theta\in\mathbb{R}^n: \text{there exist }\{a_j\}_{j=0}^k\text{ and }\{\mu_j\}_{j=1}^k\text{ such that }\\
&~~0=a_0\leq a_1\leq\cdots\leq a_k=n,\\ 
&~~\mu_1\leq\mu_2\leq\cdots\leq\mu_k, \text{ and }\theta_i=\mu_j\text{ for all }i\in(a_{j-1}:a_j]\Big\}.
\end{align*}
The notation $(a:b]$ stands for the set of all integers $i$ that satisfy $a<i\leq b$.
For any vector $\theta^*\in\Theta_k^{\uparrow}$, it is a piecewise constant signal with at most $k$ steps that take different values. When $k=n$, the space $\Theta_k^{\uparrow}$ contains all vectors $\theta^*$ that satisfy $\theta^*_1\leq\theta^*_2\leq\cdots\leq\theta^*_n$. Estimation of $\theta^*$ under this condition is recognized as isotonic regression. It has been one of the most popular and successful directions in the shape-constrained analysis literature. General discussions on relevant methods and theory 
can be found 
in \cite{robertson1988order}, \cite{groeneboom1992information}, \cite{silvapulle2011constrained}, and \cite{groeneboom2014nonparametric}, to name just a few. However, in certain cases, isotonic regression may overfit the data by producing a result with too many steps. This inspires research on fitting isotonic regression with the restriction of the number of steps. According to \cite{schell1997reduced}, the problem is termed as reduced isotonic regression. The parameter space $\Theta_k^{\uparrow}$ precisely describes such regression functions.

Despite its practical importance in change-point and shape-constrained analyses, the fundamental limit of estimating $\theta^*$ in the class $\Theta_k^{\uparrow}$ is still unknown. We summarize the results in the literature by assuming that $X\sim N(\theta^*,\sigma^2I_n)$. In terms of upper bound, \cite{chatterjee2015risk} show explicitly that
$$\inf_{\wh{\theta}}\sup_{\theta^*\in\Theta_k^{\uparrow}}\mathbb{E}\|\wh{\theta}-\theta^*\|^2\leq C\sigma^2 k\log(en/k),$$
and the rate $\sigma^2 k\log(en/k)$ can be adaptively achieved by isotonic regression. See \cite{bellec2015sharp} and \cite{bellec2015sharp2} for results with the same rate. In terms of lower bound, \cite{bellec2015sharp2} show
$$\inf_{\wh{\theta}}\sup_{\theta^*\in\Theta_k^{\uparrow}}\mathbb{E}\|\wh{\theta}-\theta^*\|^2\geq c\sigma^2k.$$
We can see the above upper and lower bounds do not match, and it is unclear if either bound is sharp. 

In this paper, we settle a solution to this open problem by deriving the precise minimax rate of the space $\Theta_k^{\uparrow}$. Thus, the gap between the upper and lower bounds in the literature is closed. Surprisingly, neither the upper nor the lower bound in the literature is sharp. We prove that for $k\geq 2$, the minimax rate takes the form
$$\inf_{\wh{\theta}}\sup_{\theta^*\in\Theta_k^{\uparrow}}\mathbb{E}\|\wh{\theta}-\theta^*\|^2\asymp \sigma^2 k\log\log(16n/k).$$
It is interesting that the minimax rate of the problem has an iterated logarithmic dependence on $n/k$, an engaging 
feature of the space $\Theta_k^{\uparrow}$. 

We show that the minimax rate can be achieved by solving a least-squares problem in the space $\Theta_k^{\uparrow}$. This is exactly the procedure of reduced isotonic regression. In comparison, the ordinary isotonic regression proves to achieve only a sub-optimal rate $\sigma^2k\log(en/k)$. Therefore, our results provide a theoretical justification that the reduced isotonic regression can avoid overfitting the data and practically attain better performances over 
the ordinary isotonic regression (cf. \cite{schell1997reduced}, \cite{salanti2003nonparametric}, \cite{haiminen2008algorithms}). 

The proof of the result is non-trivial.  Our analysis involves repeatedly partitioning the studied sequence according to the nature of the reduced isotonic regression estimator. This allows us to use martingale maximal inequalities by Levy and Doob, and gives us the sharp minimax rate.

Besides understanding the fundamental challenge in estimating the piecewise monotone functions, in practice, it is always the case that: (i) the number of steps or pieces $k$ is unknown; (ii) the model could be misspecified. In addition, practically we would love to have computationally feasible algorithm to compute the exact solution. Indeed, in this manuscript we propose a penalized least-squares (reduced isotonic regression) estimator that achieves the minimax rate without knowing $k$. We further allow the model to be misspecified and prove oracle inequalities with the optimal rates.  Moreover, by exploring a key property of reduced isotonic regression and by leveraging the pool-adjacent-violators algorithm (PAVA) 
\citep{mair2009isotone}, 
we develop a computationally efficient algorithm to compute 
the $k$-piece least-squares estimator for all $k$ and thus
the penalized least-squares estimator.

This paper also obtains exact minimax rates under the $\ell_p$ loss with $1\leq p<2$. In contrast to the case $p = 2$, the minimax rates are now parametric. Furthermore, we show that this rate can be adaptively achieved by isotonic regression, but not by the reduced isotonic regression procedure. In other words, the nature of the problem can be {dramatically} changed by using a different loss function.

The rest of the paper is organized as follows. In Section \ref{sec:upper}, we introduce the problem setting and present the minimax rate. We then introduce an adaptive estimation procedure in Section \ref{sec:adapt}. The computational issues of the estimators are discussed in Section \ref{sec:disc}. We will also put our results in a larger picture and discuss a few other related problems in Section \ref{sec:disc}. All the proofs are relegated to Section \ref{sec:proof} and the supplement.

\paragraph{Notation}
Let $\Z$ and $\reals$ be the sets of integers and real numbers. For any positive integer $d$, we use $[d]$ to denote the set $\{1,2,...,d\}$. Let $\ind(\cdot)$ denote the indicator function. For a real number $x$, $\ceil{x}$ is the smallest integer no smaller than $x$, $\floor{x}$ is the largest integer no larger than $x$, $x_+=x\ind(x\geq 0)$ and $x_-=-x\ind(x<0)$ are the positive and negative components of $x$. For any $a,b\in\reals$, write $a\wedge b=\min\{a,b\}$ and $a\vee b=\max\{a,b\}$. For an arbitrary vector $\theta=(\theta_1,\ldots,\theta_n)^\T\in\reals^n$ and an index set $J\subset [n]$, we denote $\theta_J$ to be the sub-vector of $\theta$ with entries indexed by $J$, and for any $p\geq 1$,
\[
\norm{\theta}_p=\Big(\sum_{i=1}^n|\theta_i|^p \Big)^{1/p},~~{\rm and}~~\norm{\theta}_{J,p}=\Big(\sum_{i\in J}\theta_i^p \Big)^{1/p}.
\]
In particular, we denote $\norm{\theta}=\norm{\theta}_2$ and $\norm{\theta}_J=\norm{\theta}_{J,2}$. Let $\overline{\theta}_{J}=\frac{1}{|J|}\sum_{i\in J}\theta_i$ represent the sample mean across the sequence $\theta_{J}$. For any real value $a$ and positive integer $n$, define 
\[
\{a\}^n=(\underbrace{a,a,\ldots,a}_{n})^\T. 
\]
For any sets of vectors $\Theta_1\subset \reals^{n_1},\ldots,\Theta_m\subset \reals^{n_m}$, denote
\[
\bigtimes_{\ell=1}^m\Theta_{\ell} = \Big\{\theta=(\theta_{(1)}^\T,\ldots,\theta_{(m)}^\T)^\T\in\reals^{\sum_{i=1}^m n_i}:\theta_{(\ell)}\in\Theta_\ell\Big\}.
\]
Throughout the paper, let $c, C, c_1, C_1,  c_2, C_2,\ldots$ be generic universal positive constants whose actual values may vary at different places. For any two positive data sequences $\{a_n, n=1,2,\ldots\}$ and $\{b_n,n=1,2,\ldots\}$, we write $a_n\lesssim b_n$ or $a_n=O(b_n)$ if there exists a constant $C>0$ such that $a_n\leq C b_n$ for all $n$ from natural numbers. The notation $a_n\asymp b_n$ means $a_n\lesssim b_n$ and $b_n\lesssim a_n$. We use $\mathbb{P}$ and $\mathbb{E}$ to denote generic probability and expectation operations whenever the distributions can be determined from the context.

\section{Minimax rates}\label{sec:upper}

In this section, we present the minimax rate of the space $\Theta_{k}^{\uparrow}$ with respect to the squared $\ell_2$ loss. We first consider the upper bound. Given the observation $X\in \mathbb{R}^n$, we define the constrained least-squares estimator as
\begin{equation}
\wh{\theta}(\Theta_{k}^{\uparrow})=\argmin_{\theta\in \Theta_{k}^{\uparrow}}\|X-\theta\|^2.\label{eq:ls}
\end{equation}
Computational issues related to this estimator will be discussed in Section \ref{sec:compute}. Note that if $X\sim N(\theta^*,\sigma^2I_n)$, $\wh{\theta}(\Theta_{k}^{\uparrow})$ is simply the maximum likelihood estimator (MLE) restricted onto the parameter space $\Theta_{k}^{\uparrow}$. However, we do not need to assume a Gaussian error for the risk bound presented below.
In detail, consider the observation 
\[
X=\theta^*+Z,
\]
where we assume the error variables $\{Z_i\}_{i=1}^n$ are independent with zero mean and satisfy one of the following conditions,
\begin{equation}
\begin{cases} \max_{1\leq i\leq n}\mathbb{E}\Bigl|Z_i/\sigma\Bigr|^{2+\epsilon}\leq C_1, 
& \hbox{not identically distributed } Z_i\hbox{'s},
\cr \mathbb{E}(Z_1^2/\sigma^2)\log(e+Z_1^2/\sigma^2) \le C_1, & \hbox{identically distributed } Z_i\hbox{'s},
\end{cases}
\label{eq:moment}
\end{equation}
for some number $\sigma>0$, 
an arbitrarily small universal constant $\epsilon\in(0,1)$, and some universal constant $C_1>0$. 
It is easy to see that the Gaussian error $Z\sim N(0,\sigma^2 I_n)$ is a special case.

\begin{thm} \label{thm:upper2}
Consider $X=\theta^*+Z$ with any $\theta^*\in\mathbb{R}^n$ and $Z$ satisfying (\ref{eq:moment}).
Then, we have
$$\mathbb{E}\|\wh{\theta}(\Theta_k^{\uparrow})-\theta^*\|^2\leq C\Big[\inf_{\theta\in\Theta_k^{\uparrow}}\|\theta-\theta^*\|^2 + \sigma^2 + \sigma^2k\log\log(16n/k)\ind\{k\geq 2\}\Big]$$
for all $k\in[n]$ with some universal constant $C>0$.
\end{thm}

Note that Theorem \ref{thm:upper2} is an oracle inequality without any assumption on the true mean vector $\theta^*$.
Besides the trivial bound $C\left(\inf_{\theta\in\Theta_1^{\uparrow}}\|\theta-\theta^*\|^2+\sigma^2\right)$ for $k=1$, it is interesting that the stochastic error scales as $\sigma^2k\log\log(16n/k)$ for $k\geq 2$. This iterated logarithmic term appears due to the isotonic constraint of the solution $\wh\theta(\Theta_k^\uparrow)$ as well as the properties of partial sum processes. More technical discussions on this point will be given in Section \ref{sec:compare-piece}, which discusses the importance of the isotonic constraint in more details. 

If the condition $\theta^*\in\Theta_k^\uparrow$ holds, then we immediately obtain the following corollary
$$\sup_{\theta^*\in\Theta_k^\uparrow}\mathbb{E}\|\wh{\theta}(\Theta_k^{\uparrow})-\theta^*\|^2\leq C\sigma^2k\log\log(16n/k),$$
when $k\geq 2$. This improves previous risk bounds for the space $\theta^*\in\Theta_{k}^{\uparrow}$ in the literature. 
For example, 
for the ordinary isotonic regression estimator
\begin{equation}
\wh{\theta}^{(iso)}=\wh{\theta}(\Theta_{n}^{\uparrow})=\argmin_{\theta:\theta_1\leq\theta_2\leq\cdots\leq\theta_n}\|X-\theta\|^2,\label{eq:iso-def}
\end{equation}
Theorem 2.1 of \cite{zhang2002risk} gives
$$
\sum_{i=n_1+1}^{n_2}\big|\wh{\theta}^{(iso)}_i-\theta^*_i\big|^2 \le \int_0^{n_2-n_1} 
\frac{C\sigma^2}{1\vee x}dx, 
$$
whenever $0\le n_1< n_2\le n$ and $\theta^*_{n_2} = \theta^*_{n_1+1}$ for a nondecreasing $\theta^*$.
Thus, as explicitly derived in \cite{chatterjee2015risk}, 
$$\sup_{\theta^*\in\Theta_{k}^{\uparrow}}\mathbb{E}\norm{\wh{\theta}^{(iso)}-\theta^*}^2\leq C\sigma^2 k\log(en/k).$$
Our result shows that the logarithmic error term in 
the above bound can be improved by restricting the least-squares optimization to the space $\theta^*\in\Theta_{k}^{\uparrow}$. 
This modification of the estimator is necessary, as shown below. 
\begin{proposition}\label{prop:lower-LSE}
There exists a universal constant $c>0$, such that
$$\sup_{\theta^*\in\Theta_{k}^{\uparrow}}\mathbb{E}\norm{\wh{\theta}^{(iso)}-\theta^*}^2\geq c\sigma^2k\log(en/k).$$
\end{proposition}

Next, we show that the rate obtained by Theorem \ref{thm:upper2} is optimal by giving a matching minimax lower bound. To this end, we consider the Gaussian distribution $X\sim N(\theta^*,\sigma^2I_n)$. In the following a lower bound construction for $k=2$ is provided, with the generalization to $k\geq 2$ briefly sketched.

By Fano's inequality (Proposition \ref{prop:fano}), we need to find some subset $T\subset\Theta_2^{\uparrow}$ such that the ratio $$\frac{\max_{\theta,\theta'\in T}\|\theta-\theta'\|^2/(2\sigma^2)}{\log\mathcal{M}(\epsilon,T,\|\cdot\|)}$$ is bounded by a sufficiently small constant. Here, $\mathcal{M}(\epsilon,T,\|\cdot\|)$ stands for the packing number of $T$ with radius $\epsilon$ and distance $\|\cdot\|$. We will take $\epsilon^2\asymp\log\log (16n)$. Since the minimax rate is simply $\sigma^2$ if $n$ is bounded by a constant, we only need to construct $T$ with a sufficiently large $n$. For each $\ell\in\{1,2,...,\ceil{\log_2n}\}$, construct the vector $\theta_\ell\in\mathbb{R}^n$ by filling the last $\ceil{n2^{-\ell}}$ entires with $\sqrt{\alpha\sigma^2 2^\ell\log\log_2 n/n}$ and the remaining entries $0$. It is easy to see that $\theta_\ell\in\Theta_2^{\uparrow}$ for all $\ell\in\{1,2,...,\ceil{\log_2n}\}$. For any $j<\ell$, we have
\begin{eqnarray*}
\norm{\theta_\ell-\theta_j}^2 &\geq& \ceil{n2^{-\ell}}\Bigl(\sqrt{\frac{\alpha\sigma^2 2^\ell\log\log_2 n}{n}}-\sqrt{\frac{\alpha\sigma^2 2^j\log\log_2 n}{n}}\Bigr)^2 \\
&\geq& \alpha\sigma^2\log\log_2n\Bigl(1-2^{\frac{j-\ell}{2}}\Bigr)^2 \\
&\geq& \frac{\alpha\sigma^2}{20}\log\log_2n.
\end{eqnarray*}
Therefore,
\begin{equation}
\log\mathcal{M}\Bigl(\sqrt{\frac{\alpha\sigma^2}{20}\log\log_2n},T,\norm{\cdot}\Bigr)\geq \log\log_2n, \label{eq:packing}
\end{equation}
where $T=\big\{\theta_\ell:\ell=1,2,...,\ceil{\log_2n}\big\}$. Moreover, since $\norm{\theta_\ell}^2\leq 3\alpha\sigma^2\log\log_2n$ for all $\ell$, we have
\begin{equation}
\max_{\theta,\theta'\in T}\frac{1}{2\sigma^2}\norm{\theta-\theta'}^2\leq 6\alpha\log\log_2n. \label{eq:KL}
\end{equation}
Hence, by (\ref{eq:packing}) and (\ref{eq:KL}), we can choose a very small $\alpha>0$ to ensure the ratio $\frac{\max_{\theta,\theta'\in T}\|\theta-\theta'\|^2/(2\sigma^2)}{\log\mathcal{M}(\epsilon,T,\|\cdot\|)}$ to be small. This leads to the minimax lower bound
$$\inf_{\wh{\theta}}\sup_{\theta^*\in\Theta_2^{\uparrow}}\mathbb{E}\|\wh{\theta}-\theta\|^2\geq c\sigma^2\log\log (16n),$$
for $k=2$.

For a general $k>2$, the idea is to divide the integer set $[n]$ into $\ceil{k/2}-1$ consecutive intervals with length approximately $\floor{2n/k}$. Then, we can apply the above construction to each of the $\ceil{k/2}-1$ interval. For each interval, a lower bound $c\sigma^2\log\log(2n/k)$ is obtained. Summing up these lower bounds over all the $k/2$ intervals, we get the desired rate. Details of this argument will be given in Section \ref{sec:proof-lower}, and the according minimax lower bound is presented as follows.

\begin{thm} \label{thm:lower}
There exists some universal constant $c>0$, such that
$$\inf_{\wh{\theta}}\sup_{\theta^*\in\Theta_{k}^{\uparrow}}\mathbb{E}\|\wh{\theta}-\theta^*\|^2\geq \begin{cases}
c\sigma^2, & k=1, \\
c\sigma^2k\log\log(16n/k), & k\geq 2,
\end{cases}$$
where the infimum is taken over all measurable functions  of $X$ 
and the expectation is taken under which $X\sim N(\theta^*,\sigma^2 I_n)$.
\end{thm}

Combining the results of Theorem \ref{thm:upper2} and Theorem \ref{thm:lower}, we obtain the minimax rate of the problem
$$\inf_{\wh{\theta}}\sup_{\theta^*\in\Theta_{k}^{\uparrow}}\mathbb{E}\|\wh{\theta}-\theta^*\|^2\asymp  \begin{cases}
\sigma^2, & k=1, \\
\sigma^2k\log\log(16n/k), & 2 \le k\le n.
\end{cases}$$
The minimax rate implies that the iterated logarithmic dependence on $n$ is an essential feature of the space $\Theta_{k}^{\uparrow}$. 

\section{Adaptive estimation}\label{sec:adapt}

The estimator \eqref{eq:ls} that achieves the minimax rate requires the knowledge of $k$. This section proposes an adaptive estimator that can also achieve the minimax rate without knowing the value of $k$. Recalling the notation $\wh{\theta}(\Theta_{k}^{\uparrow})=\argmin_{\theta\in\Theta_k^{\uparrow}}\|X-\theta\|^2$, we propose an adaptive estimator $\wh{\theta}=\wh{\theta}(\Theta_{\wh{k}}^{\uparrow})$ with a data-driven $\wh{k}$. The data-driven $\wh{k}$ is defined through the following penalized least-squares optimization. That is,
\begin{equation}
\wh{k}=\argmin_{k\in[n]}\Big\{\|X-\wh{\theta}(\Theta_{k}^{\uparrow})\|^2+\text{pen}_{\tau}(k)\Big\}.\label{eq:wh-k-def}
\end{equation}
Inspired by the minimax rate, the penalty function is defined by
\begin{equation}
\text{pen}_{\tau}(k)=\begin{cases}
\tau, & k=1, \\
\tau k\log\log(16n/k), & 2\leq k\leq n.
\end{cases}\label{eq:pen-n-1}
\end{equation}
The estimator $\wh{\theta}$ enjoys the following adaptive oracle inequality.

\begin{thm}\label{thm:adaptive-oracle}
Consider $X=\theta^*+Z$ with any $\theta^*\in\mathbb{R}^n$ and $Z$ satisfying (\ref{eq:moment}).
We use the estimator $\wh{\theta}=\wh{\theta}(\Theta_{\wh{k}}^{\uparrow})$ with $\wh{k}$ defined in \eqref{eq:wh-k-def}. The tuning parameter is chosen as $\tau= C'\sigma^2$ for some sufficiently large universal constant $C'>0$. Then, we have
$$\mathbb{E}\|\wh{\theta}-\theta^*\|^2\leq C\min_{1\leq k\leq n}\Big\{\inf_{\theta\in\Theta_k^{\uparrow}}\|\theta-\theta^*\|^2 + \text{pen}_{\tau}(k)\Big\}$$
with some universal constant $C>0$.
\end{thm}

\begin{remark}
Unlike in isotonic regression, an implicit assumption of Theorem \ref{thm:adaptive-oracle} is that we need to know the order of the variance $\sigma^2$. When $Z_i\sim N(0,\sigma^2)$, the unknown $\sigma$ can be estimated by the following robust procedure, 
$$
{\wh \sigma} = \frac{\hbox{\rm Median}\Big(|X_{i+1}-X_{i}|, 1\le i<n\Big)}
{\sqrt{2}\,\hbox{\rm Median}(|N(0,1)|)},
$$
As $\left|\{i: \big|\E\big[ X_{i+1}-X_{i} \big]\big| > \epsilon_0\sigma\big\}\right|$ 
is bounded by $k-1$ when $\theta^*$ has $k$ pieces and by $\|\theta^*\|_1/(\epsilon_0\sigma)$ in general, 
the above ${\wh \sigma}$ is consistent when $\min(k,\|\theta^*\|_1/\sigma)=o(n)$ and is of the order $\sigma$ when $\min(k,\|\theta^*\|_1/\sigma)\le c_0n$ for some fixed small enough constant $c_0>0$. 
On the other hand, estimation of $\sigma^2$, or even just its order, is impossible when $\theta^*$ is arbitrary. In this case, whether it is still possible to achieve the oracle inequality in Theorem \ref{thm:adaptive-oracle} is an interesting open problem.
\end{remark}

Theorem \ref{thm:adaptive-oracle} can be viewed as an adaptive version of Theorem \ref{thm:upper2}. The oracle inequality automatically selects the best $k$ that achieves the optimal bias-variance tradeoff. When the true mean vector $\theta^*$ does belong to the space $\Theta_k^{\uparrow}$, we have $\mathbb{E}\|\wh{\theta}-\theta^*\|^2\lesssim\text{pen}_{\tau}(k)$, and thus the minimax rate is achieved without the knowledge of $k$.

When $\theta^*\in\Theta_n^{\uparrow}$ so that it is isotonic, the above oracle inequality can be further improved.
By \cite{meyer2000degrees} and \cite{zhang2002risk}, as $\theta^*$ is isotonic, the estimator $\wh{\theta}^{(iso)}=\wh{\theta}(\Theta_n^{\uparrow})$ satisfies the risk bound
\begin{equation}
\mathbb{E}\|\wh{\theta}^{(iso)}-\theta^*\|^2\lesssim \sigma^2\Big\{\log(en)+n^{1/3}(V(\theta^*)/\sigma)^{2/3}\Big\},\label{eq:cubic rate}
\end{equation}
where $V(\theta^*)=\theta_n^*-\theta_1^*$ is the total variation of the vector $\theta^*$. This risk bound can be significantly smaller than $\text{pen}_{\tau}(k)$ when $V(\theta^*)/\sigma$ is small and $k$ is large.
This motivates us to modify the value of $\text{pen}_{\tau}(n)$ to achieve
the better rate between (\ref{eq:cubic rate}) and (\ref{eq:pen-n-1}). A direct choice of the modified penalty is just the bound on the right hand side of (\ref{eq:cubic rate}). However, this option depends on the value of $V(\theta^*)$, which may not be available in practice. Inspired by the risk analysis in \cite{zhang2002risk}, we consider
\begin{align}
\label{eq:pen-n} \tau\Big\{\log(en) +   \sum_{\{\ell\geq 0:2^{\ell}\leq n/3\}}\frac{\wh{l}_{\tau}(2^{\ell+1})-\wh{l}_{\tau}(2^{\ell})}{2^{\ell+1}}\Big\},
\end{align}
where
$$\hat{l}_{\tau}(m):=\min\Big\{n,3m + m\sqrt{m+1}\Big(\wb{X}_{[n-m:n-m/2)}-\wb{X}_{(1+m/2:1+m]}\Big)/\sqrt{\tau}\Big\}.$$
Note that (\ref{eq:pen-n}) is a data-driven estimate of the risk of $\wh{\theta}^{(iso)}$.
Then, we have a well-defined penalty function on $[n]$ by combining (\ref{eq:pen-n-1}) and (\ref{eq:pen-n}). The modified penalty function in summary is 
\begin{equation*}
\wt{\text{pen}}_{\tau}(k)=\begin{cases}
\tau, & k=1, \\
\tau \text{pen}_{\tau}(k), & 2\leq k\leq n-1, \\
\tau\Big\{\log(en) +   \sum_{\{\ell\geq 0:2^{\ell}\leq n/3\}}\frac{\wh{l}_{\tau}(2^{\ell+1})-\wh{l}_{\tau}(2^{\ell})}{2^{\ell+1}}\Big\}, &k=n.
\end{cases}
\end{equation*}
With some appropriate choice of $\tau$, the performance of $\wh{\theta}=\wh{\theta}(\Theta_{\wh{k}}^{\uparrow})$ is given by the following theorem.

\begin{thm}\label{thm:g-adaptation}
Consider $X=\theta^*+Z$ with any $\theta^*\in\Theta_n^{\uparrow}$ and $Z$ satisfying $\displaystyle \max_{1\leq i\leq n}\mathbb{E}|Z_i/\sigma|^{2+\epsilon}\leq C_1$.
We use the estimator $\wh{\theta}=\wh{\theta}(\Theta_{\wh{k}}^{\uparrow})$ with $\wh{k}$ selected by the modified penalty function $\wt{\text{pen}}_{\tau}(k)$. The tuning parameter is chosen as $\tau= C'\sigma^2$ for some sufficiently large universal constant $C'>0$. Then, we have
$$\mathbb{E}\|\wh{\theta}-\theta^*\|^2\leq C\min_{1\leq k\leq n}\Big\{\inf_{\theta\in\Theta_k^{\uparrow}}\|\theta-\theta^*\|^2 + {\rm isoerr}_k(\theta^*)\Big\},$$
for some universal constant $C>0$. The stochastic error term ${\rm isoerr}_k(\theta^*)$ is defined by
$${\rm isoerr}_k(\theta^*)= \begin{cases}
\sigma^2, & k=1, \\
\sigma^2\min\Big\{k\log\log(\frac{16n}{k}),\log(en)\!+\!n^{1/3}(\frac{V(\theta^*)}{\sigma})^{2/3}\Big\}, & k\geq 2.
\end{cases}$$
\end{thm}

We remark that the rate in the above theorem is always no greater than that of Theorem \ref{thm:adaptive-oracle}. If we further impose the condition that $V(\theta^*)/\sigma\leq n^{1-\delta}$ for some universal constant $\delta\in(0,1)$, the rate given by Theorem \ref{thm:g-adaptation} can be summarized into three phases,
$${\rm isoerr}_k(\theta^*)\asymp \begin{cases}
\sigma^2, & k=1, \\
\sigma^2 k\log\log(16n), & 2\leq k\leq \frac{\log(en)+n^{1/3}(V(\theta^*)/\sigma)^{2/3}}{\log\log(16n)}, \\
\sigma^2\Big\{\log(en)\!+\!n^{1/3}(V(\theta^*)/\sigma)^{2/3}\Big\}, & k > \frac{\log(en)+n^{1/3}(V(\theta^*)/\sigma)^{2/3}}{\log\log(16n)}.
\end{cases}$$
In other words, the adaptive estimator with the modified penalty can achieve both the minimax rates of the class $\Theta_k^{\uparrow}$ derived in this paper and the rate of isotonic regression in \cite{meyer2000degrees} and \cite{zhang2002risk}.

An interesting open problem is whether it is possible to obtain sharp oracle inequalities with the constant before the approximation error to be exactly one. The counter example constructed by \cite{rigollet2012sparse} in a sparse linear regression setting seems to suggest that this task may be impossible for the penalized least-squares procedure considered in this paper.

\section{Discussion}\label{sec:disc}

\subsection{Computational issues}\label{sec:compute}

The optimization problem (\ref{eq:ls}) is recognized as reduced isotonic regression in the literature \citep{schell1997reduced}, and related $\ell_0$ optimization problems have been studied in literature (see, for example, \cite{friedrich2008complexity} and \cite{jewell2017exact} among many others). As $k=n$, the solution to the isotonic regression problem, $\wh{\theta}(\Theta_n^{\uparrow})$,  can be computed efficiently in $O(n)$ time using the pool-adjacent-violators algorithm (PAVA) \citep{mair2009isotone}. Computation of $\wh{\theta}(\Theta_k^{\uparrow})$ for $k=1,2,...,n-1$ may seem to be combinatorial, but by taking advantage of the PAVA solution, it can be reduced to a simple dynamic programming. 

In detail, denote the set of knots (change points) of $\wh{\theta}(\Theta_k^{\uparrow})$ by $\wh{A}_k$. The following two properties are immediate from Lemma \ref{lem:cun-hui} (that will be stated in Section \ref{sec:LEMMA}):
\begin{enumerate}
\item For any $k\in[n]$, we have $\wh{A}_k\subset \wh{A}_n$;
\item For any $k\in[n]$, $\wh{\theta}(\Theta_k^{\uparrow})$ is a piecewise constant function with knots in $\wh{A}_k$. Moreover, each piece is a sample average of the $X_i$'s in that block.
\end{enumerate}
The first property asserts that the knots of $\wh{\theta}(\Theta_k^{\uparrow})$ are always contained in the solution of PAVA. The second property implies that $\wh{\theta}(\Theta_k^{\uparrow})$ can be obtained by averaging consecutive entries of $\wh{\theta}(\Theta_n^{\uparrow})$. Since $\wh{\theta}(\Theta_n^{\uparrow})$ is already isotonic, one does not need to worry about the isotonic constraint anymore, and the only task is to find the best change points among $\wh{A}_n$ that minimize the squared error loss.
Therefore, one can first run PAVA and obtain a set of potential knots $\wh{A}_n=\{t_j\}_{j=1}^{\wh{n}}$. Then, the search for the knots of $\wh{\theta}(\Theta_k^{\uparrow})$ in $\{t_j\}_{j=1}^{\wh{n}}$ can be implemented efficiently through dynamic programming. Note that $\wh{A}_k=\wh{A}_n$ for all $k\geq \wh{n}$, and we only need to find $\wh{A}_k$ for $k<\wh{n}$. Details of implementation are given in Algorithm \ref{alg:knots} for completedness.
\def\Loss{\hbox{{\rm Loss}}}
\def\knots{\hbox{{\rm knots}}}
\def\means{\hbox{{\rm means}}}
\def\knot{\hbox{{\rm left.knot}}}
\def\TLoss{\hbox{{\rm T.Loss}}}
\begin{algorithm}
\DontPrintSemicolon
\SetKwInOut{Input}{Input}\SetKwInOut{Output}{Output}
\Input{$\{X_i\}_{i=1}^n$, $t_0=0$, knots $t_1<\cdots<t_{\wh{n}}=n$ from PAVA}
\Output{$\wh{A}_k$ and the corresponding piecewise average for all $k<\wh{n}$} 
\nl For $j$ in $1:\wh{n}$, compute the partial sums of $X_i$ and $X_i^2$,\;
\qquad$S(j)\leftarrow \sum_{0<i\leq t_j}X_i,\quad SS(j) \leftarrow \sum_{0<i\le t_j}X_i^2$.

\nl For all $(\ell,j)$ such that $0\leq \ell<j\leq \wh{n}$, compute the loss for fitting by mean in $(t_{\ell}:t_j]$,\;
\qquad $\Loss(\ell,j) \leftarrow SS(j)-SS(\ell) - (S(j)-S(\ell))^2/(a_j-a_\ell)$.

\nl For $j$ in $1:\wh{n}$, copy the loss for fitting by mean in $(0:t_j]$,\;
\qquad $\TLoss(1,j) \leftarrow \Loss(0,j)$.

\nl For $k$ in $2:\wh{n}-1$\;
\qquad For $j$ in $k:\wh{n}$, compute the minimal loss for $k$-piece monotone fit in $(0:t_j]$,\;
\qquad\qquad $\knot(k,j) \leftarrow \argmin_{1\le \ell<j}\{ \TLoss(k-1,\ell) + \Loss(\ell,j)\}$,\;
\qquad\qquad	$\TLoss(k,j) \leftarrow \TLoss(k-1,\knot(k,j)) + \Loss(\knot(k,j),j)$.\;
\qquad	$\knots(k,k) = \wh{n}$.\;
\qquad For $j$ in $(k-1):1$, compute $\wh{A}_k$, \;
\qquad\qquad	$\knots(k,j) \leftarrow \knot(j+1,\knots(k,j+1))$.

\caption{Computation of $\wh{A}_k$ for all $k<\wh{n}$\label{alg:knots}}
\end{algorithm}

Since Algorithm \ref{alg:knots} computes $\wh{\theta}(\Theta_k^{\uparrow})$ for all $k$, one can directly use the results to obtain the adaptive estimator $\wh{\theta}=\wh{\theta}(\Theta_{\wh{k}}^{\uparrow})$ via (\ref{eq:wh-k-def}). By \cite{friedrich2008complexity}, the complexity of Algorithm \ref{alg:knots} is $O(\wh{n}^3)$ after PAVA. Therefore, the overall complexity of (\ref{eq:wh-k-def}) is $O(n+\wh{n}^3)$. This leads to a worst-case complexity bound $O(n^3)$. However, since $\wh{n}$ enjoys the rate $\sigma^2\{V/\sigma + \log(en)+n^{1/3}(V/\sigma)^{2/3}\}$ 
by Theorem 1 of \cite{meyer2000degrees}, 
with high probability the isotonic regression (or PAVA) yields an $\wh{n}$ of order $O(n^{1/3})$ when $V/\sigma=O(1)$. This leads to a linear complexity $O(n)$.


\subsection{Comparison with piecewise constant models}\label{sec:compare-piece}

A closely related problem to estimating isotonic piecewise constant functions is the estimation of piecewise constant signals without the monotone condition. We define the space of piecewise constant models as
\begin{align} \label{eq:tkg}
\Theta_k &= \Big\{\theta\in\mathbb{R}^n: \text{there exist }\{a_j\}_{j=0}^k\text{ and }\{\mu_j\}_{j=1}^k\text{ such that }\\
&0=a_0\leq a_1\leq\cdots\leq a_k=n,\text{ and }\theta_i=\mu_j\text{ for all }i\in(a_{j-1}:a_j]\Big\}. \notag
\end{align}
This section shows that $\Theta_{k}^{\uparrow}$ and $\Theta_{k}$ have different error behaviors.

\begin{thm}\label{thm:upper1}
For any $k\in[n]$, the minimax rate for the space $\Theta_{k}$ is given by
$$\inf_{\wh{\theta}}\sup_{\theta^*\in\Theta_{k}}\mathbb{E}\|\wh{\theta}-\theta^*\|^2\asymp\begin{cases}
\sigma^2, & k=1, \\
\sigma^2\log\log(16n), & k=2, \\
\sigma^2k\log(en/k), & k\geq 3,
\end{cases}$$
where the expectation is taken over the distribution $X\sim N(\theta^*,\sigma^2I_n)$.
\end{thm}

The upper bound in Theorem \ref{thm:upper1} can be achieved by the least-squares estimator $\wh{\theta}(\Theta_{k})=\argmin_{\theta\in\Theta_{k}}\|X-\theta^*\|^2$ when $k$ is known, or achieved by its penalized version when $k$ is unknown. The penalty can be chosen proportional to the minimax rate, following the classic approach in, for example, \cite{birge1993rates} and \cite{birge2001gaussian}. These estimators can be computed efficiently via dynamic programming \citep{friedrich2008complexity}.

We emphasize that the results for $k\geq 3$ are well known in the literature \citep{donoho1994minimax,birge2001gaussian,boysen2009consistencies,raskutti2011minimax,li2016fdr} and we claim no originality there. Instead, our stress is on comparing $\Theta_{k}^{\uparrow}$ and $\Theta_{k}$. First, it can be seen that the main difference between these two spaces is that the minimax rate of the former scales as $\sigma^2k\log\log(16n/k)$, while that of the latter scales as $\sigma^2k\log(en/k)$, for $k\geq 3$. The case $k=2$ is special, and both spaces have minimax rates $\log\log (16n)$. This is because the signals in $\Theta_{2}$ is either nondecreasing or nonincreasing. 


Secondly, we emphasize that the minimax rate of $\Theta_k$ is only for  the Gaussian observations $X\sim N(\theta^*,\sigma^2I_n)$. With regard to the upper bound, the assumption of Gaussian errors can be easily relaxed to sub-Gaussian errors. However, the sub-Gaussianity cannot be further relaxed, as illustrated below.
Consider the observation $X=\theta^*+Z\in\mathbb{R}^n$. Assume i.i.d. error variables $Z_1,...,Z_n\sim p_{\gamma}$, where the density function is specified as
\begin{equation}
p_{\gamma}(x)\propto \exp\Bigl(-|x|^{\gamma}\Bigr),\label{eq:gamma-dist}
\end{equation}
for some $\gamma\in(0,2]$. When $\gamma=2$, we recover the Gaussian-like (sub-Gaussian) error. For $\gamma\in(0,2)$, we get a heavier tail than the Gaussian one. The following proposition shows that the sub-Gaussian assumption cannot be relaxed.
\begin{proposition}\label{prop:imp}
Consider the error distribution (\ref{eq:gamma-dist}) for some $\gamma\in(0,2]$. For the space $\Theta_3$, we have the lower bound,
$$\inf_{\wh{\theta}}\sup_{\theta^*\in\Theta_3}\mathbb{E}\norm{\wh{\theta}-\theta^*}^2\geq c\Bigl\{\log (en)\Bigr\}^{2/\gamma},$$
for some universal constant $c>0$.
\end{proposition}
Since the desired minimax rate for $\Theta_3$ is $\sigma^2\log (en)$, Proposition \ref{prop:imp} implies that the minimax rate under the Gaussian assumption cannot be achieved unless $\gamma=2$. In other words, unlike Theorem \ref{thm:upper2}, a sub-Gaussian tail is necessary for the result of Theorem \ref{thm:upper1}, the second important difference between the two spaces $\Theta_{k}^{\uparrow}$ and $\Theta_{k}$.

We end this section with a relatively technical discussion of the difference between models $\Theta_k$ and $\Theta_k^\uparrow$. Denoting the estimated change points of $\hat\theta(\Theta_k^\uparrow)$ as $\{\wh{a}_j\}$. Consider the case where $\theta^*_i=\mu$ for $a\le i\le b$ and 
$a\le\wh{a}_{j-1} < \wh{a}_j\le b$. 
The error for $\wh{a}_{j-1}\le i\le  \wh{a}_j$ is $|\wb{X}_{(\wh{a}_{j-1}:\wh{a}_j]}-\mu|^2$. 
For simplicity of discussion, let us suppose $\wh{a}_j-\wh{a}_{j-1}$ is of order $n/k$. 
Without isotonic constraint, few additional structure is exploitable between $\wh{a}_{j-1}$ and $\wh{a}_j$, and the optimal fit is shown to suffer an extra logarithmic factor.  With isotonic constraint, on the other hand, the two change points $\wh{a}_{j-1}$ and $\wh{a}_j$ have an additional constraint:
\begin{equation*}
|\wb{X}_{(\wh{a}_{j-1}:\wh{a}_j]}-\mu|^2\leq |\wb{X}_{(a:\wh{a}_j]}-\mu|^2\vee|\wb{X}_{(\wh{a}_{j-1}:b]}-\mu|^2.
\end{equation*}
Now for each term on the right hand side above, one end point is random and the other is fixed. Therefore, both $|\wb{X}_{(a:\wh{a}_j]}-\mu|^2$ and $|\wb{X}_{(\wh{a}_{j-1}:b]}-\mu|^2$ are of order $k\log\log(16n/k)/n$, implied by the asymptotics of partial sum processes (cf. Lemma \ref{lem:simple-lil}).

\subsection{Implications for change-point detection}

The lower bound result in the paper is strongly related to the problem of determining the ``region of detectability" (ROD) in the change-point detection literature. On one hand, when there are multiple change-points, the ROD has been established in \cite{arias2005near}, where these authors show that in various settings a signal strength of the order at least $\sqrt{\log (en)/n}$ is necessary for consistent detection. A gap exists when there is only one change-point. 

The result of Theorem \ref{thm:lower} helps close this gap. As a matter of fact, by a slight modification of the proof of Theorem \ref{thm:lower} for the case $k=2$, it is straightforward to prove the following proposition. The result shows that it is impossible to differentiate the one-step function from a two-step function when the signal gap is of  order smaller than $\sqrt{\log\log (16n)/n}$. On the other hand, consistent detection of signal when the gap is of a comparable order has already been established (see, for example, Chapter 1.5 in \cite{csorgo1997limit}).
\begin{proposition}\label{cor:test} Let $\mathbb{E}_{\theta}$ stand for the expectation induced by $N(\theta,\sigma^2I_n)$. Define the following parameter space:
\[
\Theta_2(c) := \Big\{ \theta\in\Theta_2: (\mu_2-\mu_1)^2\cdot (a_1\wedge (n-a_1))>c\sigma^2\log\log (16n)  \Big\},
\]
where $\mu_1,\mu_2,a_1$ are defined in \eqref{eq:tkg}. We then have, for some small enough universal constant $c>0$,
\[
\inf_{0\leq \phi\leq 1}\Big\{\sup_{\theta\in\Theta_1}\mathbb{E}_\theta\phi+\sup_{\theta\in\Theta_2(c)}\mathbb{E}_{\theta}(1-\phi)\Big\}\geq c_1,
\]
where $c_1$ is another universal constant in $(0,1)$.
\end{proposition}

Proposition \ref{cor:test} complements Theorem 2.3 in \cite{arias2005near}, and both results together give a clear picture of the ROD when one or multiple change-points are present.

\subsection{Minimax rates for unimodal piecewise constant functions}\label{sec:unimodal}

The class of unimodal functions is widely studied in the literature \citep{bickel1996some, birge1997estimation, shoung2001least, kollmann2014unimodal}. It is often studied side by side with the isotonic functions \citep{boyarshinov2006linear, stout2008unimodal}. In this section, we show that the techniques developed in this paper also lead to the derivation of the minimax rate of the class of unimodal piecewise constant functions. We define the parameter space of interest as follow,
\begin{align*}
\Theta_k^{\wedge} =& \Big\{\theta\in\mathbb{R}^n: \text{there exist }\{a_j\}_{j=0}^k\text{ and }\{\mu_j\}_{j=1}^k\text{ such that }\\
&~~0=a_0\leq a_1\leq\cdots\leq a_k=n,\mu_1\leq\cdots\leq\mu_{\ell-1}\leq\mu_{\ell}\geq\mu_{\ell+1}\geq\cdots\geq\mu_k, \\
&~~\text{ and }\theta_i=\mu_j\text{ for all }i\in(a_{j-1}:a_j]\Big\}.
\end{align*}
This class has been studied by \cite{chatterjee2015adaptive}, who provide an upper bound of order $\sigma^2k\log(en)$.
It is interesting to note the relation
$$\Theta_k^{\uparrow}\subset\Theta_k^{\wedge}\subset\Theta_k,
$$
which indicates that the minimax rate of $\Theta_k$ is between those of $\Theta_k^{\uparrow}$ and $\Theta_k$. The following theorem gives the exact minimax rate.
\begin{thm}\label{thm:unimodal}
For any $k\in[n]$, the minimax rate for the space $\Theta_{k}^{\wedge}$ is given by
$$\inf_{\wh{\theta}}\sup_{\theta\in\Theta_{k}^{\wedge}}\mathbb{E}\|\wh{\theta}-\theta^*\|^2\asymp\begin{cases}
\sigma^2, & k=1, \\
\sigma^2\log\log(16n), & k=2, \\
\sigma^2\log(en), & 3\leq k\leq \frac{\log(en)}{\log\log(16n)}, \\
\sigma^2k\log\log(16n/k), & k>\frac{\log(en)}{\log\log(16n)},
\end{cases}$$
where the expectation is taken over the distribution $X\sim N(\theta^*,\sigma^2I_n)$.
\end{thm}
Interestingly, we observe that the minimax rates have four phases, and can have either a logarithmic behavior or an iterated logarithmic behavior, depending on the regime of $k$. 
When $k=2$, the minimax rate is driven by the isotonic structure. When $3\leq k\leq \frac{\log(en)}{\log\log(16n)}$, the rate $\sigma^2\log(en)$ results from the uncertainty of the mode of the function.
Finally,
the $\sigma^2k\log\log(16n/k)$ rate for a large $k$ is again driven by the isotonic structure of a unimodal function. 


\nocite{han2016multivariate,kim2016adaptation}

\subsection{Minimax rates under $\ell_p$ loss with $1\leq p<2$} \label{sec:general-lp}

Section \ref{sec:upper} gives the minimax rate of the space $\Theta_{k}^\uparrow$ with respect to the squared $\ell_2$ loss. In particular, Theorems \ref{thm:upper2} and \ref{thm:lower} show that the minimax rate involves an interesting iterated logarithmic term. This is in contrast with the original isotonic regression estimator $\hat\theta(\Theta_n^{\uparrow})$, which is of an additional logarithmic term in view of Proposition \ref{prop:lower-LSE}. 

In this section we present an interesting phenomenon that a reversed argument applies to an $\ell_p$ loss with $1\leq p<2$. For this, we first reveal that the difference between the minimax risk of $\Theta_{k}^\uparrow$ and the rate of $\hat\theta(\Theta_n^{\uparrow})$ will vanish when we consider an $\ell_p$ loss with $1\leq p<2$. 

\begin{proposition}\label{prop:lp-risk}
Consider $X=\theta^*+Z$ with $Z_1,\ldots,Z_n$ independent, mean zero, and satisfying $\E|Z_i/\sigma|^2\leq C_1$ for some universal constant $C_1>0$. We then have, for any $k\in[n]$ and $1\leq p<2$, 
\[
\sup_{\theta^*\in\Theta_{k}^{\uparrow}}\E\norm{\hat\theta(\Theta_n^{\uparrow})-\theta^*}_p^p\leq C\sigma^pn(k/n)^{p/2}
\]
for some universal constant $C>0$. On the other hand, there exists some universal constant $c>0$ such that, for any $1\leq p<2$,
\[
\inf_{\hat\theta}\sup_{\theta^*\in\Theta_{k}^{\uparrow}}\E\norm{\hat\theta-\theta^*}^p_p\geq c\sigma^pn(k/n)^{p/2},
\]
where the infimum is taken over all measurable functions of $X$ and the expectation is taken under which $X\sim N(\theta^*,\sigma^2I_n)$.
\end{proposition}

Secondly we show that, quite interestingly, the reduced isotonic regression estimator cannot recover the above minimax risk under an $\ell_p$ loss with $1\leq p<2$, even if it is the maximum likelihood estimator of the truth.   

\begin{proposition}\label{prop:lp-reduced}
Consider $X=\theta^*+Z$ with $Z\sim N(0, \sigma^2I_n)$. Then, for any $1\leq p\leq 2$ and $2\leq k\leq n$, we have
\[
\sup_{\theta^*\in\Theta_{k}^\uparrow}\E\norm{\hat\theta(\Theta_{k}^\uparrow)-\theta^*}_p^p \asymp \sigma^pn\{k\log\log(16n/k)/n\}^{p/2}.
\]
\end{proposition}
Unlike the estimator $\wh{\theta}(\Theta_n^{\uparrow})$, for $\wh{\theta}(\Theta_k^{\uparrow})$, the iterated logarithmic term does not disappear when an $\ell_p$ loss with $1\leq p<2$ is considered. Since Proposition \ref{prop:lp-reduced} gives both upper and lower bounds for the $\ell_p$ risk, the reduced isotonic regression estimator $\wh{\theta}(\Theta_k^{\uparrow})$ is not optimal for the class $\Theta_k^{\uparrow}$ when $1\leq p<2$, compared with the minimax rate given in Proposition \ref{prop:lp-risk}. This indicates that, compared to the classical isotonic regression estimator, the performance of the reduced isotonic regression estimator hinges more on its definition, that is, minimizing the squared $\ell_2$ risk. This interesting phenomenon is summarized by the following table. The rates displayed are for the normalized $\ell_p$ loss $\|\hat\theta-\theta^*\|_p^p/(n\sigma^p)$.

\begin{center}
  \begin{tabular}{ l | c | c | r }
    \hline
     & minimax rate & $\wh{\theta}(\Theta_k^{\uparrow})$ & $\wh{\theta}(\Theta_n^{\uparrow})$\\ \hline
    $1\leq p<2$ & $\left(\frac{k}{n}\right)^{p/2}$ & $\left(\frac{k\log\log(16n/k)}{n}\right)^{p/2}$ & $\left(\frac{k}{n}\right)^{p/2}$\\ \hline
    $p=2$ & $\frac{k\log\log(16n/k)}{n}$ & $\frac{k\log\log(16n/k)}{n}$ & $\frac{k\log(en/k)}{n}$\\
    \hline
  \end{tabular}
\end{center}

\section{Proofs}\label{sec:proof}

This section contains the proofs of the main results in Sections \ref{sec:upper} and \ref{sec:adapt}, with the remaining proofs and auxiliary lemmas relegated to the supplement. In the sequel, by convention the summation over an empty set is set to be 0.


\subsection{A critical lemma}\label{sec:LEMMA}
Before stating the proofs of all theorems in the paper, we first present a very important lemma that characterizes the solution $\wh{\theta}(\Theta_k^{\uparrow})$ of the reduced isotonic regression (\ref{eq:ls}). Below, we use the notation $\wh{\theta}^{(k)}$ for $\wh{\theta}(\Theta_k^{\uparrow})$, and recall the set of knots of $\wh{\theta}^{(k)}=\{\wh{a}_j\}$ is denoted as $\wh{A}_k$.
\begin{lemma}\label{lem:cun-hui}
The following properties of estimator $\wh{\theta}^{(k)}=\wh{\theta}(\Theta_k^{\uparrow})$ hold.
\begin{enumerate}
\item For each $j$, $\wh{\theta}_i^{(k)}=\wb{X}_{(\wh{a}_{j-1}:\wh{a}_j]}$ for all $i\in(\wh{a}_{j-1}:\wh{a}_j]$. 
\item For each $j$, we have $\wb{X}_{(s:\wh{a}_j]}<\frac{\wh{\theta}_{\wh{a}_j}^{(k)}+\wh{\theta}^{(k)}_{\wh{a}_{j+1}}}{2}<\wb{X}_{(\wh{a}_j:t]}$ for all $0\leq s<\wh{a}_j<t\leq n$. As a consequence, $\wh{\theta}_{\wh{a}_j}^{(n)}<\frac{\wh{\theta}^{(k)}_{\wh{a}_j}+\wh{\theta}^{(k)}_{\wh{a}_{j+1}}}{2}<\wh{\theta}_{\wh{a}_j+1}^{(n)}$.
\item The set of knots satisfies $\wh{A}_k\subset\wh{A}_n$.
\end{enumerate}
\end{lemma}
These three results in Lemma \ref{lem:cun-hui} are all deterministic consequences of the optimization problem (\ref{eq:ls}). The first conclusion asserts that given the set of knots $\wh{A}_k$, the value of $\wh{\theta}^{(k)}_i$ is a simple average of $X$ in each block $(\wh{a}_{j-1}:\wh{a}_j]$. The second conclusion is due to the isotonic constraint in (\ref{eq:ls}), and is also the reason why we can apply a non-asymptotic law of iterated logarithm bound for the risk (see the proof of Theorem \ref{thm:upper2}). Finally, the last conclusion $\wh{A}_k\subset\wh{A}_n$ leads to the efficient computational strategy we discuss in Section \ref{sec:compute}. The proof of the lemma is given below.
\begin{proof}[Proof of Lemma \ref{lem:cun-hui}]
For notational simplicity, we use $\wh{\theta}$ for $\wh{\theta}^{(k)}=\wh{\theta}(\Theta_k^{\uparrow})$ in the proof.
We first show that $\wh{\theta}_i=\wb{X}_{(\wh{a}_{j-1}:\wh{a}_j]}$ for all $i\in(\wh{a}_{j-1}:\wh{a}_j]$. By the definition of $\Theta_k^{\uparrow}$, the optimization $\min_{\theta\in\Theta_k^{\uparrow}}\|X-\theta\|^2$ can be equivalently written as
\begin{align*}
&\min_{a_0\leq\cdots\leq a_k}\min_{\mu_1\leq\cdots\leq\mu_k}\sum_{j=1}^k\sum_{i=a_{j-1}+1}^{a_j}(X_i-\mu_j)^2= \\ &\min_{a_0\leq\cdots\leq a_k}\Big\{\sum_{j=1}^k\sum_{i=a_{j-1}+1}^{a_j}(X_i\!-\!\wb{X}_{(a_{j-1}:a_j]})^2+\!\!\!\!\min_{\mu_1\leq\cdots\leq\mu_k}\sum_{j=1}^k(a_j\!-\!a_{j-1})(\mu_j\!-\!\wb{X}_{(a_{j-1}:a_j]})^2\Big\}.
\end{align*}
The optimization problem 
$$\min_{\mu_1\leq\cdots\leq\mu_k}\sum_{j=1}^k(a_j-a_{j-1})(\mu_j-\wb{X}_{(a_{j-1}:a_j]})^2$$
is in the form of weighted isotonic regression. Therefore,
its solution can be represented as
\begin{equation}
\wh{\mu}_j=\min_{v\geq j}\max_{u\leq j}\wb{X}_{(a_{u-1}:a_v]}.\label{eq:w-minmax}
\end{equation}
This fact can be derived using the same proof of the minimax formula of isotonic regression (cf. Proposition 2.4.2 in \cite{silvapulle2011constrained}). Now suppose $(\wt{a}_0,...,\wt{a}_k)$ is a minimizer, then the solution has the form $\wh{\theta}_i=\min_{v\geq j}\max_{u\leq j}\wb{X}_{(\wt{a}_{u-1}:\wt{a}_v]}$ for all $i\in(\wt{a}_{j-1}:\wt{a}_j]$. Note that the values in the $k$ intervals satisfy $\wh{\mu}_1\leq\cdots\leq\wh{\mu}_k$. We can combine any two adjacent interval if $\wh{\mu}_{j-1}=\wh{\mu}_j$. Then, by the formula (\ref{eq:w-minmax}), there exist $\{\wh{a}_j\}$ such that $\wh{\theta}_i=\wb{X}_{(\wh{a}_{j-1}:\wh{a}_j]}$ for all $i\in(\wh{a}_{j-1}:\wh{a}_j]$.

Now we prove the second point. By symmetry, it is sufficient to prove $(\wh{\theta}_{\wh{a}_j}+\wh{\theta}_{\wh{a}_{j+1}})/2<\wb{X}_{(\wh{a}_j:t]}$. Moreover, as $\wh{\theta}_i$ is nondecreasing in $i$, it suffices to consider $\wh{a}_j<t<\wh{a}_{j+1}$. There are three possible cases.

\textit{Case 1.} ${\overline X}_{(t:\wh{a}_{j+1}]} \neq \wh{\theta}_{\wh{a}_{j+1}}$ and ${\overline X}_{(t:\wh{a}_{j+1}]}\geq \wh{\theta}_{\wh{a}_j}$.
By the optimality of $\wh{\theta}(\Theta_k^\uparrow)$, assigning $\wh{\theta}_{\wh{a}_j}$ to $\wh{\theta}_i$ for all 
$i\in (\wh{a}_j:t]$ does not provide a better fit, 
$$
\sum_{i = \wh{a}_j+1}^t (X_i - \wh{\theta}_{\wh{a}_{j+1}})^2
+ \sum_{i = t+1}^{\wh{a}_{j+1}} (X_i - \wh{\theta}_{\wh{a}_{j+1}})^2
\le \sum_{i = \wh{a}_j+1}^t (X_i - \wh{\theta}_{\wh{a}_j})^2 
+ \sum_{i = t+1}^{\wh{a}_{j+1}} (X_i - {\overline X}_{(t:\wh{a}_{j+1}]})^2. 
$$
It follows that 
$$
(t-\wh{a}_j)({\overline X}_{(\wh{a}_j:t]} - \wh{\theta}_{\wh{a}_{j+1}})^2
+ (\wh{a}_{j+1}-t)({\overline X}_{(t:\wh{a}_{j+1}]} - \wh{\theta}_{\wh{a}_{j+1}})^2
\le (t-\wh{a}_j)({\overline X}_{(\wh{a}_j:t]} - \wh{\theta}_{\wh{a}_j})^2. 
$$
This leads to $|{\overline X}_{(\wh{a}_j:t]} - \wh{\theta}_{\wh{a}_{j+1}}| 
< |{\overline X}_{(\wh{a}_j:t]} - \wh{\theta}_{\wh{a}_j}|$, which further implies $\wb{X}_{(\wh{a}_j:t]}>(\wh{\theta}_{\wh{a}_j}+\wh{\theta}_{\wh{a}_{j+1}})/2$. 

\textit{Case 2.} ${\overline X}_{(t:\wh{a}_{j+1}]} = \wh{\theta}_{\wh{a}_{j+1}}$. Since $\wh{\theta}_{\wh{a}_{j+1}}=\wb{X}_{(\wh{a}_j:\wh{a}_{j+1}]}$ is a weighted average of ${\overline X}_{(t:\wh{a}_{j+1}]}$ and ${\overline X}_{(\wh{a}_j:t]}$, we have ${\overline X}_{(\wh{a}_j:t]}=\wh{\theta}_{\wh{a}_{j+1}}>\wh{\theta}_{\wh{a}_{j}}$. Thus, we still have $\wb{X}_{(\wh{a}_j:t]}>(\wh{\theta}_{\wh{a}_j}+\wh{\theta}_{\wh{a}_{j+1}})/2$.

\textit{Case 3.} ${\overline X}_{(t:\wh{a}_{j+1}]}< \wh{\theta}_{\wh{a}_j}$. By the definition of $\{\wh{a}_j\}$, we have $\wh{\theta}_{\wh{a}_{j+1}}>\wh{\theta}_{\wh{a}_j}$. Moreover, since $\wh{\theta}_{\wh{a}_{j+1}}=\wb{X}_{(\wh{a}_j:\wh{a}_{j+1}]}$ is a weighted average of ${\overline X}_{(t:\wh{a}_{j+1}]}$ and ${\overline X}_{(\wh{a}_j:t]}$, we must have ${\overline X}_{(\wh{a}_j:t]}>\wh{\theta}_{\wh{a}_{j+1}}$ and ${\overline X}_{(\wh{a}_j:t]}>\wh{\theta}_{\wh{a}_{j}}$, which also leads to $\wb{X}_{(\wh{a}_j:t]}>(\wh{\theta}_{\wh{a}_j}+\wh{\theta}_{\wh{a}_{j+1}})/2$.

Finally, we have
$$\wh{\theta}_{\wh{a}_j+1}^{(n)}=\min_{b\geq \wh{a}_j+1}\max_{a\leq \wh{a}_j+1}\wb{X}_{[a:b]}\geq \min_{b>\wh{a}_j}\wb{X}_{(\wh{a}_j:b]}>(\wh{\theta}_{\wh{a}_j}+\wh{\theta}_{\wh{a}_{j+1}})/2.$$
By symmetry, we also have $\wh{\theta}_{\wh{a}_j}^{(n)}<(\wh{\theta}_{\wh{a}_j}+\wh{\theta}_{\wh{a}_{j+1}})/2$, and therefore $\wh{\theta}_{\wh{a}_j}^{(n)}<\wh{\theta}_{\wh{a}_j+1}^{(n)}$, meaning that $\wh{a}_j$ is also a change point for $\wh{\theta}^{(n)}$, which immediately implies the last conclusion $\wh{A}_k\subset\wh{A}_n$.
\end{proof}

\subsection{Proofs of upper bounds}\label{sec:pf-upper}

In this section, we state the proofs of Theorems \ref{thm:upper2}  and \ref{thm:adaptive-oracle}.

\begin{proof}[Proof of Theorem \ref{thm:upper2}]
We first introduce notations that are needed in the proof. We shorthand $\wh{\theta}(\Theta_k^{\uparrow})$ by $\wh{\theta}$. The set of knots of $\wh{\theta}$ is denoted by $\wh{A}_k=\{\wh{a}_h\}$. Define the oracle
\begin{equation}
\theta^{(k)} = \argmin_{\theta\in\Theta_k^{\uparrow}}\|\theta-\theta^*\|^2. \label{eq:oracle-k}
\end{equation}
The set of knots of $\theta^{(k)}$ is denoted by $A_k=\{a_j\}$ where we allow overlaps within $a_1,\ldots,a_k$.
For the error vector $Z=X-\theta^*$ and two integers $1\leq a\leq b\leq n$, define random variables
\begin{align}
\label{eq:def+1} \xi_+(a,b,\ell) &\!=\! 2^\ell\max\Bigl\{|\wb{Z}_{(a:t]}|^2\!:\!a+2^{\ell-1}\!\leq\! t\!\leq\! b\wedge(a+2^\ell-1)\Bigr\}, \\
\label{eq:def+2} \delta_+(a,b,\ell) &\!=\! \max_{\{h:a<\wh{a}_h\leq b\}}\ind\Big\{a+2^{\ell-1}\!\leq\! \wh{a}_h\!\leq\!b\wedge(a+2^\ell-1)\Big\}, \\
\label{eq:def-1} \xi_-(a,b,\ell) &\!=\! 2^\ell\max\Bigl\{|\wb{Z}_{(t:b]}|^2:{a}\vee(b\!+\!2\!-\!2^\ell)\!\leq\! t\!\leq\! b\!+\!1\!-\!2^{\ell-1}\Bigr\}, \\
\label{eq:def-2} \delta_-(a,b,\ell) &\!=\! \max_{\{h:a<\wh{a}_{h}\leq b\}}\ind\Big\{{a}\vee(b+2-2^\ell)\!\leq\! \wh{a}_{h}\!\leq\! b+1-2^{\ell-1}\Big\}.
\end{align}
We adopt the convention that maximum over an empty set is zero. The random variables defined above satisfy the following lemma, which will be proved in Section \ref{sec:aux} in the supplement.
\begin{lemma} \label{lem:er12}
There exists a universal constant $C>0$, such that for any integer $f\geq 0$,
\begin{align*}
\sum_{j=1}^{k}\sum_{\{\ell\geq 1:a_{j-1}+2^{\ell-1}\leq a_j\}}\mathbb{E}\delta_+(a_{j-1},a_j,\ell+f)\xi_+(a_{j-1},a_j,\ell) &\leq C\sigma^2k\log\log(16n/k),\\
\sum_{j=1}^{k}\sum_{\{\ell\geq 1:a_{j-1}\leq a_j-2^{\ell-1}\}}\mathbb{E}\delta_-(a_{j-1},a_j,\ell+f)\xi_-(a_{j-1},a_j,\ell) &\leq C\sigma^2k\log\log(16n/k).
\end{align*}
\end{lemma}
We also need the following lemma to facilitate the proof. Its proof will also be given in Section \ref{sec:aux} in the supplement.
\begin{lemma}\label{lem:simple-lil}
There exists a universal constant $C>0$, such that
\begin{eqnarray*}
\sum_{j=1}^k\mathbb{E}\max_{a_{j-1}<a\leq a_j}(a-a_{j-1})\wb{Z}^2_{(a_{j-1}:a]} &\leq& C\sigma^2k\log\log(16n/k), \\
\sum_{j=1}^k\mathbb{E}\max_{a_{j-1}\leq a\leq a_j}(a_j-a)\wb{Z}_{(a:a_j]}^2 &\leq& C\sigma^2k\log\log(16n/k).
\end{eqnarray*}
\end{lemma}

The proof of Theorem \ref{thm:upper2} starts with the basic inequality $\|X-\wh{\theta}\|^2\leq \|X-\theta^{(k)}\|^2$, a direct consequence of the definition of $\wh{\theta}$. Since
\begin{eqnarray}
\label{eq:qexpand-1} \|X-\wh{\theta}\|^2 &=& \|X-\theta^*\|^2 + \|\theta^*-\wh{\theta}\|^2 + 2\iprod{X-\theta^*}{\theta^*-\wh{\theta}}, \\
\label{eq:qexpand-2} \|X-\theta^{(k)}\|^2 &=& \|X-\theta^*\|^2 + \|\theta^*-\theta^{(k)}\|^2 + 2\iprod{X-\theta^*}{\theta^*-\theta^{(k)}},
\end{eqnarray}
we have
\begin{equation}
\|\wh{\theta}-\theta^*\|^2\leq \|\theta^{(k)}-\theta^*\|^2 + 2\iprod{X-\theta^*}{\wh{\theta}-\theta^{(k)}}. \label{eq:basic-k-piece}
\end{equation}
For each $j$, define $h_j=\max\{h:\wh{a}_h\leq a_j\}$. It is easy to see that $\wh{a}_{h_j}\leq a_{j-1}$ if and only if $\wh{\theta}$ is a constant in the interval $(a_{j-1}:a_j]$.
Then, the inner product term above is
\begin{eqnarray}
\nonumber && 2\iprod{X-\theta^*}{\wh{\theta}-\theta^{(k)}} \\
\nonumber &=& 2\sum_{j=1}^k\ind\{\wh{a}_{h_j}\leq a_{j-1}\}\sum_{i\in(a_{j-1}:a_j]}(X_i-\theta_i^*)(\wh{\theta}_i-\theta_i^{(k)}) \\
\nonumber && + 2\sum_{j=1}^k\ind\{\wh{a}_{h_j}> a_{j-1}\}\sum_{i\in(a_{j-1}:a_j]}(X_i-\theta_i^*)(\wh{\theta}_i-\theta_i^{(k)}) \\
\nonumber &=& 2\sum_{\{j\in[k]:\wh{a}_{h_j}\leq a_{j-1}\}}(a_j-a_{j-1})\wb{Z}_{(a_{j-1}:a_j]}(\wh{\theta}_{a_j}-\theta_{a_j}^{(k)}) \\
\nonumber && + 2\sum_{\{j\in[k]:\wh{a}_{h_j}> a_{j-1}\}}(a_j-\wh{a}_{h_j})\wb{Z}_{(\wh{a}_{h_j}:a_j]}(\wh{\theta}_{a_j}-\theta_{a_j}^{(k)}) \\
\nonumber && + 2\sum_{\{j\in[k]:\wh{a}_{h_j}> a_{j-1}\}}(\wh{a}_{h_{j-1}+1}-a_{j-1})\wb{Z}_{(a_{j-1}:\wh{a}_{h_{j-1}+1}]}(\wh{\theta}_{\wh{a}_{h_{j-1}+1}}-\theta_{\wh{a}_{h_{j-1}+1}}^{(k)}) \\
\label{eq:the-hardest-4th-term} && + 2\sum_{\{j\in[k]:\wh{a}_{h_j}> a_{j-1}\}}\sum_{\{h: (\wh{a}_{h-1}:\wh{a}_h]\subset(a_{j-1}:a_j]\}}(\wh{a}_h-\wh{a}_{h-1})\wb{Z}_{(\wh{a}_{h-1}:\wh{a}_h]}(\wh{\theta}_{\wh{a}_h}-\theta_{\wh{a}_h}^{(k)}).
\end{eqnarray}
The summation over an empty set is understood as zero.
The inner product $2\iprod{X-\theta^*}{\wh{\theta}-\theta^{(k)}}$ is bounded by four terms. For the first three terms, we can use Cauchy-Schwarz and, for any $\eta\in(0,1)$, get the bound
\begin{eqnarray}
\label{eq:bound-first-three-terms} && 3\eta\|\wh{\theta}-\theta^{(k)}\|^2 + \eta^{-1}\sum_{j=1}^k(a_j-a_{j-1})\wb{Z}_{(a_{j-1}:a_j]}^2 \\
\nonumber &&  + \eta^{-1}\sum_{j=1}^k(a_j-\wh{a}_{h_j})\wb{Z}_{(\wh{a}_{h_j}:a_j]}^2  + \eta^{-1}\sum_{j=1}^k(\wh{a}_{h_{j-1}+1}-a_{j-1})\wb{Z}_{(a_{j-1}:\wh{a}_{h_{j-1}+1}]}^2.
\end{eqnarray}

Bounding the fourth term (\ref{eq:the-hardest-4th-term}) is involved. We need some extra notations. For each $h$ such that $(\wh{a}_{h-1}:\wh{a}_h]\subset(a_{j-1}:a_j]$, define
$$a_{h-1}'=\floor{\frac{a_{j-1}+\wh{a}_h}{2}},\quad b_{h-1}'=\wh{a}_{h-1}\wedge a_{h-1}',$$
$$a_h''=\ceil{\frac{\wh{a}_{h-1}+a_j}{2}},\quad b_h''=\wh{a}_h\vee a_h''.$$
Given any integers $1\leq a\leq b\leq n$, define the random variables
\begin{eqnarray*}
\wb{Z}_{(a:b]}' &=& \max_{b'\in(a:b]}\frac{b'-a}{b-a}\left|\wb{Z}_{(a:b']}\right|, \\
\wb{Z}_{(a:b]}'' &=& \max_{a'\in[a:b)}\frac{b-a'}{b-a}\left|\wb{Z}_{(a':b]}\right|.
\end{eqnarray*}
By Lemma \ref{lem:cun-hui}, we have $\wh{\theta}_{\wh{a}_h}\leq \wb{X}_{(\wh{a}_h:b_h'']}$. Since $\wb{X}_{(\wh{a}_{h-1}:b_h'']}$ is a weighted average of $\wh{\theta}_{\wh{a}_h}=\wb{X}_{(\wh{a}_{h-1}:\wh{a}_h]}$ and $\wb{X}_{(\wh{a}_h:b_h'']}$, we get $\wh{\theta}_{\wh{a}_h}\leq \wb{X}_{(\wh{a}_{h-1}:b_h'']}$. With this bound, we have
\begin{eqnarray*}
\wh{\theta}_{\wh{a}_h}-\theta_{\wh{a}_h}^{(k)} &\leq& \wb{X}_{(\wh{a}_{h-1}:b_h'']} - \wb{\theta}^*_{(\wh{a}_{h-1}:b_h'']}  + \wb{\theta}^*_{(\wh{a}_{h-1}:b_h'']} - \theta_{\wh{a}_h}^{(k)} \\
&=& \wb{Z}_{(\wh{a}_{h-1}:b_h'']}  + \wb{\theta}^*_{(\wh{a}_{h-1}:b_h'']} - \theta_{\wh{a}_h}^{(k)} \\
&=& \frac{a_j-\wh{a}_{h-1}}{b_h''-\wh{a}_{h-1}}\wb{Z}_{(\wh{a}_{h-1}:a_j]} - \frac{a_j-b_h''}{b_h''-\wh{a}_{h-1}}\wb{Z}_{(b_h'':a_j]} + \wb{\theta}^*_{(\wh{a}_{h-1}:b_h'']} - \theta_{\wh{a}_h}^{(k)} \\
&\leq& 4\wb{Z}_{(\wh{a}_{h-1}:a_j]}'' + |\wb{\theta}^*_{(\wh{a}_{h-1}:b_h'']} - \theta_{\wh{a}_h}^{(k)}|.
\end{eqnarray*}
A symmetric argument also gives
$$\wh{\theta}_{\wh{a}_h}-\theta_{\wh{a}_h}^{(k)} \geq -4\wb{Z}'_{(a_{j-1}:\wh{a}_h]}-|\wb{\theta}^*_{(b_{h-1}':\wh{a}_h]} - \theta_{\wh{a}_h}^{(k)}|.$$
Therefore, we have the inequality
\begin{equation}
|\wh{\theta}_{\wh{a}_h}-\theta_{\wh{a}_h}^{(k)}| \leq 4(\wb{Z}'_{(a_{j-1}:\wh{a}_h]}\vee\wb{Z}_{(\wh{a}_{h-1}:a_j]}'') + |\wb{\theta}^*_{(\wh{a}_{h-1}:b_h'']} - \theta_{\wh{a}_h}^{(k)}|\vee|\wb{\theta}^*_{(b_{h-1}':\wh{a}_h]} - \theta_{\wh{a}_h}^{(k)}|. \label{eq:cunhui's-magic}
\end{equation}
Since (\ref{eq:the-hardest-4th-term}) is a sum of $k$ terms, we can bound each of the term separately. 

For each $j\in[k]$, {recalling $\wb{Z}_{(\wh{a}_{h-1}:\wh{a}_h]}=\wh{\theta}_{\wh{a}_h}-\wb{\theta}^*_{(\wh{a}_{h-1}:\wh{a}_h]}$,}
\begin{eqnarray}
\nonumber && \sum_{\{h: (\wh{a}_{h-1}:\wh{a}_h]\subset(a_{j-1}:a_j]\}}(\wh{a}_h-\wh{a}_{h-1})\wb{Z}_{(\wh{a}_{h-1}:\wh{a}_h]}(\wh{\theta}_{\wh{a}_h}-\theta_{\wh{a}_h}^{(k)}) \\
\nonumber &{=}& \sum_{\{h: (\wh{a}_{h-1}:\wh{a}_h]\subset(a_{j-1}:a_j]\}}(\wh{a}_h-\wh{a}_{h-1})(\wh{\theta}_{\wh{a}_h}-\theta_{\wh{a}_h}^{(k)})^2 \\
\nonumber && + \sum_{\{h: (\wh{a}_{h-1}:\wh{a}_h]\subset(a_{j-1}:a_j]\}}(\wh{a}_h-\wh{a}_{h-1})(\theta_{\wh{a}_h}^{(k)}-\wb{\theta}^*_{(\wh{a}_{h-1}:\wh{a}_h]})(\wh{\theta}_{\wh{a}_h}-\theta_{\wh{a}_h}^{(k)}) \\
\nonumber &\leq& 32\sum_{\{h: (\wh{a}_{h-1}:\wh{a}_h]\subset(a_{j-1}:a_j]\}}(\wh{a}_h-\wh{a}_{h-1})|\wb{Z}'_{(a_{j-1}:\wh{a}_h]}|^2 \\
\nonumber && + 32\sum_{\{h: (\wh{a}_{h-1}:\wh{a}_h]\subset(a_{j-1}:a_j]\}}(\wh{a}_h-\wh{a}_{h-1})|\wb{Z}_{(\wh{a}_{h-1}:a_j]}''|^2 \\
\label{eq:hardest-one} && + 2\sum_{\{h: (\wh{a}_{h-1}:\wh{a}_h]\subset(a_{j-1}:a_j]\}}(\wh{a}_h-\wh{a}_{h-1})|\wb{\theta}^*_{(\wh{a}_{h-1}:b_h'']} - \theta_{\wh{a}_h}^{(k)}|^2 \\
\label{eq:hardest-two} && + 2\sum_{\{h: (\wh{a}_{h-1}:\wh{a}_h]\subset(a_{j-1}:a_j]\}}(\wh{a}_h-\wh{a}_{h-1})|\wb{\theta}^*_{(b_{h-1}':\wh{a}_h]} - \theta_{\wh{a}_h}^{(k)}|^2 \\
\nonumber && + \frac{\eta}{2}\|\wh{\theta}-\theta^{(k)}\|^2_{(a_{j-1}:a_j]} + \frac{1}{2\eta}\|\theta^{(k)}-\theta^*\|^2_{(a_{j-1}:a_j]}.
\end{eqnarray}
Among the terms in the above bound, we need to further analyze (\ref{eq:hardest-one}) and (\ref{eq:hardest-two}). We have
\begin{eqnarray*}
&& \sum_{\{h: (\wh{a}_{h-1}:\wh{a}_h]\subset(a_{j-1}:a_j]\}}(\wh{a}_h-\wh{a}_{h-1})|\wb{\theta}^*_{(\wh{a}_{h-1}:b_h'']} - \theta_{\wh{a}_h}^{(k)}|^2 \\
&\leq& \sum_{\{h: (\wh{a}_{h-1}:\wh{a}_h]\subset(a_{j-1}:a_j]\}}(\wh{a}_h-\wh{a}_{h-1})|\wb{\theta}^*_{(\wh{a}_{h-1}:\wh{a}_h]} - \theta_{\wh{a}_h}^{(k)}|^2 \\
&& + \sum_{\{h: (\wh{a}_{h-1}:\wh{a}_h]\subset(a_{j-1}:a_j]\}}(\wh{a}_h-\wh{a}_{h-1})|\wb{\theta}^*_{(\wh{a}_{h-1}:a_h'']} - \theta_{\wh{a}_h}^{(k)}|^2 \\
&\leq& \sum_{\{h: (\wh{a}_{h-1}:\wh{a}_h]\subset(a_{j-1}:a_j]\}}\|\theta^{(k)}-\theta^*\|^2_{(\wh{a}_{h-1}:\wh{a}_h]} \\
&& + \sum_{\{h: (\wh{a}_{h-1}:\wh{a}_h]\subset(a_{j-1}:a_j]\}}\frac{\wh{a}_h-\wh{a}_{h-1}}{a_h''-\wh{a}_{h-1}}\sum_{i\in(\wh{a}_{h-1}:a_h'']}(\theta^*_i-\theta^{(k)}_i)^2 \\
&\leq& \|\theta^{(k)}-\theta^*\|^2_{(a_{j-1}:a_j]} \\
&& + \sum_{i\in(a_{j-1}:a_j]}(\theta^*_i-\theta^{(k)}_i)^2\sum_{\{h: (\wh{a}_{h-1}:\wh{a}_h]\subset(a_{j-1}:a_j]\}}\frac{\wh{a}_h-\wh{a}_{h-1}}{a_h''-\wh{a}_{h-1}}\ind\left\{\wh{a}_{h-1}<i\leq a_h''\right\},
\end{eqnarray*}
where
\begin{eqnarray}
\nonumber && \sum_{\{h: (\wh{a}_{h-1}:\wh{a}_h]\subset(a_{j-1}:a_j]\}}\frac{\wh{a}_h-\wh{a}_{h-1}}{a_h''-\wh{a}_{h-1}}\ind\left\{\wh{a}_{h-1}<i\leq a_h''\right\} \\
\nonumber &\leq& 2\sum_{\{h: (\wh{a}_{h-1}:\wh{a}_h]\subset(a_{j-1}:a_j]\}}\frac{\wh{a}_h-\wh{a}_{h-1}}{a_j-\wh{a}_{h-1}}\ind\left\{\wh{a}_{h-1}<i\leq a_h''\right\} \\
\nonumber &\leq& 2\sum_{\{h: (\wh{a}_{h-1}:\wh{a}_h]\subset(a_{j-1}:a_j]\}}\frac{\wh{a}_h-\wh{a}_{h-1}}{a_j-i+1}\ind\left\{\wh{a}_{h-1}<i\leq a_h''\right\} \\
\label{eq:crazy-magic} &\leq& 2\sum_{\{h: (\wh{a}_{h-1}:\wh{a}_h]\subset(a_{j-1}:a_j]\}}\frac{\wh{a}_h-\wh{a}_{h-1}}{a_j-i+1}\ind\left\{a_j-\wh{a}_{h-1}\leq 2(a_j-i+1)\right\} \\
\nonumber &\leq& 2\max\left\{\frac{a_j-\wh{a}_t}{a_j-i+1}: a_j-\wh{a}_t\leq 2(a_j-i+1)\right\} \\
\nonumber &\leq& 4.
\end{eqnarray}
The inequality (\ref{eq:crazy-magic}) above is due to the fact that $i\leq a_h''\leq \frac{\wh{a}_{h-1}+a_j+1}{2}$ implies
$$2(a_j-i+1)\geq 2a_j+2-(\wh{a}_{h-1}+a_j+1)=a_j-\wh{a}_{h-1}+1.$$
Therefore, we obtain
$$ \sum_{\{h: (\wh{a}_{h-1}:\wh{a}_h]\subset(a_{j-1}:a_j]\}}(\wh{a}_h-\wh{a}_{h-1})|\wb{\theta}^*_{(\wh{a}_{h-1}:b_h'']} - \theta_{\wh{a}_h}^{(k)}|^2\leq 5\|\theta^{(k)}-\theta^*\|^2_{(a_{j-1}:a_j]},$$
which leads to a bound for (\ref{eq:hardest-one}). A symmetric argument gives
$$\sum_{\{h: (\wh{a}_{h-1}:\wh{a}_h]\subset(a_{j-1}:a_j]\}}(\wh{a}_h-\wh{a}_{h-1})|\wb{\theta}^*_{(b_{h-1}':\wh{a}_h]} - \theta_{\wh{a}_h}^{(k)}|^2\leq 5\|\theta^{(k)}-\theta^*\|^2_{(a_{j-1}:a_j]},$$
which leads to a bound for (\ref{eq:hardest-two}). Summing over $j\in[k]$, a bound for (\ref{eq:the-hardest-4th-term}) is given by
\begin{eqnarray*}
&& 64\sum_{j=1}^k\sum_{\{h: (\wh{a}_{h-1}:\wh{a}_h]\subset(a_{j-1}:a_j]\}}(\wh{a}_h-\wh{a}_{h-1})|\wb{Z}'_{(a_{j-1}:\wh{a}_h]}|^2 \\
&& + 64\sum_{j=1}^k\sum_{\{h: (\wh{a}_{h-1}:\wh{a}_h]\subset(a_{j-1}:a_j]\}}(\wh{a}_h-\wh{a}_{h-1})|\wb{Z}_{(\wh{a}_{h-1}:a_j]}''|^2 \\
&& + (40+\eta^{-1})\|\theta^{(k)}-\theta^*\|^2 + \eta\|\wh{\theta}-\theta^{(k)}\|^2.
\end{eqnarray*}
We can plug the above bound and (\ref{eq:bound-first-three-terms}) into (\ref{eq:basic-k-piece}), and we get
\begin{eqnarray*}
\|\wh{\theta}-\theta^*\|^2 &\leq& (41+\eta^{-1})\|\theta^{(k)}-\theta^*\|^2 + 4\eta\|\wh{\theta}-\theta^{(k)}\|^2 \\
&& + \eta^{-1}\sum_{j=1}^k(a_j-a_{j-1})\wb{Z}_{(a_{j-1}:a_j]}^2 + \eta^{-1}\sum_{j=1}^k(a_j-\wh{a}_{h_j})\wb{Z}_{(\wh{a}_{h_j}:a_j]}^2 \\
&&  + \eta^{-1}\sum_{j=1}^k(\wh{a}_{h_{j-1}+1}-a_{j-1})\wb{Z}_{(a_{j-1}:\wh{a}_{h_{j-1}+1}]}^2 \\
&& + 64\sum_{j=1}^k\sum_{\{h: (\wh{a}_{h-1}:\wh{a}_h]\subset(a_{j-1}:a_j]\}}(\wh{a}_h-\wh{a}_{h-1})|\wb{Z}'_{(a_{j-1}:\wh{a}_h]}|^2 \\
&& + 64\sum_{j=1}^k\sum_{\{h: (\wh{a}_{h-1}:\wh{a}_h]\subset(a_{j-1}:a_j]\}}(\wh{a}_h-\wh{a}_{h-1})|\wb{Z}_{(\wh{a}_{h-1}:a_j]}''|^2.
\end{eqnarray*}
Use the inequality $\|\wh{\theta}-\theta^{(k)}\|^2\leq 2\|\wh{\theta}-\theta^*\|^2 + 2\|\theta^{(k)}-\theta^*\|^2$, set $\eta=1/16$, and some rearrangement of the above bound gives
\begin{eqnarray*}
\|\wh{\theta}-\theta^*\|^2 &\leq& C\|\theta^{(k)}-\theta^*\|^2 + C\sum_{j=1}^k(a_j-a_{j-1})\wb{Z}_{(a_{j-1}:a_j]}^2 \\
&& + C\sum_{j=1}^k(\wh{a}_{h_{j-1}+1}-a_{j-1})\wb{Z}_{(a_{j-1}:\wh{a}_{h_{j-1}+1}]}^2 + C\sum_{j=1}^k(a_j-\wh{a}_{h_j})\wb{Z}_{(\wh{a}_{h_j}:a_j]}^2 \\
&& + C\sum_{j=1}^k\sum_{\{h: (\wh{a}_{h-1}:\wh{a}_h]\subset(a_{j-1}:a_j]\}}(\wh{a}_h-\wh{a}_{h-1})|\wb{Z}'_{(a_{j-1}:\wh{a}_h]}|^2 \\
&& + C\sum_{j=1}^k\sum_{\{h: (\wh{a}_{h-1}:\wh{a}_h]\subset(a_{j-1}:a_j]\}}(\wh{a}_h-\wh{a}_{h-1})|\wb{Z}_{(\wh{a}_{h-1}:a_j]}''|^2,
\end{eqnarray*}
where $C>0$ is some universal constant. Note that
$$\mathbb{E}\sum_{j=1}^k(a_j-a_{j-1})\wb{Z}_{(a_{j-1}:a_j]}^2=k\sigma^2\lesssim \sigma^2k\log\log(16n/k).$$
By Lemma \ref{lem:er12}, we have
\begin{align*}
& \mathbb{E}\sum_{j=1}^k\sum_{\{h: (\wh{a}_{h-1}:\wh{a}_h]\subset(a_{j-1}:a_j]\}}(\wh{a}_h-\wh{a}_{h-1})|\wb{Z}'_{(a_{j-1}:\wh{a}_h]}|^2 \\
=& \mathbb{E}\sum_{j=1}^k\sum_{\{h: (\wh{a}_{h-1}:\wh{a}_h]\subset(a_{j-1}:a_j]\}} (\wh{a}_h-\wh{a}_{h-1})\max_{b\in(a_{j-1}:\wh{a}_h]}\frac{(b-a_{j-1})^2}{(\wh{a}_h-a_{j-1})^2}|\wb{Z}_{(a_{j-1}:b]}|^2 \\
\leq& \mathbb{E}\sum_{j=1}^k\sum_{\{h: (\wh{a}_{h-1}:\wh{a}_h]\subset(a_{j-1}:a_j]\}}\sum_{\{\ell\geq 1:a_{j-1}+2^{\ell-1}\leq a_j\}}\ind\{a_{j-1}+2^{\ell-1}\leq\wh{a}_h<a_{j-1}+2^{\ell}\}\\
 & \quad\quad (\wh{a}_h-\wh{a}_{h-1})\max_{b\in(a_{j-1}:\wh{a}_h]}\frac{(b-a_{j-1})^2}{(\wh{a}_h-a_{j-1})^2}|\wb{Z}_{(a_{j-1}:b]}|^2 \\
\leq& \mathbb{E}\sum_{j=1}^k\sum_{\{h: (\wh{a}_{h-1}:\wh{a}_h]\subset(a_{j-1}:a_j]\}}\sum_{\{\ell\geq 1:a_{j-1}+2^{\ell-1}\leq a_j\}}\ind\{a_{j-1}+2^{\ell-1}\leq\wh{a}_h<a_{j-1}+2^{\ell}\}\\
& \quad\quad (\wh{a}_h-\wh{a}_{h-1})\max_{b\in(a_{j-1}:a_j\wedge(a_{j-1}+2^{\ell}-1)]}\frac{(b-a_{j-1})^2}{(2^{\ell-1})^2}|\wb{Z}_{(a_{j-1}:b]}|^2 \\
\leq& \mathbb{E}\sum_{j=1}^k\sum_{\{\ell\geq 1:a_{j-1}+2^{\ell-1}\leq a_j\}}\delta_+(a_{j-1},a_j,\ell)2^{\ell}\max_{b\in(a_{j-1}:a_j\wedge(a_{j-1}+2^{\ell}-1)]}\frac{(b-a_{j-1})^2}{(2^{\ell-1})^2}|\wb{Z}_{(a_{j-1}:b]}|^2 \\\nonumber \leq& 4\mathbb{E}\sum_{j=1}^k\sum_{\{\ell\geq 1:a_{j-1}+2^{\ell-1}\leq a_j\}}\delta_+(a_{j-1},a_j,\ell)2^{-\ell}\max_{b\in(a_{j-1}:a_j\wedge(a_{j-1}+2^{\ell}-1)]}(b-a_{j-1})^2|\wb{Z}_{(a_{j-1}:b]}|^2 \\
\leq& 4\mathbb{E}\sum_{j=1}^k\sum_{\{\ell\geq 1:a_{j-1}+2^{\ell-1}\leq a_j\}}\delta_+(a_{j-1},a_j,\ell)2^{-\ell}\sum_{f=1}^{\ell}2^{f}\xi_+(a_{j-1},a_j,f) \\
\nonumber \leq& 4\mathbb{E}\sum_{j=1}^k\sum_{f\geq 0}2^{-f}\sum_{\{\ell\geq 1:a_{j-1}+2^{\ell-1}\leq a_j\}}\delta_+(a_{j-1},a_j,\ell+f)\xi_+(a_{j-1},a_j,\ell),
\end{align*}
which leads to the conclusion
\begin{align}
  \notag &\mathbb{E}\sum_{j=1}^k\sum_{\{h: (\wh{a}_{h-1}:\wh{a}_h]\subset(a_{j-1}:a_j]\}}(\wh{a}_h-\wh{a}_{h-1})|\wb{Z}'_{(a_{j-1}:\wh{a}_h]}|^2\\
\label{eq:radiohead}=& 4\sum_{f\geq 0}2^{-f}\sum_{j=1}^k\sum_{\{\ell\geq 1:a_{j-1}+2^{\ell-1}\leq a_j\}}\mathbb{E}\delta_+(a_{j-1},a_j,\ell+f)\xi_+(a_{j-1},a_j,\ell) \\
\nonumber \lesssim& \sigma^2k\log\log(16n/k).
\end{align}
Similarly, we also have
$$\mathbb{E}\sum_{j=1}^k\sum_{\{h: (\wh{a}_{h-1}:\wh{a}_h]\subset(a_{j-1}:a_j]\}}(\wh{a}_h-\wh{a}_{h-1})|\wb{Z}_{(\wh{a}_{h-1}:a_j]}''|^2\lesssim \sigma^2k\log\log(16n/k).$$
Finally, by Lemma \ref{lem:simple-lil}, we have
$$\mathbb{E}\sum_{j=1}^k(\wh{a}_{h_{j-1}+1}-a_{j-1})\wb{Z}_{(a_{j-1}:\wh{a}_{h_{j-1}+1}]}^2 + \mathbb{E}\sum_{j=1}^k(a_j-\wh{a}_{h_j})\wb{Z}_{(\wh{a}_{h_j}:a_j]}^2\lesssim \sigma^2k\log\log(16n/k).$$
Combining the above bounds, we obtain the desired oracle inequality as long as $k\geq 2$.

To complete the proof, we also give the argument for $k=1$. In this case $\wh{\theta}_i=\wb{X}$ and $\theta^{(1)}_i=\wb{\theta}^*$ for all $i\in[n]$. Therefore, $\mathbb{E}\|\wh{\theta}-\theta^{(1)}\|^2=\sigma^2$, which leads to $\mathbb{E}\|\wh{\theta}-\theta^*\|^2\leq 2\|\theta^{(1)}-\theta^*\|^2+2\sigma^2.$
\end{proof}

\begin{proof}[Proof of Theorem \ref{thm:adaptive-oracle}]
We use the same notations in the proof of Theorem \ref{thm:upper2}, except that $\wh{\theta}$ is now for $\wh{\theta}(\Theta_{\wh{k}}^{\uparrow})$ and $\wh{A}_{\wh{k}}=\{\wh{a}_h\}$. By the definition of $\wh{\theta}$, we have
$$\|X-\wh{\theta}\|^2+\text{pen}_{\tau}(\wh{k})\leq \|X-\wh{\theta}(\Theta_k^{\uparrow})\|^2+\text{pen}_{\tau}(k)\leq \|X-\theta^{(k)}\|^2+\text{pen}_{\tau}(k).$$
By (\ref{eq:qexpand-1}) and (\ref{eq:qexpand-2}), we obtain the following inequality
\begin{equation}
\|\wh{\theta}-\theta^*\|^2+\text{pen}_{\tau}(\wh{k})\leq \|\theta^{(k)}-\theta^*\|^2 + 2\iprod{X-\theta^*}{\wh{\theta}-\theta^{(k)}}+\text{pen}_{\tau}(k). \label{eq:basic-k-piece-pen}
\end{equation}
After bounding $2\iprod{X-\theta^*}{\wh{\theta}-\theta^{(k)}}$ by the same argument in the proof of Theorem \ref{thm:upper2}, we obtain
\begin{eqnarray}
\nonumber && \|\wh{\theta}-\theta^*\|^2 + 2\text{pen}_{\tau}(\wh{k}) - 2\text{pen}_{\tau}(k) \\
\label{eq:used-in-the-next-proof} &\leq& C\|\theta^{(k)}-\theta^*\|^2 + C\sum_{j=1}^k(a_j-a_{j-1})\wb{Z}_{(a_{j-1}:a_j]}^2 \\
\label{eq:also-need-label} && + C\sum_{j=1}^k(\wh{a}_{h_{j-1}+1}-a_{j-1})\wb{Z}_{(a_{j-1}:\wh{a}_{h_{j-1}+1}]}^2 + C\sum_{j=1}^k(a_j-\wh{a}_{h_j})\wb{Z}_{(\wh{a}_{h_j}:a_j]}^2 \\
\label{eq:error-term-pen1} && + C\sum_{j=1}^k\sum_{\{h: (\wh{a}_{h-1}:\wh{a}_h]\subset(a_{j-1}:a_j]\}}(\wh{a}_h-\wh{a}_{h-1})|\wb{Z}'_{(a_{j-1}:\wh{a}_h]}|^2 \\
\label{eq:error-term-pen2} && + C\sum_{j=1}^k\sum_{\{h: (\wh{a}_{h-1}:\wh{a}_h]\subset(a_{j-1}:a_j]\}}(\wh{a}_h-\wh{a}_{h-1})|\wb{Z}_{(\wh{a}_{h-1}:a_j]}''|^2,
\end{eqnarray}
where $C>0$ is some universal constant. Take expectation on both sides of the inequality, and then the right hand side can all be bounded similarly as in the proof of Theorem \ref{thm:upper2} except for (\ref{eq:error-term-pen1}) and (\ref{eq:error-term-pen2}). In fact, (\ref{eq:error-term-pen1}) and (\ref{eq:error-term-pen2}) can be bounded by the same argument that leads to (\ref{eq:radiohead}). The only difference is that now the $\{\wh{a}_h\}$ in the definitions of $\delta_+(a_{j-1},a_j,\ell)$ and $\delta_-(a_{j-1},a_j,\ell)$ are from $\wh{A}_{\wh{k}}$ instead of $\wh{A}_k$.
Therefore, we need the following lemma, whose proof will be given in Section \ref{sec:aux} in the supplement.
\begin{lemma} \label{lem:er12-all-k}
There exists a universal constant $C>0$, such that
\begin{align*}
 &\max\Big\{\sum_{f\geq 0}2^{-f}\sum_{j=1}^{k}\sum_{\{\ell\geq 1:a_{j-1}+2^{\ell-1}\leq a_j\}}\mathbb{E}\delta_+(a_{j-1},a_j,\ell+f)\xi_+(a_{j-1},a_j,\ell), \\
&\quad\quad\quad \sum_{f\geq 0}2^{-f}\sum_{j=1}^{k}\sum_{\{\ell\geq 1:a_{j-1}\leq a_j-2^{\ell-1}\}}\mathbb{E}\delta_-(a_{j-1},a_j,\ell+f)\xi_-(a_{j-1},a_j,\ell) \Big\}\\
\leq& C\left\{\sigma^2k\log\log\left(\frac{16n}{k}\right)+\sigma^2\mathbb{E}\wh{k}\log\log\left(\frac{16n}{\wh{k}}\right)\right\},
\end{align*}
where the $\{\wh{a}_h\}$ in the definitions of $\delta_+(a_{j-1},a_j,\ell)$ and $\delta_-(a_{j-1},a_j,\ell)$ are from $\wh{A}_{\wh{k}}$ instead of $\wh{A}_k$.
\end{lemma}

Then, for some (possibly different) universal constant $C>0$, we have
\begin{eqnarray*}
&& \mathbb{E}\|\wh{\theta}-\theta^*\|^2 + 2\mathbb{E}\text{pen}_{\tau}(\wh{k}) \\
&\leq& C\|\theta^{(k)}-\theta^*\|^2 + 2\text{pen}_{\tau}(k) \\
&& + C\left\{\sigma^2k\log\log\left(\frac{16n}{k}\right)+\sigma^2\mathbb{E}\wh{k}\log\log\left(\frac{16n}{\wh{k}}\right)\right\}.
\end{eqnarray*}
Choosing $\tau=C_1\sigma^2$ with a sufficiently large constant $C_1>0$, we get
$$\mathbb{E}\|\wh{\theta}-\theta^*\|^2\lesssim \|\theta^{(k)}-\theta^*\|^2+\sigma^2k\log\log\left(\frac{16n}{k}\right),$$
which is the desired results for $k\geq 2$.

To complete the proof, we also need to give the analysis for $k=1$. It is easy to see that in this case the bounds in Lemma \ref{lem:simple-lil} and Lemma \ref{lem:er12-all-k} can be improved to $C\sigma^2$ and $C\sigma^2 + C\sigma^2\mathbb{E}\wh{k}\log\log\left(\frac{16n}{\wh{k}}\right)\ind\{\wh{k}\geq 2\}$, respectively. Therefore, the choice $\tau=C_1\sigma^2$ with a large constant $C_1>0$ leads to 
\begin{equation}
\mathbb{E}\|\wh{\theta}-\theta^*\|^2\lesssim \|\theta^{(1)}-\theta^*\|^2 + \sigma^2.\label{eq:pf-k=1-later}
\end{equation}
The proof is thus complete.
\end{proof}

\subsection{Proofs of lower bounds}\label{sec:proof-lower}

This section is devoted to proving the lower bounds in Section \ref{sec:upper}. 

\begin{proof}[Proof of Proposition \ref{prop:lower-LSE}]
Without loss of generality, consider the case when $n/k$ is an integer. Then, $[n]=\bigcup_{j=1}^{k}\mathcal{C}_j$, where $\mathcal{C}_j$ is the $j$th consecutive interval with cardinality $n/k$. Then, we take $\theta^*\in\Theta_{k}^{\uparrow}$ with $\theta_i^*=\mu_j$ if $i\in \mathcal{C}_j$. Use the notation $\mathcal{H}_n=\{\theta\in\mathbb{R}^n:\theta_1\leq ...\leq \theta_n\}$. Then, as long as $\mu_1,...,\mu_k$ are sufficiently separated,
$$\min_{\theta\in\mathcal{H}_n}\sum_{i=1}^n(X_i-\theta_i)^2=\sum_{j=1}^{k}\min_{\theta\in\mathcal{H}_{n/k}}\sum_{i\in\mathcal{C}_j}(X_i-\theta_i)^2,$$
with high probability. This high-probability event is denoted as $E$. We take $\mu_j=\kappa j$ for some $\kappa>0$. Then, as $\kappa\rightarrow\infty$, $\mathbb{P}(E^c)$ converges to $0$. In other words, $\mathbb{P}(E^c)$ is arbitrarily small for sufficiently large $\kappa$. We have
$$\mathbb{E}\|\hat{\theta}-\theta^*\|^2\geq \sum_{j=1}^{k}\mathbb{E}\|\hat{\theta}_{\mathcal{C}_j}-\theta^*_{\mathcal{C}_j}\|^2-\mathbb{E}\|\hat{\theta}-\theta^*\|^2\ind_{E^c}.$$
Since $\mathbb{E}\|\hat{\theta}-\theta^*\|^2\ind_{E^c}\leq \sqrt{\mathbb{E}\|\hat{\theta}-\theta^*\|^4}\sqrt{\mathbb{P}(E^c)}$ is arbitrarily small for sufficiently large $\kappa$, the term $\mathbb{E}\|\hat{\theta}-\theta^*\|^2\ind_{E^c}$ can be neglected. It is sufficient to give a lower bound for $\sum_{j=1}^{k}\mathbb{E}\|\hat{\theta}_{\mathcal{C}_j}-\theta^*_{\mathcal{C}_j}\|^2$. Note that
$$\sum_{j=1}^{k}\mathbb{E}\|\hat{\theta}_{\mathcal{C}_j}-\theta^*_{\mathcal{C}_j}\|^2=\sum_{j=1}^{k}\mathbb{E}\|\Pi_{\mathcal{H}_{n/k}}Z_{\mathcal{C}_j}\|^2,$$
where $\Pi_{\mathcal{H}_{n/k}}$ is the projection operator onto the space $\mathcal{H}_{n/k}$. By \cite{amelunxen2014living}, $\|\Pi_{\mathcal{H}_{n/k}}Z_{\mathcal{C}_j}\|^2\geq C\log(en/k)$, leading to the desired result.
\end{proof}

We continue to state the proofs of other results. The main tool we will use is Fano's lemma. For any probability measures $\mathbb{P},\mathbb{Q}$, define the Kullback-Leibler divergence to be 
\[
D(\mathbb{P}||\mathbb{Q})=\int\Bigl(\log\frac{d\mathbb{P}}{d\mathbb{Q}}\Bigr)d\mathbb{P}. 
\]
The Fano's lemma is stated as follows. See \cite{ibragimov2013statistical} and \cite{tsybakov2009introduction} for references.

\begin{proposition} \label{prop:fano}
Let $(\Theta,\rho)$ be a metric space and $\{\mathbb{P}_{\theta}:\theta\in\Theta\}$ be a collection of probability measures. For any totally bounded $T\subset\Theta$, define the Kullback-Leibler diameter by
$$d_{\text{KL}}(T)=\sup_{\theta,\theta'\in T}D(\mathbb{P}_{\theta}||\mathbb{P}_{\theta'}).$$
Then
\begin{equation}
\label{eq:fanoKL} \inf_{\wh{\theta}}\sup_{\theta\in\Theta}\mathbb{P}_{\theta}\Bigl[\rho^2\Big\{\wh{\theta}(X),\theta\Big\}\geq\frac{\epsilon^2}{4}\Bigr] \geq 1-\frac{d_{\text{KL}}(T)+\log 2}{\log\mathcal{M}(\epsilon,T,\rho)},
\end{equation}
for any $\epsilon>0$, where $\mathcal{M}(\epsilon,T,\rho)$ stands for the packing number of $T$ with radius $\epsilon$ with respect to the metric $\rho$.
\end{proposition}

\begin{proof}[Proof of Theorem \ref{thm:lower}]
We only need to deal with the case when $n>C$ for a sufficiently large constant, since when $n\leq C$, the rate is a constant and the conclusion automatically holds. 

When $k=1$, the standard lower bound argument for the one-dimensional normal mean problem \citep{lehmann2006theory} applies here, and we get the desired rate.

The case $k=2$ is studied in Section \ref{sec:upper}. 
Combining (\ref{eq:fanoKL}), (\ref{eq:packing}), and (\ref{eq:KL}) gives
$$\inf_{\wh{\theta}}\sup_{\theta\in\Theta_2}\mathbb{P}\Bigl(\norm{\wh{\theta}-\theta}^2\geq \frac{\alpha\sigma^2}{80}\log\log_2n\Bigr)\geq 1-\frac{6\alpha\log\log_2n+\log 2}{\log\log_2n}\geq c.$$
with $\alpha=1/60$ and a sufficiently small value $c>0$. Thus, with an application of Markov's inequality, we obtain the desired minimax lower bound in expectation.

Now we derive the lower bound for $k\geq 3$. 

We first consider the case $n>C$, $k>C$ and $n/k>C$ for some sufficiently large constant $C>0$. Define the space $\Theta_2^{\uparrow}(\wt{n},a,b)\subset\mathbb{R}^{\wt{n}}$ to be the class of vectors of length $\wt{n}$ that have two non-decreasing pieces taking values between $a$ and $b$ respectively. Then, construct the following space
$$\wt{T}=\bigtimes_{\ell=1}^{\ceil{\frac{k}{2}}}\wt{T}_\ell.$$ 
where for $1\leq \ell\leq \ceil{\frac{k}{2}}-1$, we define
$$\wt{T}_\ell=\Theta_2^{\uparrow}\Bigl\{\floor{\frac{2n}{k}},(2\ell-2)\sqrt{2\alpha\sigma^2\log\log_2n},(2\ell-1)\sqrt{2\alpha\sigma^2\log\log_2n}\Bigr\},$$
and
\[
\wt{T}_{\ceil{\frac{k}{2}}}= \Bigl\{k\sqrt{2\alpha\sigma^2\log\log_2n}\Bigr\}^{n-\floor{\frac{2n}{k}}\bigl(\ceil{\frac{k}{2}}-1\bigr)}.
\]
Observe that $\wt{T}\subset\Theta_k^{\uparrow}$. Thus,
\begin{eqnarray}
\nonumber \inf_{\wh{\theta}}\sup_{\theta\in\Theta_k^{\uparrow}}\mathbb{E}\norm{\wh{\theta}-\theta}^2 &\geq& \inf_{\wh{\theta}}\sup_{\theta\in \wt{T}}\mathbb{E}\norm{\wh{\theta}-\theta}^2 \\
\label{eq:suff} &=& \inf_{\wh{\theta}=(\wh{\eta}_1,...,\wh{\eta}_{\ceil{k/2}})}\sum_{\ell=1}^{\ceil{\frac{k}{2}}}\sup_{\eta_\ell\in\wt{T}_\ell}\mathbb{E}\norm{\wh{\eta}_\ell-\eta_\ell}^2 \\
\nonumber &\geq& \sum_{\ell=1}^{\ceil{\frac{k}{2}}-1}\inf_{\wh{\eta}_\ell}\sup_{\eta_\ell\in\wt{T}_\ell}\mathbb{E}\norm{\wh{\eta}_\ell-\eta_\ell}^2 \\
\label{eq:useF} &\geq& c_1\Bigl(\left\lceil\frac{k}{2}\right\rceil-1\Bigr)\log\log\floor{\frac{2n}{k}} \\
\nonumber &\geq& c_2k\log\log\Bigl(\frac{16n}{k}\Bigr),
\end{eqnarray}
where the equality (\ref{eq:suff}) is by taking advantage of the separable structure and a sufficiency argument, and the inequality (\ref{eq:useF}) is by the same argument that we use to derive the lower bound for the case $k=2$. 

Secondly, we consider the rest of settings. When $n\leq C$, the rate is a constant and the result automatically holds. When $3\leq k\leq C$, the rate $\log\log 16n$ is immediately a lower bound by the fact that $\Theta_2^{\uparrow}\subset\Theta_k^{\uparrow}$. When $n/k\leq C$, we have $\Theta_{n/C}^{\uparrow}\subset\Theta_k^{\uparrow}$. Therefore, 
$$\inf_{\wh{\theta}}\sup_{\theta\in\Theta_k^{\uparrow}}\mathbb{E}\norm{\wh{\theta}-\theta}^2\geq \inf_{\wh{\theta}}\sup_{\theta\in\Theta_{n/C}^{\uparrow}}\mathbb{E}\norm{\wh{\theta}-\theta}^2\geq c_3n.$$
Hence, the proof is complete.
\end{proof}

\section*{Acknowledgement}

The research of C. Gao was supported in part by NSF grant DMS-1712957. The research of F. Han was supported in part by NSF grant DMS-1712536. The research of C.-H. Zhang was supported in part by NSF grants
{DMS-1513378,} IIS-1407939, DMS-1721495, and IIS-1741390.
The authors thank Qiyang Han for carefully reading the manuscript and many insightful suggestions and Antoine Picard for pointing out an error in the proof. The authors also thank two referees and an associate editor for their helpful feedbacks that greatly improve the paper.

\bibliographystyle{apalike}
\bibliography{PCN}

\newpage{}


\appendix

\begin{center}
{\large {Supplement to ``On Estimation of Isotonic Piecewise Constant Signals"}}
\end{center}

\noindent This supplementary material provides proofs of remaining results in Section \ref{sec:disc}, as well as some auxiliary lemmas.

\section{Proofs of remaining upper bounds}\label{sec:proof-lower}

\begin{proof}[Proofs of Theorem \ref{thm:g-adaptation}]
We adopt the notations in the proof of Theorem \ref{thm:adaptive-oracle}. 
The proof is separated to three steps. In the first step, we show that, universally for all $k\in[n-1]$,
\[
\E\|\wh{\theta}-\theta^*\|^2\leq C\left(\|\theta^{(k)}-\theta^*\|^2+\sigma^2k\log\log(16n/k)\right).
\]
In the second step, we show that, universally over $k\in[n]$, 
\[
\E\|\wh{\theta}-\theta^*\|^2\leq C\sigma^2\min\Big\{n, \log(en)+n^{1/3}(V/\sigma)^{2/3}\Big\}.
\]
In the third step, we show that, for $k=1$,
\[
\E\|\wh{\theta}-\theta^*\|^2\leq C\left(\|\theta^{(1)}-\theta^*\|^2+\sigma^2\right).
\]
Combining the above three inequalities, we get the desired bound.

{\bf{Step 1.}} Using the same argument in proving Theorem \ref{thm:adaptive-oracle}, we obtain the bounds (\ref{eq:used-in-the-next-proof})-\eqref{eq:error-term-pen2}. The two terms in (\ref{eq:also-need-label}) can be bounded by $\sigma^2k\log\log(16n/k)$ up to a constant in expectation according to Lemma \ref{lem:simple-lil}.
For (\ref{eq:error-term-pen1}) and (\ref{eq:error-term-pen2}), we bound them by the following lemma.
\begin{lemma} \label{lem:er12-all-k-2-part}
There exist two random variables $\wt{R}_1$ and $\wt{R}_2$ such that
\begin{align}
\label{eq:replica} & \max\Big\{\sum_{f\geq 0}2^{-f}\sum_{j=1}^{k}\sum_{\{\ell\geq 1:a_{j-1}+2^{\ell-1}\leq a_j\}}\mathbb{E}\delta_+(a_{j-1},a_j,\ell+f)\xi_+(a_{j-1},a_j,\ell), \\
\nonumber &\quad\quad\quad \sum_{f\geq 0}2^{-f}\sum_{j=1}^{k}\sum_{\{\ell\geq 1:a_{j-1}\leq a_j-2^{\ell-1}\}}\mathbb{E}\delta_-(a_{j-1},a_j,\ell+f)\xi_-(a_{j-1},a_j,\ell) \Big\}\\
\nonumber \leq& \wt{R}_1+\wt{R}_2,
\end{align}
where the $\{\wh{a}_h\}$ in the definitions of $\delta_+(a_{j-1},a_j,\ell)$ and $\delta_-(a_{j-1},a_j,\ell)$ are from $\wh{A}_{\wh{k}}$ instead of $\wh{A}_k$. For the two terms in the bound, there exist universal constants $C,C'>0$, such that $\mathbb{E}\wt{R}_1\leq C\sigma^2k\log\log(16n/k)$ and $\wt{R}_2$ satisfies, for any $t\geq 0$,
\begin{equation}
\mathbb{P}\left[\wt{R}_2>C\sigma^2\left(k\log\log(16n/k)+\wh{k}\log\log(16n/\wh{k})+t\right)\right] \leq \exp(-C't). \label{eq:R-2-high-prob}
\end{equation}
\end{lemma}

By Lemma \ref{lem:er12-all-k-2-part}, we obtain the bound
\begin{eqnarray}
\nonumber && \|\wh{\theta}-\theta^*\|^2 + 2\wt{\text{pen}}_{\tau}(\wh{k}) \\
\label{eq:the-R-bound} &\leq& C\|\theta^{(k)}-\theta^*\|^2 + R_1 + R_2 + 2\wt{\text{pen}}_{\tau}(k),
\end{eqnarray}
where $\mathbb{E}R_1\leq C\sigma^2k\log\log(16n/k)$ and $R_2=\wt{R}_2$ satisfies (\ref{eq:R-2-high-prob}).

Now we derive an alternative bound. Starting from (\ref{eq:basic-k-piece-pen}), it is sufficient to bound $2\iprod{X-\theta^*}{\wh{\theta}-\theta^{(k)}}$. Using the same argument in the proof of Theorem \ref{thm:upper2}, it can be bounded by the sum of (\ref{eq:the-hardest-4th-term}) and (\ref{eq:bound-first-three-terms}). Here, we give an alternative bound for (\ref{eq:the-hardest-4th-term}). By Cauchy-Schwarz, it can be bounded as
$$\eta\|\wh{\theta}-\theta^{(k)}\|^2+\eta^{-1}\sum_{j=1}^k\sum_{\{h:(\wh{a}_{h-1}:\wh{a}_h]\subset(a_{j-1}:a_j]\}}(\wh{a}_h-\wh{a}_{h-1})\wb{Z}_{(\wh{a}_{h-1}:\wh{a}_h]}^2.$$
Since $\wh{A}_{\wh{k}}\subset\wh{A}_n$, the second term above is bounded by $\eta^{-1}\|\wh{\theta}^{(n)}-\theta^*\|^2$, where $\wh{\theta}^{(n)}=\wh{\theta}(\Theta_n^{\uparrow})$. The risk $\mathbb{E}\|\wh{\theta}^{(n)}-\theta^*\|^2$ is bounded in the following lemma.
\begin{lemma}\label{lem:isotonic-bound}
Define $\wb{l}(m)=\min\Big\{n,3m + m\sqrt{m+1}\Big(\wb{\theta}^*_{[n-m:n-m/2)}-\wb{\theta}^*_{(1+m/2:1+m]}\Big)/\sigma\Big\}$ for all $m\leq n/3$, $\wb{l}(m)=n$ for all $m>n/3$ and $\hat{l}(m)=\min\big\{n,3m + m\sqrt{m+1}\big(\wb{X}_{[n-m:n-m/2)}$ $-\wb{X}_{(1+m/2:1+m]}\big)/\sigma\big\}$. Then, there exist constants $C_1'$ and $C_2'$, such that for any $t>0$,
$$\mathbb{P}\Big(\|\wh{\theta}^{(n)}-\theta^*\|^2>C_1'(1+t)\sigma^2\sum_{\ell\geq 0}\frac{\bar{l}(2^{\ell+1})-\bar{l}(2^{\ell})}{2^{\ell+1}}\Big)\leq C_2'\Big(\frac{1}{1+t}\Big)^{1+\epsilon/2}$$
and
$$\mathbb{P}\Big(\Big|\sum_{\ell\geq 0:2^{\ell}\leq n/3}\frac{\wh{l}(2^{\ell+1})-\wb{l}(2^{\ell+1})}{2^{\ell+1}}\Big|> C_1'(1+t)\log (en)\Big)\leq C_2'\Big(\frac{1}{1+t}\Big)^2.$$
Moreover, we also have
$$\sigma^2\sum_{\ell\geq 0}\frac{\bar{l}(2^{\ell+1})-\bar{l}(2^{\ell})}{2^{\ell+1}}\leq C_1'\sigma^2\min\Big\{n,\log(en)+n^{{1/3}}\Big(\frac{V}{\sigma}\Big)^{2/3}\Big\}.$$
\end{lemma}
To summarize, we have
\begin{eqnarray}\label{eq:the-L-bound}
  \|\wh{\theta}-\theta^*\|^2 + 2\wt{\text{pen}}_{\tau}(\wh{k}) 
 \leq C\|\theta^{(k)}-\theta^*\|^2 + L_1 + L_2 + 2\wt{\text{pen}}_{\tau}(k),
\end{eqnarray}
where $L_1$ corresponds to the last three terms in \eqref{eq:bound-first-three-terms} satisfying $\mathbb{E}L_1\leq C\sigma^2k\log\log(16n/k)$ and $L_2=C\|\wh{\theta}-\theta^*\|^2$ is bounded by Lemma \ref{lem:isotonic-bound}.

Combining the two bounds (\ref{eq:the-R-bound}) and (\ref{eq:the-L-bound}), we get
\begin{eqnarray*}
&& \|\wh{\theta}-\theta^*\|^2 + 2\wt{\text{pen}}_{\tau}(\wh{k}) \\
&\leq& C\|\theta^{(k)}-\theta^*\|^2 + \min\left\{L_1 + L_2, R_1+R_2\right\} + 2\wt{\text{pen}}_{\tau}(k) \\
&\leq& C\|\theta^{(k)}-\theta^*\|^2 + L_1+R_1 + \min\left\{L_2, R_2\right\} + 2\wt{\text{pen}}_{\tau}(k).
\end{eqnarray*}
Since $\mathbb{E}(L_1+R_1)\lesssim \sigma^2k\log\log(16n/k)$, it is sufficient to give a bound for $\mathbb{E}\min\{L_2,R_2\}$. For this, we have 
\begin{align*}
& \mathbb{P}\Big(\min\{L_2,R_2\}> (1+t)\min\Big\{C_1'\sigma^2\sum_{\ell\geq 0}\frac{\bar{l}(2^{\ell+1})-\bar{l}(2^{\ell})}{2^{\ell+1}},R_2\Big\}\Big) \\
\leq& \mathbb{P}\Big(\min\{L_2,R_2\}> \min\Big\{C_1'(1+t)\sigma^2\sum_{\ell\geq 0}\frac{\bar{l}(2^{\ell+1})-\bar{l}(2^{\ell})}{2^{\ell+1}},R_2\Big\}\Big) \\
\leq& \mathbb{P}\Big(L_2>C_1'(1+t)\sigma^2\sum_{\ell\geq 0}\frac{\bar{l}(2^{\ell+1})-\bar{l}(2^{\ell})}{2^{\ell+1}}\Big) \\
\leq& C_2'\Big(\frac{1}{1+t}\Big)^{1+\epsilon/2},
\end{align*}
where the second inequality is by separately studying the cases $L_2\geq R_2$ and $L_2<R_2$, and the last inequality is by Lemma  \ref{lem:isotonic-bound}. Integrating the probability tail over $t>0$, we have
$$\mathbb{E}\min\{L_2,R_2\}\lesssim \mathbb{E}\min\Big\{C_1'\sigma^2\sum_{\ell\geq 0}\frac{\bar{l}(2^{\ell+1})-\bar{l}(2^{\ell})}{2^{\ell+1}},R_2\Big\}.$$
Now using Lemma \ref{lem:er12-all-k-2-part}, we get
\begin{align*}
& \mathbb{P}\Big(\min\Big\{C_1'\sigma^2\sum_{\ell\geq 0}\frac{\bar{l}(2^{\ell+1})-\bar{l}(2^{\ell})}{2^{\ell+1}},R_2\Big\}>  \min\Big\{C_1'\sigma^2\sum_{\ell\geq 0}\frac{\bar{l}(2^{\ell+1})-\bar{l}(2^{\ell})}{2^{\ell+1}},\\
&~~~~~C_3'\sigma^2\Big(k\log\log(16n/k) + \wh{k}\log\log(16n/\wh{k})\Big)\Big\}+C_3'\sigma^2 t\Big) \\
\leq& \mathbb{P}\Big(\min\Big\{C_1'\sigma^2\sum_{\ell\geq 0}\frac{\bar{l}(2^{\ell+1})-\bar{l}(2^{\ell})}{2^{\ell+1}},R_2\Big\}> \min\Big\{C_1'\sigma^2\sum_{\ell\geq 0}\frac{\bar{l}(2^{\ell+1})-\bar{l}(2^{\ell})}{2^{\ell+1}},\\
&~~~~~C_3'\sigma^2\Big(k\log\log(16n/k) + \wh{k}\log\log(16n/\wh{k})+t\Big)\Big\}\Big) \\
\leq& \mathbb{P}\Big\{R_2> C_3'\sigma^2\Big(k\log\log(16n/k) + \wh{k}\log\log(16n/\wh{k})+t\Big)\Big\} \\
\leq& \exp\Big(-C_4't\Big).
\end{align*}
Again, integrating the above probability tail bound over $t>0$, we have
\begin{align*}
& \mathbb{E}\min\Big\{C_1'\sigma^2\sum_{\ell\geq 0}\frac{\bar{l}(2^{\ell+1})-\bar{l}(2^{\ell})}{2^{\ell+1}},R_2\Big\} \\
\lesssim& \mathbb{E}\min\Big\{C_1'\sigma^2\sum_{\ell\geq 0}\frac{\bar{l}(2^{\ell+1})-\bar{l}(2^{\ell})}{2^{\ell+1}},C_3'\sigma^2\Big(k\log\log(16n/k) \!+\! \wh{k}\log\log(16n/\wh{k})\Big)\Big\} \\
\notag&+ \sigma^2 k.
\end{align*}
By noticing that 
$\sum_{\ell\geq0:2^\ell>n/3}\frac{\bar{l}(2^{\ell+1})-\bar{l}(2^{\ell})}{2^{\ell+1}}=0,$
we have the bound
\begin{align*}
&\mathbb{E}\min\left\{L_1 + L_2, R_1+R_2\right\} \\
\lesssim& \sigma^2k\log\log(16n/k) \!+\! \mathbb{E}\min\Big\{\sigma^2\!\!\!\!\!\!\sum_{\ell\geq 0:2^{\ell}\leq n/3}\!\!\!\frac{\bar{l}(2^{\ell+1})-\bar{l}(2^{\ell})}{2^{\ell+1}},\sigma^2\wh{k}\log\log(16n/\wh{k})\Big\} \\
\leq& \sigma^2k\log\log(16n/k) \!+\! \mathbb{E}\min\Big\{\sigma^2\!\!\!\!\!\!\sum_{\ell\geq 0:2^{\ell}\leq n/3}\!\!\!\frac{\wh{l}(2^{\ell+1})-\wh{l}(2^{\ell})}{2^{\ell+1}},\sigma^2\wh{k}\log\log(16n/\wh{k})\Big\} \\
& + \mathbb{E}\min\Big\{\sigma^2\left|\sum_{\ell\geq 0:2^{\ell}\leq n/3}\frac{\wh{l}(2^{\ell+1})-\wb{l}(2^{\ell+1})}{2^{\ell+1}}\right|,\sigma^2\wh{k}\log\log(16n/\wh{k})\Big\}.
\end{align*}
By Lemma \ref{lem:isotonic-bound}, we have
\begin{eqnarray*}
&& \mathbb{P}\Big(\min\Big\{\sigma^2\left|\sum_{\ell\geq 0:2^{\ell}\leq n/3}\frac{\wh{l}(2^{\ell+1})-\wb{l}(2^{\ell+1})}{2^{\ell+1}}\right|,\sigma^2\wh{k}\log\log(16n/\wh{k})\Big\} > \\
&&\quad\quad  (1+t)\min\Big\{C_1'\sigma^2\log (en),\sigma^2\wh{k}\log\log(16n/\wh{k})\Big\}\Big) \\
&\leq& \mathbb{P}\Big(\min\Big\{\sigma^2\left|\sum_{\ell\geq 0:2^{\ell}\leq n/3}\frac{\wh{l}(2^{\ell+1})-\wb{l}(2^{\ell+1})}{2^{\ell+1}}\right|,\sigma^2\wh{k}\log\log(16n/\wh{k})\Big\} > \\
&&\quad\quad  \min\Big\{C_1'(1+t)\sigma^2\log (en),\sigma^2\wh{k}\log\log(16n/\wh{k})\Big\}\Big) \\
&\leq& \mathbb{P}\Big(\left|\sum_{\ell\geq 0:2^{\ell}\leq n/3}\frac{\wh{l}(2^{\ell+1})-\wb{l}(2^{\ell+1})}{2^{\ell+1}}\right|> C_1'(1+t)\log (en)\Big) \\
&\leq& C_2'\Big(\frac{1}{1+t}\Big)^2.
\end{eqnarray*}
Integrating the probability tail bound over $t>0$, we have
\begin{align*}
&\mathbb{E}\min\Big\{\sigma^2\Big|\sum_{\ell\geq 0:2^{\ell}\leq n/3}\frac{\wh{l}(2^{\ell+1})-\wb{l}(2^{\ell+1})}{2^{\ell+1}}\Big|,\sigma^2\wh{k}\log\log(16n/\wh{k})\Big\}\\
\lesssim& \mathbb{E}\min\Big\{\sigma^2\log(en),\sigma^2\wh{k}\log\log(16n/\wh{k})\Big\}.
\end{align*}
Therefore, we obtain the bound
\begin{align*}
&\mathbb{E}\min\left\{L_1 + L_2, R_1+R_2\right\}\lesssim \sigma^2k\log\log(16n/k)+\\
&\quad\quad\quad\quad\mathbb{E}\min\Big\{\sigma^2\sum_{\ell\geq 0:2^{\ell}\leq n/3}\frac{\wh{l}(2^{\ell+1})-\wh{l}(2^{\ell})}{2^{\ell+1}}+\sigma^2\log(en),\sigma^2\wh{k}\log\log(16n/\wh{k})\Big\},
\end{align*}
which is bounded by $\sigma^2+\wt{\text{pen}}_{\tau}(k)+\wt{\text{pen}}_{\tau}(\wh{k})$ up to a constant if we choose $\tau=C_1\sigma^2$ for some large constant $C_1>0$. Therefore, for some (possibly different) universal constant $C>3$,
we have
$$\mathbb{E}\|\wh{\theta}-\theta^*\|^2+2\mathbb{E}\wt{\text{pen}}_{\tau}(\wh{k}) \leq C\|\theta^{(k)}-\theta^*\|^2+2\wt{\text{pen}}_{\tau}(k)+C\left(\mathbb{E}\wt{\text{pen}}_{\tau}(\wh{k})+\wt{\text{pen}}_{\tau}(k)\right),$$
which implies $\mathbb{E}\|\wh{\theta}-\theta^*\|^2\lesssim \|\theta^{(k)}-\theta^*\|^2+\wt{\text{pen}}_{\tau}(k)$, the desired conclusion for all $2\leq k\leq n-1$.

{{\bf Step 2}.} For $k\in [n]$, we observe that (\ref{eq:wh-k-def}) is equivalent to $\wh{k}=\argmin_k\{\|\wh{\theta}(\Theta_k^{\uparrow})-\wh{\theta}^{(n)}\|^2+\wt{\text{pen}}_{\tau}(k)\}$, which implies $\|\wh{\theta}-\wh{\theta}^{(n)}\|^2\leq \wt{\text{pen}}_{\tau}(n)$. Therefore, $$\mathbb{E}\|\wh{\theta}-\theta^*\|^2\leq 2\mathbb{E}\wt{\text{pen}}_{\tau}(n)+2\mathbb{E}\|\wh{\theta}^{(n)}-\theta^*\|^2.$$ The desired bound is thus implied by Lemma \ref{lem:isotonic-bound}.

{\bf Step 3}.  This is similar to the argument that leads to (\ref{eq:pf-k=1-later}) in the proof of Theorem \ref{thm:adaptive-oracle}. So we omit the details here.
\end{proof}

\begin{proof}[Proof of Theorem \ref{thm:upper1} (upper bound)]
Consider the estimator $\wh{\theta}=\argmin_{\theta\in\Theta_k}\|X-\theta\|^2$. The observation $X$ follows $N(\theta^*,\sigma^2I_n)$ with some $\theta^*\in\Theta_{k}$. The conclusion for $k=1$ is obvious. When $k\geq 3$, the risk bound $\sigma^2k\log(en/k)$ has been derived in the literature \citep{birge2001gaussian,boysen2009consistencies,li2016fdr}.
The bound $\sigma^2\log\log(16n)$ for $k=2$ follows the same argument in proving Theorem \ref{thm:upper2} because $\theta^*$ is monotone in this case.
\end{proof}

\begin{proof}[Proof of Theorem \ref{thm:unimodal} (upper bound)]
Since $\Theta_k^{\wedge}=\Theta_k$ for $k=1,2$, we only need to prove the upper bound for $k\geq 3$.
We construct an estimator using the aggregation strategy in \cite{leung2006information}. Using $X\sim N(\theta^*,\sigma^2I_n)$,  we construct two i.i.d. vectors $U=X+W$ and $V=X-W$, where $X\sim N(0,\sigma^2I_n)$ is independent of $X$. Then, it is easy to see that $U,V\sim N(\theta^*,2\sigma^2I_n)$ and are independent from each other.

We first use $U$ to construct some estimators. For any $k\geq 3$, define
$$\Omega_k=\Big\{(u,v,\ell)\in\mathbb{Z}^3:0\leq u,v,\ell\leq n, u+v=k, u\leq \ell, v\leq n-\ell\Big\}.$$
For any $(u,v,\ell)\in\Omega_k$, we construct an estimator that is unimodal with the mode at $\ell$ and has at most $u$ and $v$ steps to the left and to the right of $\ell$, respectively. We use $\Theta_{(k,m)}^{\uparrow}$ and $\Theta_{(k,m)}^{\downarrow}$ to denote non-decreasing and non-increasing vectors in $\mathbb{R}^{m}$ that have at most $k$ steps. In particular, the space $\Theta_k^{\uparrow}$ can be written as $\Theta_{(k,n)}^{\uparrow}$. We define $\wh{\theta}_{(u,v,\ell)}$ to be the concatenation of vectors $\argmin_{\eta\in\Theta_{(u,\ell)}^{\uparrow}}\|U_{(0:\ell]}-\eta\|^2$ and $\argmin_{\eta\in\Theta_{(v,n-\ell)}^{\downarrow}}\|U_{(\ell:n]}-\eta\|^2$. Then, using the results of Theorem \ref{thm:upper2}, we have
\begin{eqnarray*}
\mathbb{E}\|\wh{\theta}_{(u,v,\ell)}-\theta^*\|^2 &\leq& C\sigma^2\Big(u\log\log(16\ell/u)+v\log\log(16(n-\ell)/v)\Big) \\
&\leq& C\sigma^2\Big(u\log\log(16n/u)+(k-u)\log\log(16n/(k-u))\Big) \\
&\leq& 2C\sigma^2 k\log\log(16n/k),
\end{eqnarray*}
uniformly over $\Theta_{(u,v,\ell)}$. The space $\Theta_{(u,v,\ell)}$ is defined to be the class of all $\theta$'s in $\Theta_k^{\wedge}$ such that the mode of $\theta$ is at $\ell$ and it has at most $u$ and $v$ steps to the left and to the right of $\ell$, respectively. It is easy to see that $\Theta_k^{\wedge}=\cup_{(u,v,\ell)\in\Omega_k}\Theta_{(u,v,\ell)}$.

We then use $V$ to aggregate all $\{\wh{\theta}_{(u,v,\ell)}\}$. Define the probability simplex on $\Omega_k$ by $\Lambda^{\Omega_k}=\Big\{\{\lambda_{\omega}\}_{\omega\in\Omega_k}:\lambda_{\omega}\geq 0, \sum_{\omega\in\Omega_k}=1\Big\}$. The vector $\pi\in\Lambda^{\Omega_k}$ is defined as $\pi_{\omega}=|\Omega_k|^{-1}$ for all $\omega\in\Omega_k$. Define
$$\wh{\lambda}^V=\argmin_{\lambda\in\Lambda^{\Omega_k}}\Big\{\sum_{\omega\in\Omega_k}\lambda_{\omega}\|V-\wh{\theta}_{\omega}^U\|^2+8\sigma^2D(\lambda\|\pi)\Big\}.$$
Our final aggregated estimator is $\wh{\theta}=\sum_{\omega\in\Omega_k}\wh{\lambda}^V_{\omega}\wh{\theta}^U_{\omega}$. The result of \cite{leung2006information} states that
$$
\mathbb{E}\norm{\wh{\theta}-\theta^*}^2\leq \min_{\omega\in\Omega_k}\Bigl\{\mathbb{E}\norm{\wh{\theta}_{\omega}^U-\theta^*}^2+8\sigma^2\log(1/\pi_{\omega})\Bigr\}.
$$
Therefore,
\begin{eqnarray*}
\sup_{\theta^*\in\Theta_k^{\wedge}}\mathbb{E}\|\wh{\theta}-\theta^*\|^2 &=& \max_{(u,v,\ell)\in\Omega_k}\sup_{\theta\in\Theta_{u,v,\ell}}\mathbb{E}\|\wh{\theta}-\theta^*\|^2 \\
&\leq& \max_{(u,v,\ell)\in\Omega_k}\sup_{\theta\in\Theta_{u,v,\ell}}\mathbb{E}\|\wh{\theta}_{(u,v,\ell)}^U-\theta^*\|^2 + 8\sigma^2\log|\Omega_k| \\
&\leq& 2C\sigma^2 k\log\log(16n/k) + 16\sigma^2\log(n+1),
\end{eqnarray*}
where the last inequality is by $|\Omega_k|\leq (n+1)^2$. Therefore, we obtain the desired upper bound for $k\geq 3$, and the proof is complete.
\end{proof}

\begin{proof}[Proof of Proposition \ref{prop:lp-risk} (upper bound)]
The upper bound is a direct implication of Theorem 2.1 in \cite{zhang2002risk}. 
\end{proof}

\begin{proof}[Proof of Proposition \ref{prop:lp-reduced} (upper bound)]
Let's denote $\hat\theta=\hat\theta(\Theta_k^\uparrow)$. We then have
\[
\E\norm{\hat\theta-\theta^*}_p^p=\sum_{i=1}^n\E|\hat\theta_i-\theta^*_i|^p\leq \sum_{i=1}^n(\E|\hat\theta_i-\theta^*_i|^2)^{p/2}\leq n^{1-p/2}\Big\{\sum_{i=1}^n(\E|\hat\theta_i-\theta^*_i|^2)\Big\}^{p/2}.
\]
Using the previous bound on $\sup_{\theta^*\in\Theta_k^\uparrow}\E\norm{\hat\theta-\theta^*}^2$ finishes the proof.
\end{proof}

\section{Proofs of remaining lower bounds}

\begin{proof}[Proof of Theorem \ref{thm:upper1} (lower bound)]
When $k=1$, the lower bound is trivial. When $k=2$, we have
$$\inf_{\wh{\theta}}\sup_{\theta\in\Theta_2}\mathbb{E}\norm{\wh{\theta}-\theta}^2\geq \inf_{\wh{\theta}}\sup_{\theta\in\Theta_2^{\uparrow}}\mathbb{E}\norm{\wh{\theta}-\theta}^2\geq c\log\log 16n.
$$
When $k\geq 3$, the problem is reduced to finding the minimax lower bound for a sparse normal mean estimation problem. Define the space of sparse vectors
\begin{equation}
S_\ell=\Bigl\{\theta\in\mathbb{R}^n: \sum_{i=1}^n\ind\{\theta_i\neq 0\}\leq \ell\Bigr\}.\label{eq:sparse-class}
\end{equation}
Then, we observe that $S_{\floor{\frac{k-1}{2}}}\subset\Theta_k$. This leads to the argument
\begin{align*}
\inf_{\wh{\theta}}\sup_{\theta\in\Theta_k}\norm{\wh{\theta}-\theta}^2\geq \inf_{\wh{\theta}}\sup_{\theta\in S_{\floor{\frac{k-1}{2}}}}\norm{\wh{\theta}-\theta}^2\geq C_1k\log(en/k),
\end{align*}
where the last inequality above is given by \cite{donoho1994minimax}. The proof is complete.
\end{proof}

\begin{proof}[Proof of Theorem \ref{thm:unimodal} (lower bound)]
When $k\leq 2$, $\Theta_{k}^{\wedge}=\Theta_k$. Thus, the results are the same as those for $\Theta_k$. For $k\geq 3$, we have
$$\inf_{\wh{\theta}}\sup_{\theta\in\Theta_k^{\wedge}}\mathbb{E}\|\wh{\theta}-\theta\|^2\geq \inf_{\wh{\theta}}\sup_{\theta\in\Theta_k^{\uparrow}}\mathbb{E}\|\wh{\theta}-\theta\|^2\geq c\sigma^2k\log\log(16n/k),$$
and
$$\inf_{\wh{\theta}}\sup_{\theta\in\Theta_k^{\wedge}}\mathbb{E}\|\wh{\theta}-\theta\|^2\geq \inf_{\wh{\theta}}\sup_{\theta\in\Theta_3^{\wedge}}\mathbb{E}\|\wh{\theta}-\theta\|^2\geq \inf_{\wh{\theta}}\sup_{\theta\in S_1}\mathbb{E}\|\wh{\theta}-\theta\|^2\geq c\sigma^2\log(en),$$
where $S_1$ is defined in (\ref{eq:sparse-class}). Therefore,
$$\inf_{\wh{\theta}}\sup_{\theta\in\Theta_k^{\wedge}}\mathbb{E}\|\wh{\theta}-\theta\|^2\geq c\sigma^2\max\Big\{k\log\log(16n/k),\log(en)\Big\},$$
which leads to the desired results for $k\geq 3$.
\end{proof}

We then give the proof of Proposition \ref{prop:imp}. This requires the following result to bound the Kullback-Leibler divergence.

\begin{lemma}\label{lem:KL-gamma}
Consider the density function $p_{\gamma,a}(x)\propto\exp\bigl(-|x-a|^{\gamma}\bigr)$ for some $\gamma\in(0,2]$ and $a\in\mathbb{R}$. Then, there exists some universal constant $C>0$, such that
$$D(p_{\gamma,a}||p_{\gamma,b})\leq\begin{cases}
C|a-b|^{\gamma}, & \gamma\in(0,1],\\
C\bigl(|a-b|+|a-b|^{\gamma}\bigr), &\gamma\in(1,2].
\end{cases}$$
\end{lemma}
The proof of Lemma \ref{lem:KL-gamma} is given in Section \ref{sec:aux}.

\begin{proof}[Proof of Proposition \ref{prop:imp}]
Let $e_j$ be the $j$th canonical vector of $\mathbb{R}^n$. That is, the entries of $e_j$ are all $0$ except that the $j$th entry is $1$. Again, we only consider the $n$ that is large enough. Construct the space $T=\{\alpha(\log n)^{1/\gamma}e_j\}_{j=1}^n$. It is easy to see that $T\subset\Theta_3$. For any $\theta,\theta'\in T$, we have $\norm{\theta-\theta'}^2=2\alpha(\log n)^{2/\gamma}$. Therefore,
$$\log\mathcal{M}\Bigl(\sqrt{2}\alpha(\log n)^{1/\gamma},T,\norm{\cdot}\Bigr)\geq\log n.$$
Moreover, using Lemma \ref{lem:KL-gamma}, we have
$$\max_{\theta,\theta'\in T}D(\mathbb{P}_{\theta}||\mathbb{P}_{\theta'})\leq C_1(\alpha+\alpha^{\gamma})\log n.$$
Using Fano's inequality (\ref{eq:KL}), we have
$$\inf_{\wh{\theta}}\sup_{\theta\in\Theta_3}\mathbb{P}\Bigl\{\norm{\wh{\theta}-\theta}^2\geq 2\alpha^2(\log n)^{2/\gamma}\Bigr\}\geq 1-\frac{C_1(\alpha+\alpha^{\gamma})\log n+\log 2}{\log n}\geq c,$$
as long as we choose a small enough $\alpha$. Thus, with an application of Markov's inequality, the proof is complete.
\end{proof}

\begin{proof}[Proof of Proposition \ref{cor:test}]
Recall the notation $\mathbb{E}_{\theta}$ that stands for the expectation associated with the probability measure $\mathbb{P}_{\theta}=N(\theta,\sigma^2I_n)$. We only consider the case when $n$ is large enough. 
We consider the alternative set of parameters $\cF(\rho)$ that contains vectors $\{\theta_\ell\in\Theta_2; \ell\in[\floor{\log_2 n}]\}$ that fill the last $\ceil{n2^{-\ell}}$ entries with $\rho\sigma\sqrt{2^{\ell}\log\log_2n/n}$ and the rest 0.  Let $\mu_\rho$ be the uniform measure on $\cF(\rho)$ and $\rho$ be some sufficiently small constant. We use the notation $\mathbb{P}_{\mu_{\rho}}=\int\mathbb{P}_{\theta}d\mu_{\rho}$ and $\mathbb{E}_{\mu_{\rho}}$ for its expectation. Using Le Cam's method \citep{yu1997assouad}, we have 
\begin{eqnarray*}
&& \inf_{0\leq \phi\leq 1}\Big\{\sup_{\theta\in\Theta_1}\mathbb{E}_\theta\phi+\sup_{\theta\in\Theta_2(c)}\mathbb{E}_{\theta}(1-\phi)\Big\} \\
&\geq& \inf_{0\leq \phi\leq 1}\Big\{\mathbb{E}_0\phi+\mathbb{E}_{\mu_{\rho}}(1-\phi)\Big\} \\
&\geq& 1-\frac{1}{2}\Big\{\E_{0} L^2_{\mu_{\rho}}(Y)-1\Big\}^{1/2},
\end{eqnarray*}
where we set $L_{\mu_\rho}(y):= \frac{d\P_{\mu_\rho}}{d\P_0}(y).$ 
The rest of this proof shows $\E_{0} L^2_{\mu_{\rho}}(Y)=1+o(1)$ as $n\to 0$. To this end, we calculate
\[
L_{\mu_\rho}(y)=\frac{1}{\floor{\log_2 n}}\sum_{\theta\in\cF(\rho)}\exp\Big( \frac{2\theta^\T y-\norm{\theta}^2}{2\sigma^2} \Big),
\]
yielding
\begin{align*}
\E_{0} L_{\mu_{\rho}}^2(Y)=&\frac{1}{\floor{\log_2 n}^2}\sum_{\theta_1,\theta_2\in\cF(\rho)}\exp\Big(\frac{\theta_1^\T\theta_2}{2\sigma^2}\Big)\\
=&\frac{1}{\floor{\log_2 n}^2} \sum_{j=1}^{\floor{\log_2 n}}\sum_{k=1}^{\floor{\log_2 n}}\exp\Big( \rho^22^{(j+k)/2-1}\log\log_2n/n\cdot \ceil{n2^{-\max(j,k)}}\Big)\\
=& \frac{1}{q^2}\sum_{j=1}^q\sum_{k=1}^q\big(q^{\rho^2}\big)^{2^{-|j-k|/2-1}}(1+o(1)),
\end{align*}
where $q:=\floor{\log_2n}$. We then truncate the array $\{(j,k):1\leq j,k\leq q\}$ to two parts: $T_1:=\{(j,k): |j-k|\leq 2\log_2q\}$ and $T_2:=\{(j,k): |j-k|> 2\log_2q\}$. It is immediate that $|T_1|\asymp (\log_2q)^2$ and $|T_2|= q^2(1+o(1))$. Then
\[
\frac{1}{q^2}\sum_{j=1}^q\sum_{k=1}^q\big(q^{\rho^2}\big)^{2^{-|j-k|/2-1}}= \underbrace{\frac{1}{q^2}\sum_{(j,k)\in T_1}\big(q^{\rho^2}\big)^{2^{-|j-k|/2-1}}}_{A_1}+\underbrace{\frac{1}{q^2}\sum_{(j,k)\in T_2}\big(q^{\rho^2}\big)^{2^{-|j-k|/2-1}}}_{A_2},
\]
with
\[
A_1\leq C\frac{(\log_2 q)^2}{q^2}\cdot q^{\rho^2/2}=o(1),
\]
and each element in $A_2$ satisfying
\[
1\leq \big(q^{\rho^2}\big)^{2^{-|j-k|/2-1}}\leq (q^{\rho^2})^{1/2q}=1+o(1).
\]
This yields $\E_{0} L_{\mu_{\rho}}^2(Y)=1+o(1)$ and hence completes the proof.
\end{proof}

\begin{proof}[Proof of Proposition \ref{prop:lp-risk} (lower bound)]
The lower bound is classic (see, for example, Example 4.2.2 in \cite{lehmann2006theory}).
\end{proof}

\begin{proof}[Proof of Proposition \ref{prop:lp-reduced} (lower bound)]
Without loss of generality, assume $\sigma^2=1$ and $n$ is even. Following the same logic as in the proof of Proposition \ref{prop:lower-LSE}, we only need to study large enough $n$ and can focus on the following specific model of only one change point:
\[
\theta^*_1=\cdots=\theta^*_{n/2}=0~~~{\rm and}~~~\theta^*_{n/2+1}=\cdots=\theta^*_n=\eta:=\sqrt{c\log\log n/n},
\]
for some small enough universal constant $c$. 

We first argue that for the estimated change point $\hat a$ of $\hat\theta$ such that $\hat\theta_{\hat a}\ne \hat\theta_{\hat a+1}$, either $\{\hat a<n/3\}$ or $\{\hat a>2n/3\}$ will have a nonvanishing probability. To this end, notice that $\hat a$ is the one that maximizes
\begin{align*}
\Lambda_a&:=a\bar X_{(0:a]}^2+(n-a)\bar X_{(a:n]}^2\\
&=\begin{cases} 
a\bar\epsilon_{(0:a]}^2+(n-a)\bar\epsilon_{(a:n]}^2+(n-2a)\eta\bar\epsilon_{(a:n]}+\frac{(n-2a)^2}{4(n-a)}\eta^2, & a\leq n/2,
\cr a\bar\epsilon_{(0:a]}^2+(n-a)\bar\epsilon_{(a:n]}^2+(2a-n)\eta\bar\epsilon_{(0:a]}+\frac{(2a-n)^2}{4a}\eta^2, & a>n/2.
\end{cases}
\end{align*}
By Theorem 1.1.2 and Theorem A.3.4 in \cite{csorgo1997limit} (or more explicitly, by combining Equation (1.4.5) and the proof of Theorem 1.6.1 in \cite{csorgo1997limit}), we have
\[
\max_{n/\log n\leq a\leq n-n/\log n}\frac{a\bar\epsilon_{(0:a]}^2+(n-a)\bar\epsilon_{(a:n]}^2}{2\log\log\log n} \stackrel{\P}{\longrightarrow} 1,
\]
which immediately yields
\[
\max_{n/\log n\leq a\leq n-n/\log n}\frac{c\Lambda_a}{4\log\log n} \stackrel{\P}{\longrightarrow} 1
\]
On the other hand, by Theorem 1.3.1 in \cite{csorgo1997limit}, we have
\[
\max_{a\in[n]}\frac{a\bar\epsilon_{(0:a]}^2+(n-a)\bar\epsilon_{(a:n]}^2}{2\log\log n} \stackrel{\P}{\longrightarrow} 1
\]
Accordingly, by forcing $c$ small enough, we have
\[
\lim_{n\to\infty}\Big\{\P(\hat a<n/\log n) + \P(n-\hat a<n/\log n)\Big\} =1.
\]
This proves the assertion.

Following that, without loss of generality we assume the event $\lim_{n\to\infty}\P(\hat a<n/3)>0$. With the convention that summation over an empty set is zero, we then have 
\begin{align*}
\E\norm{\hat\theta-\theta^*}_p^p&\geq \E\sum_{i\in(\hat a,n/2]}|\hat\theta_i-\theta^*_i|^p=\E (n/2-\hat a)_+\Big|\bar\epsilon_{(\hat a,n]}+\sqrt{cn\log\log n/4}/(n-\hat a)\Big|^p\\
&\geq 2^{1-p}c_1(\log\log n/n)^{p/2}\E(n/2-\hat a)_+-\E(n-\hat a)|\bar\epsilon_{(\hat a,n]}|^p\\
&\geq c_2n^{1-p/2}(\log\log n)^{p/2}-\E\sup_{a\in(0,n]}a|\bar\epsilon_{a}^2|\\
&\geq c_3n^{1-p/2}(\log\log n)^{p/2}.
\end{align*}
Here the third inequality uses the result in Step 1 and the last inequality uses Lemma \ref{lem:lp1} given below, whose proof will be given in Section \ref{sec:aux}. The case for a general $k$ follows the same argument used in the proof of Proposition \ref{prop:lower-LSE}.
\end{proof}

\begin{lemma}\label{lem:lp1}
Let $X_1,\ldots,X_n$ be i.i.d. of mean zero, variance one, and denote $S_k=\sum_{i=1}^k X_i$. We then have, for any $1\leq p< 2$. 
\[
\E\Big\{\max_{k\in [n]} k\Big(\frac{|S_k|}{k}\Big)^p\Big\}\lesssim n^{1-p/2}.
\]
\end{lemma}

\section{Proofs of auxiliary results} \label{sec:aux}

This section collects the proofs of Lemma \ref{lem:er12}, Lemma \ref{lem:simple-lil}, Lemma \ref{lem:er12-all-k}, Lemma \ref{lem:er12-all-k-2-part}, Lemma \ref{lem:KL-gamma}, and Lemma \ref{lem:lp1}.

To prove Lemma \ref{lem:er12}, we need the following two famous maximal inequalities. The versions we present here are Corollary II.1.6 in \cite{revuz1999continuous} and Proposition 1.1.2 in \cite{de2012decoupling}.

\begin{lemma}[Doob's maximal inequality]\label{lem:doob}
Given a martingale $\{M_i, i=1,2,\ldots\}$ and a scalar $p>1$, we have for any $n\geq 1$,
$$\Bigl\{\mathbb{E}\Bigl(\max_{1\leq i\leq n}|M_i|^p\Bigr)\Bigr\}^{1/p}\leq \frac{p}{p-1}\Bigl(\mathbb{E}|M_n|^p\Bigr)^{1/p}.$$
\end{lemma}
\begin{lemma}[Levy-Ottaviani inequality] \label{lem:levy}
Given $n$ independent random variables $X_1,...,X_n$, we have for any $x>0$,
$$\mathbb{P}\Bigl(\max_{1\leq k\leq n}\Bigl|\sum_{i=1}^kX_i\Bigr|>x\Bigr)\leq 3\max_{1\leq k\leq n}\mathbb{P}\Bigl(\Bigl|\sum_{i=1}^kX_i\Bigr|>x/3\Bigr).$$
\end{lemma}

\begin{proof}[Proof of Lemma \ref{lem:er12}]
Since the proofs of the two inequalities in the lemma are the same, we only state the proof of the first one that involves $\delta_+(a_{j-1},a_j,\ell)$ and $\xi_+(a_{j-1},a_j,\ell)$. We first consider the case $f=0$. For $a+2^{\ell-1}\leq t\wedge b$, we observe that
$$|\bar{Z}_{(a:t\wedge b]}|^2\leq 2|\bar{Z}_{(a:(a+2^{\ell-1})]}|^2 + 2\Big(\frac{1}{2^{\ell-1}}\Big)^2(t\wedge b-a-2^{\ell-1})^2|\bar{Z}_{((a+2^{\ell-1}):t\wedge b]}|^2.$$
This leads to the inequality
$$\xi_+(a,b,\ell)\leq 2^{\ell+1}|\bar{Z}_{(a:(a+2^{\ell-1})]}|^2 + {8}\bar{\xi}_+(a,b,\ell),$$
where
\begin{align*}
\bar{\xi}_+(a,b,\ell)&:=2^{-\ell}\max\Big\{(t\wedge b-a-2^{\ell-1})^2|\bar{Z}_{((a+2^{\ell-1}):t\wedge b]}|^2:a+2^{\ell-1}< t\leq b\wedge (a+2^{\ell}-1)\Big\}\\
&\leq {\max\Big\{(t\wedge b-a-2^{\ell-1})|\bar{Z}_{((a+2^{\ell-1}):t\wedge b]}|^2:a+2^{\ell-1}< t\leq b\wedge (a+2^{\ell}-1)\Big\}}.
\end{align*}
Therefore, it is sufficient to bound the sum of
\begin{equation}
{8}\sum_{j=1}^{k}\sum_{\{\ell\geq 1:a_{j-1}+2^{\ell-1}\leq a_j\}}\mathbb{E}\delta_+(a_{j-1},a_j,\ell)\bar{\xi}_+(a_{j-1},a_j,\ell)\label{eq:indep1}
\end{equation}
and
\begin{equation}
2\sum_{j=1}^{k}\sum_{\{\ell\geq 1:a_{j-1}+2^{\ell-1}\leq a_j\}}\mathbb{E}\delta_+(a_{j-1},a_j,\ell)2^{\ell}|\bar{Z}_{(a_{j-1}:(a_{j-1}+2^{\ell-1})\wedge a_j]}|^2.\label{eq:indep2}
\end{equation}

\paragraph{Bounding (\ref{eq:indep1}).}
The proof consists of a two-layer truncation argument. We first split each $Z_i$ into two parts. That is,
$$Z_{i\ell}'=Z_i\ind\Big\{Z_i^2\leq \sigma^22^\ell/\ell\Big\},\quad \text{and}\quad Z_{i\ell}''=Z_i\ind\Big\{Z_i^2> \sigma^22^\ell/\ell\Big\}.$$
We also define
$$Y_{i\ell}'=Z_{i\ell}'-\mathbb{E}Z_{i\ell}',\quad\text{and}\quad Y_{i\ell}''=Z_{i\ell}''-\mathbb{E}Z_{i\ell}''.$$
Then, it is easy to see that $Z_i=Y_{i\ell}'+Y_{i\ell}''$ given that $\E Z_i=\mathbb{E}Z_{i\ell}'+\mathbb{E}Z_{i\ell}''=0$.
Using the notation
$$\mathcal{C}(a,b,\ell)=\big\{t:a+2^{\ell-1}< t\leq b\wedge(a+2^\ell-1)\big\},$$
we have the bound
\begin{eqnarray}
\label{eq:t'} \mathbb{E}\delta_+(a,b,\ell)\bar{\xi}_+(a,b,\ell) &\leq& 2\mathbb{E}\Bigl(\delta_+(a,b,\ell)2^{-\ell}\max_{t\in\mathcal{C}(a,b,\ell)}\Bigl|\sum_{i=a+2^{\ell-1}+1}^tY_{i\ell}'\Bigr|^2\Bigr) \\
\label{eq:t''} && + 2\mathbb{E}\Bigl(\delta_+(a,b,\ell)2^{-\ell}\max_{t\in\mathcal{C}(a,b,\ell)}\Bigl|\sum_{i=a+2^{\ell-1}+1}^tY_{i\ell}''\Bigr|^2\Bigr),
\end{eqnarray}
and we will bound the two terms \eqref{eq:t'} and \eqref{eq:t''} separately. We first give a bound for (\ref{eq:t''}).
\begin{eqnarray}
\nonumber && \mathbb{E}\Bigl(\delta_+(a,b,\ell)2^{-\ell}\max_{t\in\mathcal{C}(a,b,\ell)}\Bigl|\sum_{i=a+2^{\ell-1}+1}^tY_{i\ell}''\Bigr|^2\Bigr) \\
\nonumber &\leq& \mathbb{E}\Bigl(2^{-\ell}\max_{t\in\mathcal{C}(a,b,\ell)}\Bigl|\sum_{i=a+2^{\ell-1}+1}^tY_{i\ell}''\Bigr|^2\Bigr) \\
\label{eq:doob} &\leq& 4\mathbb{E}\Bigl(2^{-\ell}\Bigl|\sum_{i=a+2^{\ell-1}+1}^{a+2^{\ell}-1}Y_{i\ell}''\Bigr|^2\Bigr) \\
\label{eq:sim} &=& 4\times 2^{-\ell}\sum_{i=a+2^{\ell-1}+1}^{a+2^\ell-1}\mathbb{E}(Y_{i\ell}'')^2 \\
\nonumber &\leq& 4\times 2^{-\ell}\sum_{i=a+2^{\ell-1}+1}^{a+2^\ell-1}\mathbb{E}(Z_{i\ell}'')^2 \\
\label{eq:holder} &\leq& 4\times 2^{-\ell}\sum_{i=a+2^{\ell-1}+1}^{a+2^\ell-1}\Bigl(\mathbb{E}|Z_i|^{2+\epsilon}\Bigr)^{\frac{2}{2+\epsilon}}\mathbb{P}\Bigl(Z_i^2> \sigma^22^\ell/\ell\Bigr)^{\frac{\epsilon}{2+\epsilon}} \\
\label{eq:markov} &\leq& 4\times 2^{-\ell}\sum_{i=a+2^{\ell-1}+1}^{a+2^\ell-1}\mathbb{E}|Z_i|^{2+\epsilon}\Bigl(\frac{\ell}{\sigma^22^\ell}\Bigr)^{\epsilon/2} \\
\nonumber &\leq& C_1\sigma^2\Bigl(\frac{\ell}{2^\ell}\Bigr)^{\epsilon/2}.
\end{eqnarray}
We have used Doob's maximal inequality (Lemma \ref{lem:doob}) to derive (\ref{eq:doob}). The equality \eqref{eq:sim} is because of the fact $\mathbb{E}Y_{i\ell}''=0$. Finally, we have used H\"{o}lder's inequality and Markov's inequality to derive (\ref{eq:holder}) and (\ref{eq:markov}), respectively.  
When $Z_i$'s are identically distributed, we have
\begin{eqnarray}\label{iid-case}
&& \sum_{\ell\ge 1} 4\times 2^{-\ell}\sum_{i=a+2^{\ell-1}+1}^{a+2^\ell-1}\mathbb{E}(Z_{i\ell}'')^2 \\
\nonumber &\le & 2 \sum_{\ell=1}^\infty \mathbb{E}|Z_1|^2\mathbb{I}\{Z_1^2> \sigma^22^\ell/\ell\}\\
\nonumber &\le & 2 \mathbb{E}|Z_1|^2 \sum_{\ell=1}^\infty \mathbb{I}\{(C_1/2) \log(e+Z_1^2/\sigma^2) \ge \ell\}\\
\nonumber &\lesssim& \mathbb{E}|Z_1|^2\log(e+Z_1^2/\sigma^2). 
\end{eqnarray}

Next, we are going to derive a bound for (\ref{eq:t'}). For simplicity, we use the notation
\[
\eta(a,b,\ell)=2^{-\ell}\max_{t\in\mathcal{C}(a,b,\ell)}\Bigl|\sum_{i=a+2^{\ell-1}+1}^tY_{i\ell}'\Bigr|^2.
\]
Notice that $\{\eta(a_{j-1},a_j,\ell)\}_{j,\ell}$ are independent across all $j$ and $\ell$. We first show that $\sqrt{\eta(a,b,\ell)}$ has a mixed-type sub-Gaussian and sub-exponential tail. For any $x>0$, we have
\begin{eqnarray}
\nonumber \mathbb{P}\Bigl\{\sqrt{\eta(a,b,\ell)}>\sigma x\Bigr\} &\leq& \mathbb{P}\Bigl(2^{-\ell/2}\max_{t\in\mathcal{C}(a,b,\ell)}\Bigl|\sum_{i=a+2^{\ell-1}+1}^tY_{i\ell}'\Bigr|>\sigma x\Bigr) \\
\label{eq:levy} &\leq& 3\max_{t\in\mathcal{C}(a,b,\ell)}\mathbb{P}\Bigl(2^{-\ell/2}\Bigl|\sum_{i=a+2^{\ell-1}+1}^{t}Y_{i\ell}'\Bigr|>\sigma x/3\Bigr) \\
\label{eq:bernstein} &\leq& 6\exp\Bigl(-C_2\min\{x^2,\sqrt{\ell}x\}\Bigr),
\end{eqnarray}
where we have used Levy's maximal inequality (Lemma \ref{lem:levy}) and Bernstein's inequality to derive (\ref{eq:levy}) and (\ref{eq:bernstein}), respectively. This motivates another truncation argument on $\eta(a,b,\ell)$. That is, we consider the split $\eta(a,b,\ell)=\eta'(a,b,\ell)+\eta''(a,b,\ell)$,
where
$$\eta'(a,b,\ell)=\eta(a,b,\ell)\ind\Bigl\{\eta(a,b,\ell)\leq \sigma^2\ell\Bigr\}\quad\text{and}\quad\eta''(a,b,\ell)=\eta(a,b,\ell)\ind\Bigl\{\eta(a,b,\ell)> \sigma^2\ell\Bigr\}.$$
We first give a bound for $\mathbb{E}\Bigl\{\delta_+(a,b,\ell)\eta''(a,b,\ell)\Bigr\}$:
\begin{eqnarray*}
\mathbb{E}\Bigl\{\delta_+(a,b,\ell)\eta''(a,b,\ell)\Bigr\} &\leq& \mathbb{E}\eta''(a,b,\ell) \\
&\leq& \Bigl\{\mathbb{E}\eta^2(a,b,\ell)\Bigr\}^{1/2}\mathbb{P}\Bigl\{\eta(a,b,\ell)> \sigma^2\ell\Bigr\}^{1/2} \\
&\leq& C_3\sigma^2\exp\Bigl(-C_2\ell\Bigr),
\end{eqnarray*}
where the last inequality above is obtained by integrating the tail
$$\mathbb{E}\eta^2(a,b,\ell)=\sigma^4\int_0^{\infty}\mathbb{P}\Bigl\{\sqrt{\eta(a,b,\ell)}>\sigma u^{1/4}\Bigr\}du$$
using the tail bound (\ref{eq:bernstein}). The term $\mathbb{E}\Bigl\{\delta_+(a,b,\ell)\eta'(a,b,\ell)\Bigr\}$ will be analyzed in the end. Combining all the bounds above, we have
\begin{equation}
\mathbb{E}\delta_+(a,b,\ell)\bar{\xi}_+(a,b,\ell) \leq 4C_1\sigma^2\Bigl(\frac{\ell}{2^\ell}\Bigr)^{\epsilon/2}+4C_3\sigma^2\exp\Bigl(-C_2\ell\Bigr)+4\mathbb{E}\Bigl\{\delta_+(a,b,\ell)\eta'(a,b,\ell)\Bigr\}. \label{eq:marked}
\end{equation}
Replacing $a$ and $b$ in (\ref{eq:marked}) by $a_{j-1}$ and $a_j$, and summing up over $\ell$ and $j$, we have
\begin{eqnarray}
\nonumber && \sum_{j=1}^{k}\sum_{\{\ell\geq 1:a_{j-1}+2^{\ell-1}\leq a_j\}}\mathbb{E}\delta_+(a_{j-1},a_j,\ell)\bar{\xi}_+(a_{j-1},a_j,\ell) \\
\nonumber &\leq& 2C_1k\sigma^2\sum_\ell\Bigl(\frac{\ell}{2^\ell}\Bigr)^{\epsilon/2} + 2C_3k\sigma^2\sum_\ell\exp(-C_2\ell) \\
\nonumber && + 2\sum_{j=1}^{k}\sum_{\{\ell\geq 1:a_{j-1}+2^{\ell-1}\leq a_j\}}\mathbb{E}\Bigl\{\delta_+(a_{j-1},a_j,\ell)\eta'(a_{j-1},a_j,\ell)\Bigr\} \\
\label{eq:last-one} &\leq& C_4k\sigma^2 + 2\sum_{j=1}^{k}\sum_{\{\ell\geq 1:a_{j-1}+2^{\ell-1}\leq a_j\}}\mathbb{E}\Bigl\{\delta_+(a_{j-1},a_j,\ell)\eta'(a_{j-1},a_j,\ell)\Bigr\}.
\end{eqnarray}
When $Z_i$'s are identically distributed, we are allowed to replace the term 
$\sum_\ell (\ell/2^\ell)^{\epsilon/2}$ in the above inequality by $C_1$ in view of 
(\ref{iid-case}). 
We omit the proof for identically distributed $Z_i$'s in the sequel as its difference 
only involves another application of the above argument.

Finally, it suffices to give a bound for the second term in (\ref{eq:last-one}). We shorthand $\eta'(a_{j-1},a_j,\ell)$ by $\eta_{j\ell}$. Observe that
\begin{eqnarray}\label{eq:han-add1}
&& \sum_{j=1}^{k}\sum_{\{\ell\geq 1:a_{j-1}+2^{\ell-1}\leq a_j\}}\delta_+(a_{j-1},a_j,\ell)\eta'(a_{j-1},a_j,\ell) \notag \\
&\leq& \max\Bigl\{\sum_{j=1}^{k}\sum_{\ell=1}^{\ceil{1\vee\log_2n_j}}\delta_{j\ell}\eta_{j\ell}:\delta_{j\ell}\in\{0,1\}, \sum_{j,\ell}\delta_{j\ell}\leq\wt{k}\Bigr\},
\end{eqnarray}
where $n_j=a_j-a_{j-1}$ and $\wt{k}=\min\{k,m\}$ with $m=\sum_{j=1}^{k}\ceil{1\vee\log_2n_j}$. Equation \eqref{eq:han-add1} leads to a union bound argument. That is, for any $x>0$, we have
\begin{eqnarray}
\nonumber && \mathbb{P}\Bigl\{\sum_{j=1}^{k}\sum_{\{\ell\geq 1:a_{j-1}+2^{\ell-1}\leq a_j\}}\delta_+(a_{j-1},a_j,\ell)\eta'(a_{j-1},a_j,\ell)>\frac{8\sigma^2}{C_2}\wt{k}\log\frac{em}{\wt{k}}+x\sigma^2\Bigr\} \\
\nonumber &\leq& \sum_{\{\{\delta_{j\ell}\}:\delta_{j\ell}\in\{0,1\}, \sum_{j,\ell}\delta_{j\ell}\leq \wt{k}\}}\mathbb{P}\Bigl(\sum_{j,\ell}\delta_{j\ell}\eta_{j\ell}>\frac{8\sigma^2}{C_2}\wt{k}\log\frac{em}{\wt{k}}+x\sigma^2\Bigr) \\
\nonumber &\leq&\sum_{\{\{\delta_{j\ell}\}:\delta_{j\ell}\in\{0,1\}, \sum_{j,\ell}\delta_{j\ell}\leq \wt{k}\}} \exp\Bigl(-4L\log\frac{em}{L}-C_2x/2\Bigr)\prod_{j,\ell}\mathbb{E}\exp\Bigl(\frac{C_2\delta_{j\ell}\eta_{j\ell}}{2\sigma^2}\Bigr) \\
\label{eq:yulu} &\leq& {m\choose \wt{k}}\exp\Bigl(-4\wt{k}\log\frac{em}{\wt{k}}-C_2x/2+\wt{k}\log 5\Bigr)\\
\nonumber &\leq& \exp\Big(-C_2x/2\Big).
\end{eqnarray}
To derive (\ref{eq:yulu}), note that for $\delta_{j\ell}=1$, we have
\begin{eqnarray}
\nonumber \mathbb{E}\exp\Bigl(\frac{C_2\eta_{j\ell}}{2\sigma^2}\Bigr) &=& \int_0^{\infty}\mathbb{P}\Bigl\{\exp\Bigl(\frac{C_2\eta_{j\ell}}{2\sigma^2}\Bigr)>u\Bigr\}du \\
\nonumber &\leq& 1+\int_1^{\infty}\mathbb{P}\Bigl(\sqrt{\eta_{j\ell}}>\sigma\sqrt{\frac{2\log u}{C_2}}\Bigr)du \\
\nonumber &=& 1+\int_1^{e^{C_2\ell/2}}\mathbb{P}\Bigl\{\sqrt{\eta(a_{j-1},a_j,\ell)}>\sigma\sqrt{\frac{2\log u}{C_2}}\Bigr\}du \\
\label{eq:useb} &\leq& 1+4\int_1^{e^{C_2\ell/2}}\exp\Bigl[-C_2\min\Bigl\{\frac{2\log u}{C_2},\sqrt{\frac{2\ell\log u}{C_2}}\Bigr\}\Bigr]du \\
\nonumber &\leq& 1+4\int_1^{\infty}u^{-2}du = 5,
\end{eqnarray}
where (\ref{eq:useb}) is an application of the tail bound (\ref{eq:bernstein}). The tail bound (\ref{eq:useb}) allows us to integrate out the tail and bound the expectation. That is,
$$\sum_{j=1}^{k}\sum_{\{\ell\geq 1:a_{j-1}+2^{\ell-1}\leq a_j\}}\mathbb{E}\Bigl\{\delta_+(a_{j-1},a_j,\ell)\eta'(a_{j-1},a_j,\ell)\Bigr\}\leq C_5\sigma^2\wt{k}\log(em/\wt{k}).$$
In view of (\ref{eq:last-one}), we have
$$\sum_{j=1}^{k}\sum_{\{\ell\geq 1:a_{j-1}+2^{\ell-1}\leq a_j\}}\mathbb{E}\delta_+(a_{j-1},a_j,\ell)\bar{\xi}_+(a_{j-1},a_j,\ell)\leq C_4\sigma^2k+C_5\sigma^2\wt{k}\log(em/\wt{k}).$$
This gives the desired bound for \eqref{eq:indep1} by realizing that $\wt{k}\log(em/\wt{k})\lesssim k\log\log(16n/k)$.

\paragraph{Bounding (\ref{eq:indep2}).}
For any $\ell\geq 1$ such that $a_{j-1}+2^{\ell-1}\leq a_j$, we have
\begin{align*}
& 2^{\ell}|\bar{Z}_{(a_{j-1}:(a_{j-1}+2^{\ell-1})]}|^2 \\
\leq& 2^{\ell}\Big(\frac{3}{8}|\bar{Z}_{(a_{j-1}:(a_{j-1}+2^{\ell-2})]}|^2+\frac{3}{4}|\bar{Z}_{((a_{j-1}+2^{\ell-2}):(a_{j-1}+2^{\ell-1})]}|^2\Big) \\
\leq& 2^{\ell}\Big(\frac{9}{64}|\bar{Z}_{(a_{j-1}:(a_{j-1}+2^{\ell-3})]}|^2+\frac{9}{32}|\bar{Z}_{((a_{j-1}+2^{\ell-3}):(a_{j-1}+2^{\ell-2})]}|^2+\frac{3}{4}|\bar{Z}_{((a_{j-1}+2^{\ell-2}):(a_{j-1}+2^{\ell-1})]}|^2\Big) \\
\leq& 2^{\ell}\frac{3}{4}\sum_{h=0}^{\ell-1}\Big(\frac{3}{8}\Big)^{\ell-1-h}|\bar{Z}_{((a_{j-1}+2^{h-1}):(a_{j-1}+2^{h})]}|^2.
\end{align*}
We introduce the notation
$$u_{jh}=2^{h-1}|\bar{Z}_{((a_{j-1}+2^{h-1}):(a_{j-1}+2^{h})]}|^2.$$
Notice that $\{u_{jh}\}_{j,h}$ are independent across all $j$ and $h$. Then, we have
$$2^{\ell}|\bar{Z}_{(a_{j-1}:(a_{j-1}+2^{\ell-1})]}|^2\leq 4\sum_{h=0}^{\ell-1}\Big(\frac{3}{4}\Big)^{\ell-h}u_{jh}.$$
Therefore, (\ref{eq:indep2}) can be bounded by
\begin{equation}
8\sum_{j=1}^{k}\sum_{\{\ell\geq 1: a_{j-1}+2^{\ell-1}\leq a_j\}}\mathbb{E}\delta_+(a_{j-1},a_j,\ell)\sum_{h=0}^{\ell-1}\Big(\frac{3}{4}\Big)^{\ell-h}u_{jh}.\label{eq:roy-duke}
\end{equation}
A similar double truncation argument that is used to drive (\ref{eq:marked}) also gives
$$\mathbb{E}\delta_+(a_{j-1},a_j,\ell)u_{jh}\leq C_6\sigma^2\Big(\frac{h}{2^h}\Big)^{\epsilon/2}+C_7\sigma^2\exp\Big(-C_8h\Big)+4\mathbb{E}\Big(\delta_+(a_{j-1},a_j,\ell)u_{jh}'\Big),$$
where the random variable $u_{jh}'$ satisfies $\mathbb{E}\exp\Big(\frac{tu_{jh}'}{\sigma^2}\Big)\leq e^{c't}$ for all $0<t<c$ for some small constant $c>0$, and is independent across $j$ and $h$. Summing up the first two terms over $j,\ell,h$, we get
$$8\sum_{j=1}^{k}\sum_{\{\ell\geq 1: a_{j-1}+2^{\ell-1}\leq a_j\}}\sum_{h=0}^{\ell-1}\Big(\frac{3}{4}\Big)^{\ell-h}\Big(C_6\sigma^2\Big(\frac{h}{2^h}\Big)^{\epsilon/2}+C_7\sigma^2\exp\Big(-C_8h\Big)\Big)\leq C_9\sigma^2k.$$
Summing up the third term over $j,\ell,h$, we get
$$32\sum_{j=1}^{k}\sum_{\{\ell\geq 1: a_{j-1}+2^{\ell-1}\leq a_j\}}\mathbb{E}\delta_+(a_{j-1},a_j,\ell)\sum_{h=0}^{\ell-1}\Big(\frac{3}{4}\Big)^{\ell-h}u_{jh}'.$$
We use a union bound argument to bound its value. For any $x>0$, we have
\begin{align*}
& \mathbb{P}\Big(\sum_{j=1}^{k}\sum_{\{\ell\geq 1: a_{j-1}+2^{\ell-1}\leq a_j\}}\delta_+(a_{j-1},a_j,\ell)\sum_{h=0}^{\ell-1}\Big(\frac{3}{4}\Big)^{\ell-h}u_{jh}'>x\sigma^2+\wt{C}\sigma^2\wt{k}\log\frac{em}{\wt{k}}\Big) \\
\leq& \sum_{\{\{\delta_{jl}\}:\delta_{jl}\in\{0,1\},\sum_{j,l}\delta_{jl}\leq \wt{k}\}}\mathbb{P}\Big(\sum_{j=1}^{k}\sum_{\{\ell\geq 1: a_{j-1}+2^{\ell-1}\leq a_j\}}\sum_{h=0}^{\ell-1}\Big(\frac{3}{4}\Big)^{\ell-h}\delta_{j\ell}u_{jh}'>x\sigma^2+\wt{C}\sigma^2\wt{k}\log\frac{em}{\wt{k}}\Big) \\
\leq& \sum_{\{\{\delta_{jl}\}:\delta_{jl}\in\{0,1\},\sum_{j,l}\delta_{jl}\leq \wt{k}\}}e^{-\lambda x-\lambda\wt{C}\wt{k}\log\frac{em}{\wt{k}}}\mathbb{E}\exp\Big(\lambda\sum_{j,\ell,h}\Big(\frac{3}{4}\Big)^{\ell-h}\delta_{j\ell}u_{jh}'/\sigma^2\Big) \\
=& \sum_{\{\{\delta_{jl}\}:\delta_{jl}\in\{0,1\},\sum_{j,l}\delta_{jl}\leq \wt{k}\}}e^{-\lambda x-\lambda\wt{C}\wt{k}\log\frac{em}{\wt{k}}}\prod_{j,h}\mathbb{E}\exp\Big(\lambda\sum_{\ell}\Big(\frac{3}{4}\Big)^{\ell-h}\delta_{j\ell}u_{jh}'/\sigma^2\Big) \\
\leq& \sum_{\{\{\delta_{jl}\}:\delta_{jl}\in\{0,1\},\sum_{j,l}\delta_{jl}\leq \wt{k}\}}e^{-\lambda x-\lambda\wt{C}\wt{k}\log\frac{em}{\wt{k}}}\exp\Big(c'\lambda\sum_{j,\ell,h}\Big(\frac{3}{4}\Big)^{\ell-h}\delta_{j\ell}\Big) \\
\leq& \sum_{\{\{\delta_{jl}\}:\delta_{jl}\in\{0,1\},\sum_{j,l}\delta_{jl}\leq \wt{k}\}}e^{-\lambda x-\lambda\wt{C}\wt{k}\log\frac{em}{\wt{k}}}\exp\Big(c_1\lambda\sum_{j,\ell}\delta_{j\ell}\Big) \\
\leq& {m\choose \wt{k}}\exp\Big(-\lambda x + c_1\lambda \wt{k} -\lambda\wt{C}\wt{k}\log\frac{em}{\wt{k}}\Big) \\
\leq& \exp\Big(-\lambda x\Big),
\end{align*}
where we take a sufficiently large $\wt{C}$ and $\lambda$ is chosen to be a constant so that $\lambda\sum_{\ell}\Big(\frac{3}{4}\Big)^{\ell-h}<c$ for all $h$.
Therefore, by integrating up the tail, we get
\begin{align*}
& 32\sum_{j=1}^{k}\sum_{\{\ell\geq 1: a_{j-1}+2^{\ell-1}\leq a_j\}}\mathbb{E}\delta_+(a_{j-1},a_j,\ell)\sum_{h=0}^{\ell-1}\Big(\frac{3}{4}\Big)^{\ell-h}u_{jh}'\\
\leq& C''\sigma^2\wt{k}\log(em/\wt{k}).
\end{align*}
Combining the bounds, we obtain $C'_1\sigma^2k+C'_2\sigma^2\wt{k}\log(em/\wt{k})\lesssim \sigma^2k\log\log(16n/k)$ as an upper bound for (\ref{eq:indep2}).

The proof for $f\geq 1$ is the same, because the proof only depends on the constraint that $\sum_{j,\ell}\delta_{j\ell}\leq\wt{k}$, which is not affected by the value of $f$.
\end{proof}

\begin{proof}[Proof of Lemma \ref{lem:simple-lil}]
Note that
$$\sum_{j=1}^k\mathbb{E}\max_{a_{j-1}<a\leq a_j}(a-a_{j-1})\wb{Z}_{(a_{j-1}:a]}^2\leq \sum_{j=1}^k\sum_{\{\ell\geq 1: a_{j-1}+2^{\ell-1}\leq a_j\}}\mathbb{E}\delta^-_{j\ell}\xi_-(a_{j-1},a_j,\ell),$$
where $\{\delta_{j\ell}^-\}$ are binary random variables that satisfy $\sum_{j\ell}\delta_{j\ell}^-\leq \wt{k}$. Therefore, we obtain the bound $\sigma^2k\log\log(16n/k)$ by the same argument in the proof of Lemma \ref{lem:er12}. The second term can be bounded in the same way, and thus the proof is complete.
\end{proof}

\begin{proof}[Proofs of Lemma \ref{lem:er12-all-k} and Lemma \ref{lem:er12-all-k-2-part}]
It is sufficient to prove Lemma \ref{lem:er12-all-k-2-part} because the conclusion of Lemma \ref{lem:er12-all-k} can be obtained by integrating out the tail bound given by Lemma \ref{lem:er12-all-k-2-part}.
The analysis is very similar to the proof of Lemma \ref{lem:er12}. The only difference is that Lemma \ref{lem:er12} is for a fixed $k$, while here we need to analyze a random $\wh{k}$. The quantities $\delta_+(a_{j-1},a_j,\ell)$, $\delta_-(a_{j-1},a_j,\ell)$, $\xi_+(a_{j-1},a_j,\ell)$, and $\xi_-(a_{j-1},a_j,\ell)$ are defined with this random $\wh{k}$ instead of a fixed one. With slight abuse of notation, we still use $\wt{k}=\min\{\wh{k},m\}$, where $m=\sum_{j=1}^{k}\ceil{1\vee\log_2n_j}$. Note that here $\wt{k}$ is random as well.

We only state the analysis of the first term in (\ref{eq:replica}) that involves $\delta_+(a_{j-1},a_j,\ell)$ and $\xi_+(a_{j-1},a_j,\ell)$. The analysis for the second term uses the same argument, and is thus omitted. We first consider $f=0$. Use the same argument in the proof of Lemma \ref{lem:er12}, and it is sufficient to bound the sum of
\begin{equation}
\sum_{j=1}^{k}\sum_{\{\ell\geq 1:a_{j-1}+2^{\ell-1}\leq a_j\}}\delta_+(a_{j-1},a_j,\ell)\bar{\xi}_+(a_{j-1},a_j,\ell) \label{eq:indep1-r}
\end{equation}
and
\begin{equation}
\sum_{j=1}^{k}\sum_{\{\ell\geq 1:a_{j-1}+2^{\ell-1}\leq a_j\}}\delta_+(a_{j-1},a_j,\ell)2^{\ell}|\bar{Z}_{(a_{j-1}:(a_{j-1}+2^{\ell-1})\wedge a_j]}|^2. \label{eq:indep2-r}
\end{equation}

\paragraph{Bounding (\ref{eq:indep1-r}).}
This step is very similar to the corresponding step in bounding (\ref{eq:indep1}).  Following the arguments that lead to (\ref{eq:last-one}), it can be shown that (\ref{eq:indep1-r}) can be bounded by the sum of two random variables. The first one has bound $O(k\sigma^2)$ in expectation as in the first term in (\ref{eq:last-one}), and for the second term, we need to bound
$$\sum_{j=1}^{k}\sum_{\{\ell\geq 1:a_{j-1}+2^{\ell-1}\leq a_j\}}\Bigl\{\delta_+(a_{j-1},a_j,\ell)\eta'(a_{j-1},a_j,\ell)\Bigr\}.$$
We then shorthand $\eta'(a_{j-1},a_j,\ell)$ by $\eta_{j\ell}$. Observe that
\begin{eqnarray*}
&& \sum_{j=1}^{k}\sum_{\{\ell\geq 1:a_{j-1}+2^{\ell-1}\leq a_j\}}\delta_+(a_{j-1},a_j,\ell)\eta'(a_{j-1},a_j,\ell) \\
&\leq& \max\Bigl\{\sum_{j=1}^{k}\sum_{\ell=1}^{\ceil{1\vee\log_2n_j}}\delta_{j\ell}\eta_{j\ell}:\delta_{j\ell}\in\{0,1\}, \sum_{j,\ell}\delta_{j\ell}\leq\wt{k}\Bigr\},
\end{eqnarray*}
where $n_j=a_j-a_{j-1}$ and $\wt{k}=\min\{\hat{k},m\}$ with $m=\sum_{j=1}^{k}\ceil{1\vee\log_2n_j}$. This then leads to a union bound argument. That is, for any $x>0$, we have
\begin{eqnarray}
\nonumber && \mathbb{P}\Bigl\{\sum_{j=1}^{k}\sum_{\{\ell\geq 1:a_{j-1}+2^{\ell-1}\leq a_j\}}\delta_+(a_{j-1},a_j,\ell)\eta'(a_{j-1},a_j,\ell)>\frac{8\sigma^2}{C_2}\wt{k}\log\frac{em}{\wt{k}}+x\sigma^2\Bigr\} \\
\nonumber &\leq& \sum_{L=1}^m\mathbb{P}\Bigl\{\sum_{j=1}^{k}\sum_{\{\ell\geq 1:a_{j-1}+2^{\ell-1}\leq a_j\}}\delta_+(a_{j-1},a_j,\ell)\eta'(a_{j-1},a_j,\ell)>\frac{8\sigma^2}{C_2}\wt{k}\log\frac{em}{\wt{k}}+x\sigma^2, \wt{k}=L\Bigr\} \\
\nonumber &\leq& \sum_{L=1}^m\sum_{\{\{\delta_{j\ell}\}:\delta_{j\ell}\in\{0,1\}, \sum_{j,\ell}\delta_{j\ell}\leq L\}}\mathbb{P}\Bigl(\sum_{j,\ell}\delta_{j\ell}\eta_{j\ell}>\frac{8\sigma^2}{C_2}L\log\frac{em}{L}+x\sigma^2\Bigr) \\
\nonumber &\leq& \sum_{L=1}^m\sum_{\{\{\delta_{j\ell}\}:\delta_{j\ell}\in\{0,1\}, \sum_{j,\ell}\delta_{j\ell}\leq L\}} \exp\Bigl(-4L\log\frac{em}{L}-C_2x/2\Bigr)\prod_{j,\ell}\mathbb{E}\exp\Bigl(\frac{C_2\delta_{j\ell}\eta_{j\ell}}{2\sigma^2}\Bigr) \\
\label{eq:yulu-r} &\leq& \sum_{L=1}^m{m\choose L}\exp\Bigl(-4L\log\frac{em}{L}-C_2x/2+L\log 5\Bigr)\\
\nonumber &\leq& \exp\Big(-C_2x/2\Big).
\end{eqnarray}
The inequality (\ref{eq:yulu-r}) can be derived in the same way as (\ref{eq:yulu}).

\paragraph{Bounding (\ref{eq:indep2-r}).}
Similar to the corresponding step in bounding (\ref{eq:indep2}), (\ref{eq:indep2-r}) can also be bounded by two terms. The first term has a bound $O(\sigma^2 k)$ in expectation. For the second term, we need to bound
$$\sum_{j=1}^{k}\sum_{\{\ell\geq 1: a_{j-1}+2^{\ell-1}\leq a_j\}}\delta_+(a_{j-1},a_j,\ell)\sum_{h=0}^{\ell-1}\Big(\frac{3}{4}\Big)^{\ell-h}u_{jh}',$$
where we use the same notation as in (\ref{eq:roy-duke}).
We use a union bound argument to bound its value. For any $x>0$, we have
\begin{align}
& \mathbb{P}\Big(\sum_{j=1}^{k}\sum_{\{\ell\geq 1: a_{j-1}+2^{\ell-1}\leq a_j\}}\delta_+(a_{j-1},a_j,\ell)\sum_{h=0}^{\ell-1}\Big(\frac{3}{4}\Big)^{\ell-h}u_{jh}'>x\sigma^2+\wt{C}\sigma^2\wt{k}\log\frac{em}{\wt{k}}\Big) \notag\\
\leq& \sum_{L=1}^m \mathbb{P}\Big(\sum_{j=1}^{k}\sum_{\{\ell\geq 1: a_{j-1}+2^{\ell-1}\leq a_j\}}\delta_+(a_{j-1},a_j,\ell)\sum_{h=0}^{\ell-1}\Big(\frac{3}{4}\Big)^{\ell-h}u_{jh}'>x\sigma^2+\wt{C}\sigma^2\wt{k}\log\frac{em}{\wt{k}},\wt{k}=L\Big) \notag\\
\leq& \sum_{L=1}^m\sum_{\{\{\delta_{jl}\}:\delta_{jl}\in\{0,1\},\sum_{j,l}\delta_{jl}\leq L\}}\mathbb{P}\Big(\sum_{j=1}^{k}\sum_{\{\ell\geq 1: a_{j-1}+2^{\ell-1}\leq a_j\}}\sum_{h=0}^{\ell-1}\Big(\frac{3}{4}\Big)^{\ell-h}\delta_{j\ell}u_{jh}'>x\sigma^2+\wt{C}\sigma^2L\log\frac{em}{L}\Big) \notag\\
\leq& \sum_{L=1}^m\sum_{\{\{\delta_{jl}\}:\delta_{jl}\in\{0,1\},\sum_{j,l}\delta_{jl}\leq L\}}e^{-\lambda x-\lambda\wt{C}L\log\frac{em}{L}}\mathbb{E}\exp\Big(\lambda\sum_{j,\ell,h}\Big(\frac{3}{4}\Big)^{\ell-h}\delta_{j\ell}u_{jh}'/\sigma^2\Big) \notag\\
=& \sum_{L=1}^m\sum_{\{\{\delta_{jl}\}:\delta_{jl}\in\{0,1\},\sum_{j,l}\delta_{jl}\leq L\}}e^{-\lambda x-\lambda\wt{C}L\log\frac{em}{L}}\prod_{j,h}\mathbb{E}\exp\Big(\lambda\sum_{\ell}\Big(\frac{3}{4}\Big)^{\ell-h}\delta_{j\ell}u_{jh}'/\sigma^2\Big) \notag\\
\leq& \sum_{L=1}^m\sum_{\{\{\delta_{jl}\}:\delta_{jl}\in\{0,1\},\sum_{j,l}\delta_{jl}\leq L\}}e^{-\lambda x-\lambda\wt{C}L\log\frac{em}{L}}\exp\Big(c'\lambda\sum_{j,\ell,h}\Big(\frac{3}{4}\Big)^{\ell-h}\delta_{j\ell}\Big) \label{eq:han-add2}\\
\leq& \sum_{L=1}^m\sum_{\{\{\delta_{jl}\}:\delta_{jl}\in\{0,1\},\sum_{j,l}\delta_{jl}\leq L\}}e^{-\lambda x-\lambda\wt{C}L\log\frac{em}{L}}\exp\Big(c_1\lambda\sum_{j,\ell}\delta_{j\ell}\Big) \notag\\
\leq& \sum_{L=1}^m{m\choose L}\exp\Big(-\lambda x + c_1\lambda L -\lambda\wt{C}L\log\frac{em}{L}\Big) \notag\\
\leq& \exp\Big(-\lambda x\Big),\notag
\end{align}
where we take a sufficiently large $\wt{C}$ and $\lambda$ is chosen to be a constant so that $\lambda\sum_{\ell}\Big(\frac{3}{4}\Big)^{\ell-h}<c$ for all $h$, which facilitates \eqref{eq:han-add2} by observing that $\mathbb{E}\exp\Big(\frac{tu_{jh}'}{\sigma^2}\Big)\leq e^{c't}$ for all $0<t<c$.

In fact, the above analysis holds for any $f\geq 0$, since the change from $\delta_+(a_{j-1},a_j,\ell)$ to $\delta_+(a_{j-1},a_j,\ell+f)$ does not affect the argument.
To conclude, we just derived the bound $\wt{R}_1+\wt{R}_2=\sum_{f\geq 0}2^{-f}\wt{R}_1(f)+\sum_{f\geq 0}2^{-f}\wt{R}_2(f)$ with $\mathbb{E}\wt{R}_1=\mathbb{E}\sum_{f\geq 0}2^{-f}\wt{R}_1(f)\lesssim k\sigma^2$. The term $\wt{R}_2(f)$ enjoys the tail bound that for any $t>0$,
$$\mathbb{P}\Big(\wt{R}_2(f)>C_3\sigma^2\Big(\wt{k}\log(em/\wt{k})+t\Big)\Big)\leq\exp\Big(-C_4t\Big).$$
Therefore, we have for $\wt{R}_2=\sum_{f\geq 0}2^{-f}\wt{R}_2(f)$,
\begin{eqnarray*}
&& \mathbb{P}\Big(\wt{R}_2>C_5\sigma^2\Big(\wt{k}\log(em/\wt{k})+t\Big)\Big) 
\leq \exp\Big(-C_6t\Big)
\end{eqnarray*}
by a standard argument on the sum of possibly dependent subexponential random variables.

To finish the proof, we study the order of $\wt{k}\log(em/\wt{k})$ and the goal is to prove that the following inequality
\[
\wt{k}\log(em/\wt{k})\lesssim k\log\log(16n/k) + \wh{k}\log\log(16n/\wh{k})
\]
universally holds. Recall that $\wt{k}=\min\{\wh k, m\}$, and we separate the problem to three cases. First, if $\wh k\geq m$, we have $\wt{k}\log(em/\wt{k})=m\leq \wh k\leq \wh{k}\log\log(16n/\wh{k})$. Secondly, if $\wh k\leq k$, since the function $x\log (em/x)$ is strictly increasing over the range $x\in(0,m]$, we have $\wt{k}\log(em/\wt{k})\leq k\log(em/k)\lesssim k\log\log(16n/k)$. Thirdly, if $\wh k\in (k,m)$, we have 
\[
\wt{k}\log(em/\wt{k})= \wh k\log \Big( \frac{e\sum_{j=1}^{k}[1\vee \log_2 n_j]}{\wh k}\Big)\leq \wh k\log \Big( \frac{e\sum_{j=1}^{k}[1\vee \log_2 n_j]}{k}\Big)\lesssim \wh k\log\log(16n/k).
\]
Since $\wh k\in (k,m)$ implies $k<\wh k<k\log(en/k)$ and $\log(16n/k)-\log(16n/\wh k)=\log (\wh k)-\log (k)<\log\log(16n/k)$ within this regime, we have, for $\wh k\in (k,m)$,
\[
\wt{k}\log(em/\wt{k})\leq\wh k\log\log(16n/k)\lesssim \wh{k}\log\log(16n/\wh{k}).
\]
Combining the above three cases yields $\wt{k}\log(em/\wt{k})\lesssim k\log\log(16n/k) + \wh{k}\log\log(16n/\wh{k})$,  and the proof is hence complete.
\end{proof}

\begin{proof}[Proof of Lemma \ref{lem:isotonic-bound}]
We omit the superscript and use $\hat{\theta}$ for $\hat{\theta}^{(n)}$ and $\theta$ for $\theta^*$. We decompose $\|\hat{\theta}-\theta\|^2$ as the sum of $\sum_{j=1}^n(\hat{\theta}_j-\theta_j)^2_+$ and $\sum_{j=1}^n(\hat{\theta}_j-\theta_j)^2_-$. Following \cite{zhang2002risk}, we define
$$m_j=\max\Big\{m\geq 0: \wb{\theta}_{[j:j+m]}-\theta_j\leq v(m), j+m\leq n\Big\},$$
where the value $v(m)$ will be determined later. We then have
$$\sum_{j=1}^n(\hat{\theta}_j-\theta_j)^2_+\leq 2\sum_{j=1}^nv(m_j)^2 + 2\sum_{j=1}^n\Big(\min_{j\leq l\leq j+m_j}\max_{k\leq j}\wb{Z}_{[k:l]}\Big)_+^2.$$
Note that
\begin{eqnarray*}
&& \mathbb{P}\Big(\sum_{j=1}^n\Big(\min_{j\leq l\leq j+m_j}\max_{k\leq j}\wb{Z}_{[k:l]}\Big)_+^2 > t\Big)  \\
&\leq& \frac{1}{t^{1+\epsilon/2}}\mathbb{E}\Big(\sum_{j=1}^n\Big(\min_{j\leq l\leq j+m_j}\max_{k\leq j}\wb{Z}_{[k:l]}\Big)_+^2\Big)^{1+\epsilon/2} \\
&\leq& \frac{1}{t^{1+\epsilon/2}}\mathbb{E}\Big(\sum_{j=1}^n\Big(\max_{k\leq j}\wb{Z}_{[k:j+m_j]}\Big)_+^2\Big)^{1+\epsilon/2} \\
&\leq& \frac{1}{t^{1+\epsilon/2}}\Big(\sum_{j=1}^n\Big(\mathbb{E}\Big(\max_{k\leq j}\wb{Z}_{[k:j+m_j]}\Big)_+^{2+\epsilon}\Big)^{\frac{1}{1+\epsilon/2}}\Big)^{1+\epsilon/2} \\
&\leq& C\Big(\frac{1}{t}\sum_{j=1}^n\frac{\sigma^2}{m_j+1}\Big)^{1+\epsilon/2},
\end{eqnarray*}
where the third inequality is due to triangle inequality and the last inequality is by noticing $\max_{1\leq i\leq n}\mathbb{E}\Bigl|Z_i/\sigma\Bigr|^{2+\epsilon}\leq C_1$ and employing Doob's maximal inequality for reverse submartingales as used in the end of Page 534 in \cite{zhang2002risk}. The same analysis is also applied to $\sum_{j=1}^n(\hat{\theta}_j-\theta_j)^2_-$. We take $v(m_j)=\sqrt{\frac{\sigma^2}{m_j+1}}$, and then $\sum_{j=1}^nv(m_j)^2=\sum_{j=1}^n\frac{\sigma^2}{m_j+1}$. This implies that there exist constants $C_1$ and $C_2$, such that
$$\mathbb{P}\Big(\|\hat{\theta}-\theta\|^2> C_1(1+t)\sum_{j=1}^n\frac{\sigma^2}{m_j+1}\Big)\leq C_2\Big(\frac{1}{1+t}\Big)^{1+\epsilon/2}.$$
Now it is sufficient to give an upper bound for $\sum_{j=1}^n\frac{\sigma^2}{m_j+1}$. Define $l(m)=|\{j:m_j<m\}|$. Then,
$$\sigma^2\sum_{j=1}^n\frac{1}{m_j+1}\leq \sigma^2\sum_{\ell\geq 0} \frac{1}{2^{\ell}+1}\Big(l(2^{\ell+1})-l(2^{\ell})\Big).$$
Now we derive an upper bound for $l(m)$. By the definition of $m_j$, we have
\begin{eqnarray*}
l(m) &\leq& 3m + |\{m<j\leq n-2m-1: m_j<m\}| \\
&\leq& 3m + \sum_{j=m+1}^{n-2m-1}\frac{\wb{\theta}_{[j:j+m+1]}-\theta_j}{v(m)} \\
&\leq& 3m + \sum_{j=m+1}^{n-2m-1}\frac{\theta_{j+m+1}-\theta_j}{v(m)} \\
&\leq& 3m + m\frac{\theta_{n-m}-\theta_{1+m}}{v(m)} \\
&\leq& 3m + m\sqrt{m+1}\Big(\wb{\theta}_{[n-m:n-m/2)}-\wb{\theta}_{(1+m/2:1+m]}\Big)/\sigma.
\end{eqnarray*}
Recall that $\wb{l}(m)=\min\Big\{n,3m + m\sqrt{m+1}\Big(\wb{\theta}_{[n-m:n-m/2)}-\wb{\theta}_{(1+m/2:1+m]}\Big)/\sigma\Big\}$ for $m\leq n/3$ and $\wb{l}(m)=n$ for $m>n/3$, and then we have $l(m)\leq\wb{l}(m)$. We also define
\begin{eqnarray*}
\wb{l}_1(m) &=& \min\Big\{n,3m\Big\}, \\
\wb{l}_2(m) &=& \min\Big\{n,m\sqrt{m+1}V/\sigma\Big\},
\end{eqnarray*}
so that $\wb{l}(m)\leq \wb{l}_1(m)+\wb{l}_2(m)$.
This leads to the bound
\begin{eqnarray*}
&& \sigma^2\sum_{\ell\geq 0} \frac{1}{2^{\ell}+1}\Big(l(2^{\ell+1})-l(2^{\ell})\Big) \\
&\leq& 2\sigma^2\sum_{\ell\geq 0}\frac{\wb{l}(2^{\ell+1})-\wb{l}(2^{\ell})}{2^{\ell+1}} \\
&\leq& 2\sigma^2\sum_{\ell\geq 0}\frac{\wb{l}_1(2^{\ell+1})-\wb{l}_1(2^{\ell})}{2^{\ell+1}} + 2\sigma^2\sum_{\ell\geq 0}\frac{\wb{l}_2(2^{\ell+1})-\wb{l}_2(2^{\ell})}{2^{\ell+1}} \\
&\leq& 2\sigma^2\Big( \sum_{\ell\geq 0: \wb{l}_1(2^{\ell})\leq n} \frac{\min\{n,3 \times 2^{\ell+1}\}}{2^{\ell+1}} + \sum_{\ell\geq 0: \wb{l}_2(2^{\ell})\leq n} \frac{\min\Big\{n, 2^{\ell+1}\sqrt{2^{\ell+1}+1}V/\sigma\Big\}}{2^{\ell+1}} \Big)\\
&\leq& C\sigma^2\min\Big\{n, \log (en) + n^{1/3}\Big(\frac{V}{\sigma}\Big)^{2/3}\Big\}.
\end{eqnarray*}
The above argument leads to the first and the third inequalities in Lemma \ref{lem:isotonic-bound}.

To prove the second inequality, recall that
$$\hat{l}(m)=\min\Big\{n,3m + m\sqrt{m+1}\Big(\wb{X}_{[n-m:n-m/2)}-\wb{X}_{(1+m/2:1+m]}\Big)/\sigma\Big\}.$$
Then, we have
$$\left|\sum_{\ell\geq 0:2^{\ell}\leq n/3}\frac{\wh{l}(2^{\ell+1})-\wb{l}(2^{\ell+1})}{2^{\ell+1}}\right|\leq \sum_{\ell\geq 0: 2^{\ell}\leq n/3}\frac{\sqrt{2^{\ell+1}+1}\left|\wb{Z}_{[n-2^{\ell+1}:n-2^{\ell})}-\wb{Z}_{(1+2^{\ell}:1+2^{\ell+1}]}\right|}{\sigma}.$$
By Chebyshev's inequality, we have for any $t>0$,
$$\mathbb{P}\Big(\left|\sum_{\ell\geq 0:2^{\ell}\leq n/3}\frac{\wh{l}(2^{\ell+1})-\wb{l}(2^{\ell+1})}{2^{\ell+1}}\right|> C_3(1+t)\log (en)\Big)\leq C_4\Big(\frac{1}{1+t}\Big)^2,$$
with some constants $C_3,C_4$. The proof is thus complete.
\end{proof}

\begin{proof}[Proof of Lemma \ref{lem:KL-gamma}]
For $\gamma\in(0,1]$, we have
$$D(p_{\gamma,a}||p_{\gamma,b})\leq\int p_{\gamma,a}(x)\Bigl||x-a|^{\gamma}-|x-b|^{\gamma}\Bigr|dx\leq |a-b|^{\gamma}\int p_{\gamma,a}(x)dx=|a-b|^{\gamma},$$
where we have used the inequality $|x+y|^{\gamma}\leq |x|^{\gamma}+|y|^{\gamma}$ for $\gamma\in(0,1]$. For $\gamma\in(1,2]$, we write $\beta=\gamma-1\in(0,1]$. For the function $f(\Delta)=|x+\Delta|^{\gamma}$, its absolute derivative is $|f'(\Delta)|=\gamma|x+\Delta|^{\beta}$. Then, $f(\Delta)=f(0)+f'(\xi)\Delta$, where $\xi$ is a scalar between $0$ and $\Delta$. This leads to the inequality
\begin{equation}
\Bigl||x+\Delta|^{\gamma}-|x|^{\gamma}\Bigr| \leq \gamma|\Delta||x+\xi|^{\beta}\leq \gamma|\Delta|(|x|^{\beta}+|\xi|^{\beta})\leq \gamma|\Delta||x|^{\beta}+\gamma|\Delta|^{\gamma}.\label{eq:taylor}
\end{equation}
Using (\ref{eq:taylor}), with $\Delta=a-b$, we have
\begin{eqnarray*}
D(p_{\gamma,a}||p_{\gamma,b}) &\leq& \int p_{\gamma,a}(x)\Bigl||x-a|^{\gamma}-|x-b|^{\gamma}\Bigr|dx \\
&=& \int p_{\gamma,0}(x)\Bigl||x|^{\gamma}-|x+a-b|^{\gamma}\Bigr|dx \\
&\leq& \gamma\int p_{\gamma,0}(x)|x|^{\beta}dx |a-b| + \gamma\int p_{\gamma,0}(x)dx |a-b|^{\gamma} \\
&\leq& C\Bigl(|a-b|+|a-b|^{\gamma}\Bigr).
\end{eqnarray*}
Hence, the proof is complete.
\end{proof}

\begin{proof}[Proof of Lemma \ref{lem:lp1}]
When $1\leq p<2$, we have 
\begin{align*}
\E\Big\{\max_{k\in [n]} k\Big(\frac{|S_k|}{k}\Big)^p\Big\}\leq \Big(\E\max_{k\in[n]}k^{\frac{2(1-p)}{p}}S_k^2\Big)^{p/2}.
\end{align*}
In addition, denoting $\alpha=2(p-1)/p\in [0,1)$, we have
\begin{align*}
\E\max_{k\in[n]}k^{\frac{2(1-p)}{p}}S_k^2&\leq \E\sum_{\ell=0}^{\lceil\log_2 n\rceil}\max_{2^{\ell-1}<k\leq 2^{\ell}}k^{-\alpha}S_k^2= \sum_{\ell=0}^{\lceil\log_2 n\rceil}\E\max_{2^{\ell-1}<k\leq 2^{\ell}}k^{-\alpha}S_k^2\\
&\leq \sum_{\ell=0}^{\lceil\log_2 n\rceil}2^{-\alpha(\ell-1)}\E\max_{2^{\ell-1}<k\leq 2^{\ell}}S_k^2\leq  4\sum_{\ell=0}^{\lceil\log_2 n\rceil}2^{-\alpha(\ell-1)} \cdot 2^\ell\\
&=8\sum_{\ell=0}^{\lceil\log_2 n\rceil} (2^{1-\alpha})^{\ell-1}\leq C_1 2^{(1-\alpha)\log_2 n}=C_1n^{(2-p)/p},
\end{align*}
and hence $\E\Big\{\max_{k\in [n]} k\Big(\frac{|S_k|}{k}\Big)^p\Big\}$ is of order $n^{1-p/2}$.
\end{proof}


\end{document}